\documentclass[12pt]{amsart}

\usepackage{graphicx} 
\usepackage[utf8]{inputenc}
\usepackage[english]{babel}
\usepackage{amsfonts,amsmath,amssymb,stmaryrd}
\usepackage{faktor}
\usepackage{amsthm}
\usepackage{mathabx}
\usepackage{url}
\usepackage{enumitem}
\usepackage[graphicx]{realboxes}
\usepackage{MnSymbol}
\usepackage{tikz}
\usepackage{tikz-cd}
\usetikzlibrary{calc,shapes.arrows,decorations.pathreplacing}
\usepackage[hyperindex,breaklinks, hidelinks]{hyperref}
\usepackage{cleveref}
\newcommand{\indep}[4]{#1\underset{#2}{\downfree^{#4}}#3}
\newcommand{\notIndep}[4]{#1\underset{#2}{{\not\downfree}^{#4}}#3}
\def\mod{\textup{ mod }}
\theoremstyle{plain}

\newtheorem{theorem}{Theorem}[section]
\newtheorem{proposition}[theorem]{Proposition}
\newtheorem{fact}[theorem]{Fact}
\newtheorem{lemma}[theorem]{Lemma}
\newtheorem{corollary}[theorem]{Corollary}

\theoremstyle{remark}
\newtheorem{remark}[theorem]{Remark}
\newtheorem{example}[theorem]{Example}
\theoremstyle{definition}
\newtheorem{definition}[theorem]{Definition}
\newtheorem{assumptions}[theorem]{Assumptions}
\newcommand{\tp}{\textup{tp}}
\newcommand{\qftp}{\textup{qftp}}
\newcommand{\GL}{\textup{GL}}
\newcommand{\ct}{\textup{ct}}
\newcommand{\acl}{\textup{acl}}
\newcommand{\dcl}{\textup{dcl}}

\newcommand*{\boldone}{\text{\usefont{U}{bbold}{m}{n}1}}

\renewcommand{\subset}{\subseteq}
\renewcommand{\supset}{\supseteq}
\renewcommand{\epsilon}{\varepsilon}
\newcommand{\stab}{\textup{Stab}}
\newcommand{\ram}{\textup{ram}}
\newcommand{\qf}{\textup{qf}}
\newcommand{\LC}{\textup{LC}}
\renewcommand{\div}{\textup{div}}

\newcommand{\DOAG}{\textup{DOAG}}
\newcommand{\ROAG}{\textup{ROAG}}
\newcommand{\val}{\textup{val}}
\newcommand{\Aut}{\textup{Aut}}

\newcommand{\DLO}{\textup{DLO}}

\renewcommand{\d}{\mathbf{d}}
\newcommand{\f}{\mathbf{f}}
\newcommand{\inv}{\mathbf{inv}}
\newcommand{\bo}{\mathbf{bo}}

\newcommand{\Sh}{\mathbf{Sh}}

\newcommand{\cut}{\mathbf{cut}}
\newcommand{\alg}{\mathbf{alg}}

\newcommand{\seq}{\mathbf{seq}}

\theoremstyle{plain}
\newtheorem{introThm}{Theorem}[]

\begin{document}
\author{Akash HOSSAIN}\thanks{The author was partially supported by GeoMod AAPG2019 (ANR-DFG), Geometric and Combinatorial Configurations in Model Theory.}\address{Département de mathématiques d'Orsay, Université Paris-Saclay, France}
\title[Forking and invariant types in ROAG]{Forking and invariant types in regular ordered Abelian groups}
\date{}
\begin{abstract}
We give a characterization of forking in regular ordered Abelian groups. In particular, we prove that $\indep{C}{A}{B}{f}$ if and only if $\indep{c}{A}{B}{f}$ for each singleton $c\in\dcl(AC)$ in these structures.
\end{abstract}
\maketitle


\section*{Introduction}
The non-forking global extensions of some unary type in the theory of divisible ordered Abelian groups (\DOAG) are very easy to describe: they correspond to the invariant cuts, and there are at most two of them for any base parameter set. Now, if we look at a finite tuple $c=(c_1,...,c_n)$ in \DOAG, a necessary condition for $\tp(c/AB)$ to be non-forking over $A$ ($\indep{c_1,...,c_n}{A}{B}{f}$) is to have $\indep{d}{A}{B}{f}$ for every $\mathbb{Q}$-linear combination $d$ of the $c_i$. As the type in \DOAG~ of a tuple is characterized by the cuts of the $\mathbb{Q}$-linear combinations of its components, one can wonder whether this condition is sufficient. We show in this paper (\cref{thmTechniqueEnonce}) that the answer is yes. Moreover, by using quantifier elimination in the Presburger language, we can extend our results to find a characterization of forking in the whole class of regular ordered Abelian groups (\ROAG, see \cref{defROAG} and \cref{equROAG}):
\begin{introThm}\label{mainThm}
Let $M\models$\ROAG, $A$, $B$ parameter subsets of $M$, and $c=c_1...c_n\in M^n$. Let $\kappa=\max(|AB|, 2^{\aleph_0})^+$, and suppose ${M}$ is $\kappa$-saturated and strongly-$\kappa$-homogeneous. If $M$ is discrete, interpret $1$ as its least positive element, else interpret $1$ as $0$ (this is the standard interpretation of $1$ in the Presburger language). Let $C$ be the subgroup of $M$ generated by $c$, and $A', B'$ the relative divisible closures in $M$ of the respective subgroups of $M$ generated by $A\cup\{1\}, AB\cup\{1\}$. Then $\indep{c}{A}{B}{f}$ if and only if $\indep{c}{A}{B}{d}$, if and only if $\tp(c/AB)$ admits a global $\Aut(M^{eq}/A'\cup\acl^{eq}(\emptyset))$-invariant extension in $S(M^{eq})$, if and only if the following conditions hold:
\begin{enumerate}
\item Every closed bounded interval of $B'$ that has a point in $C$ already has a point in the divisible closure of $A'$.
\item For all prime $l$, if $[M:lM]$ is infinite, then we have $(C+l^NM)\cap (B'+l^NM)=A'+l^NM$ for all $N>0$.
\end{enumerate}
Moreover, $\tp(c/AB)$ admits an $\Aut(M/A)$-invariant extension inside $S(M)$ if and only if the above conditions hold, and, additionally, for every prime $l$ for which $[M:lM]$ is finite, we have $C\subset \bigcap\limits_{N>0}A'+ l^NM$.
\end{introThm}

This description of forking is the most simple and natural way to assemble together the (well-known) descriptions of forking in the theory of torsion-free Abelian groups, and in DLO. Our results clearly imply that forking in \ROAG~ is a phenomenon that happens in dimension one: we have $\indep{C}{A}{B}{f}$ if and only if,  for each singleton $c\in\acl(AC)$ (which coincides here with $\dcl(AC)$), we have $\indep{c}{A}{B}{f}$. Note that this particular property of unary forking is well-known in stable $1$-based theories. This raises the question of how this property relates to local modularity in, say, o-minimal theories. It also suggests that one might define over unstable theories a notion (satisfied by \ROAG) which would extend that of $1$-basedness in the stable world. Note that we cannot hope to have this property in infinite (expansions of) fields, as it does not even hold in ACF.

\par Our main results rely heavily on our description of forking in \DOAG, which really is the core of our paper. This theory is an enrichment of DLO and the theory of non-trivial $\mathbb{Q}$-vector spaces, and it is a reduct of RCF, three very common theories in which we know exactly what is forking. However, the current literature lacks of a satisfying description of forking in \DOAG, and this paper fills that gap. Simon gave in (\cite{Simon2011OnDO}, Proposition 2.5) a characterization of forking in dimension one (essentially the one we describe in the first condition of \cref{mainThm}) for dp-minimal ordered structures, which include dp-minimal ordered groups. One can also use Theorem 13.7 from \cite{HHM} to establish our characterization of forking in \DOAG~ in the particular case where the base parameter set is an Archimedean-complete model.
\par Dolich gave in \cite{Dolich} (see \cite{Starchenko2008ANO} for a survey) a characterization of forking in o-minimal expansions of real-closed fields. This very nice characterization naturally links forking with our intuition on o-minimal geometry. The independence notion introduced by Dolich is shown to be stronger or equal to forking in any o-minimal theory, but, for technical reasons, one needs to be in an expansion of RCF to prove equality. In these extensions, Dolich shows that their independence coincides with the abstract model-theoretic notion of “non-1-dividing in semi-simple theories" (see their Section 8). We do not know whether non-forking coincides with the independence of Dolich in DOAG, however we point out in subsection \ref{sectDolich} that forking does coincide with 1-dividing.
\par Our work is very close in spirit to the work of Menunni and Hils-Menunni (\cite{Mennuni2020TheDM}, \cite{Hils2021TheDM}) on the domination monoid. They focus on a different notion, in a larger class of theories. However, while they describe the space of invariant types (over arbitrary small parameter sets) up to equidomination, we essentially study the space of $A$-invariant types (with $A$ fixed) up to $A$-interdefinability. One new difficulty that comes with having to deal with $A$-invariant types instead of arbitrary invariant types is that we have to deal with \textit{Archimedean extensions}, the analogue of immediate extensions for the Archimedean/convex valuation (see the end of page 58 of \cite{Gravett1956ORDEREDAG} for a formal definition).
\par This paper aims to be the first step towards a systematic  characterization of forking in every complexity class of arbitrary ordered Abelian groups. Interesting future work could be to get a nice characterization of forking for dp-minimal ordered groups, as quantifier elimination is not very complicated in this class, and the case in dimension one is already covered by the work of Simon.
\par Let us give a brief outline of the proof of our description of forking in \DOAG. Let $A$, $B$, $c$ be as in \cref{mainThm} so that the conditions of the list hold. First of all, we show that $c$ is $A$-interdefinable with a tuple $d$ which is under “normal form". After that, we split $d$ into fibers of ad-hoc group valuations with respect to which being in normal form ensures that $d$ is \textit{separated} (see \cref{defSep}). For each subtuple $d'$ of this partition of $d$, we build a global $A$-invariant extension of $\tp(d'/AB)$. Finally, we show that we can “glue" all those extensions into a global $A$-invariant extension of the full type $\tp(d/AB)$, and we are done. We will try to be as explicit as possible in our manipulations, and we will show that our ad-hoc valuations interact in a meaningful way with the model-theoretic notions of tensor product and weak orthogonality, which play an important role in \cite{Mennuni2020TheDM}, \cite{Hils2021TheDM}.
\par In the first section of this paper, we write basic definitions, and we give a more detailed and formalized outline of the proof of our result in \DOAG.
\par The way we build our global invariant extension is typically a bottom-up approach. However, the concepts have to be introduced in a particular order to be coherent, and it will be more convenient for us to present the steps of this process in a top-down order:
\par In the second section, we define our ad-hoc valuations, and we show, given global $A$-invariant extensions of the $\tp(d'/AB)$, how to glue them into a global $A$-invariant extension of $\tp(d/AB)$.
\par In the third section, we show how to build a global $A$-invariant extension of each of the $\tp(d'/AB)$.
\par In the fourth section, we show how to build the normal form $d$ from $c$. The transformation of a family into a family under normal form is carried out in a very procedural way which goes through many steps, where we apply succesively $A$-translations and maps from $\GL^n(\mathbb{Q})$ to $c$. Once $d$ is built, our characterization of forking for \DOAG~ will be established.
\par We extend our results in \DOAG~ to \ROAG~ in the fifth section. We use quantifier elimination to describe the types, their global extensions and their conjugates. We end the section by proving our main result, \cref{resultatPpal}, where we give our characterization of forking for \ROAG.
\par Let us also describe the outline of the proof. In \ROAG, we can see a complete type as a reunion of countably many partial types ($\Phi$-types) which are “independent" to each other (see \cref{bijectionQLibre}). For each such partial type, we give algebraic characterizations (essentially the conditions of \cref{mainThm}) of when there exists a partial global extension which does not fork over $A$, and when this extension is $A$-invariant. Our results in \DOAG~ allow us to deal with partial types that involve equality and order. The results about the other partial types may already follow from literature on the model theory of torsion free Abelian groups, we still prove them explicitly for completeness.

\section{Forking in DOAG}\label{sectOAG}

\subsection{Cut-independence}\label{enonce}

\paragraph{\textbf{Notation}} In the theory of some total order $<$, a set is considered an interval (be it closed, open, half-open, possibly with infinite bounds) when its lower bound is smaller or equal to its upper bound. For instance, if we consider, say an interval $]a, b[$, then it is implicit that $a\leqslant b$.

\begin{assumptions}
Let $(M, <, ...)$ be the expansion of an infinite totally ordered first-order structure. Let $C\supset A\subset B$ be parameter subsets of $M$. Suppose $M$ is $|AB|^+$-saturated and strongly-$|AB|^+$-homogeneous.
\end{assumptions}

\par In this paper, we will manipulate various linear orders (not only ordered Abelian groups, but also the chains of their convex subgroups, and other ad-hoc orders). In these linear orders, the cuts will be seen as type-definable sets. The formal definition is as follows:

\begin{definition}\label{cutDef}
Let $c\in M$. The \textit{cut} of $c$ over $A$ will be the $A$-type-definable set $\ct(c/A)$, defined as the intersection of every interval containing $c$ with bounds in $A\cup\{\pm\infty\}$ (a singleton is a closed interval).
\par Let $P$ be some partition of $M$ into convex sets. Then there is a natural linear order on $P$, the only linear order that makes the projection $M\longrightarrow P$ an increasing map.
\par The set of cuts over $A$ is clearly a partition of $M$ into convex sets, therefore the above paragraph applies.
\par Let us also write $x>A$ when $\forall a\in A,\ x>a$. If $\emptyset\neq A'\subset A$, define $\ct_>(A'/A)$ (resp. $\ct_<(A'/A)$) as the cut over $A$ that corresponds to the elements that are $>A'$ (resp. $<A'$), and strictly smaller (resp. larger) than any element of $A$ that is $>A'$ (resp. $<A'$). More generally, if $X$ is some (type/$\vee$)-definable subset of $M$, we write $\ct_<(X/A)=\ct_<(X(A)/A)$, and similarly for $\ct_>(X/A)$.

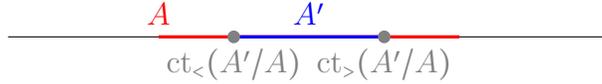
\begin{center}
\begin{tikzpicture}
\draw[black] (-4, 0) -- (4, 0);
\draw[red, very thick] (-2, 0) -- (2, 0);
\draw[blue, very thick] (-1, 0) -- (1, 0);
\draw[red, very thick] (-2, 0) node[anchor=south]{$A$};
\draw[blue, very thick] (0, 0) node[anchor=south]{$A'$};
\filldraw[gray, thick] (-1, 0) circle(2pt) node[anchor=north]{$\ct_<(A'/A)$};
\filldraw[gray, thick] (1, 0) circle(2pt) node[anchor=north]{$\ct_>(A'/A)$};
\end{tikzpicture}
\end{center}
\end{definition}

\begin{definition}\label{indep}
Let us define the following independence notions, which are ternary relations defined on the small subsets of $M$:
\begin{itemize}
\item $\indep{}{}{}{d}$ is the standard non-dividing independence (see \cite{Tent2012ACI}, Definition 7.1.2).
\item $\indep{}{}{}{f}$ is the standard non-forking independence (see \cite{Tent2012ACI}, Definition 7.1.7).
\item Let $c_1...c_n\in M$. Define $\indep{c_1...c_n}{A}{B}{inv}$ when $\tp(c_1...c_n/AB)$ admits a global (in $S(M)$) $\Aut(M/A)$-invariant extension. Define $\indep{C}{A}{B}{inv}$ when $\indep{c}{A}{B}{inv}$ for each finite tuple $c$ of $C$.
\item $\indep{}{}{}{bo}$ (for “Bounded Orbit") is a weaker notion than $\indep{}{}{}{inv}$, where we require the orbit  under $\Aut(M/A)$ of the global extension to have a bounded cardinal instead of being a single point.
\end{itemize}
\end{definition}

\begin{fact}\label{inclusionsIndepTriviales}
The inclusions $\indep{}{}{}{inv}\subset \indep{}{}{}{bo}\subset \indep{}{}{}{f}\subset\indep{}{}{}{d}$ always hold in any first-order structure.
\end{fact}


\begin{remark}\label{DOAGRappelsNIP}
Note that in \DOAG, if $A$ is not included in $\{0\}$, then $\dcl(A)$ is a model. In an o-minimal theory (in fact in any NIP theory), $\indep{}{}{}{\f}$, $\indep{}{}{}{\d}$ and $\indep{}{}{}{\inv}$ all coincide over models, and one can easily check by hand that it is also the case in \DOAG~ for $A\subset\{0\}$. More precisely, the fact that $\indep{}{}{}{\f}=\indep{}{}{}{\d}$ follows from (\cite{KapCherNTP2}, Theorem 1.1), while $\indep{}{}{}{\f}=\indep{}{}{}{\inv}$ over models trivially follows from the fact that non-forking coincides with the independence notion given by Lascar-invariance (see for instance \cite{hrushovskiPillay}, Proposition 2.1). As a result, the abstract equality $\indep{}{}{}{\f}=\indep{}{}{}{\inv}$ is already well-known in DOAG, though this equality does not help in any way to relate forking with more concrete geometric phenomena.
\par We show in \cref{resultatPpal} that $\indep{}{}{}{\f}$, $\indep{}{}{}{\d}$ and $\indep{}{}{}{\bo}$ all coincide in regular groups.
\end{remark}

Now, we are looking at these abstract model-theoretic notions in ordered Abelian groups, which are structures that come with much more concrete geometric notions (the atomic definable sets are the solution sets of $\mathbb{Q}$-linear equations and inequations with parameters). The theory \DOAG~ is even o-minimal. Thus we expect to find a purely geometric description of our abstract independence relations. The most basic example of a dependence behavior is the following: 

\begin{example}
    Suppose $M\models\DLO$, and identify some elementary substructure of $M$ with $\mathbb{Q}$. Let $A=\emptyset$, $B=\{0, 2\}$, $C=\{1\}$. Then $\acl(AC)\cap\acl(AB)=\acl(A)=A$, but $\notIndep{C}{A}{B}{\d}$. Indeed, if we define $I_n=\left[4n, 4n+2\right]$, then the $(I_n)_{n<\omega}$ are pairwise-disjoint and $A$-conjugates, thus they divide over $A$, and $1\in I_0$.
\end{example}

In this example, the definable set that forks is an interval which is disjoint from some conjugates. We can see with the same reasoning that such intervals always divide:

\begin{lemma}\label{divisionTriviale1}
    In the expansion of some total order, let $I$ be an interval with non-empty interior such that its bounds are $A$-conjugates. Then $I$ divides over $A$.
\end{lemma}

\begin{corollary}\label{typePasInvariant}
Suppose that the theory of $M$ is o-minimal, and let $I=[b_1, b_2]$ be an interval which is closed, bounded and disjoint from $\dcl(A)$ (in particular, $b_1$ and $b_2$ are $A$-conjugates). Then $I$ divides over $A$.
\end{corollary}

Note that we only consider bounded intervals: the bounds are not in $\{\pm\infty\}$, for they must be $A$-conjugates. Those considerations allow us to characterize forking in dimension one:

%
%

\begin{proposition}\label{divisionTriviale0}
Suppose the theory of $M$ is o-minimal. Then, for every singleton $c\in M$, we have $\indep{c}{A}{B}{inv}$ if and only if every closed interval with bounds in $\dcl(AB)$ containing $c$ has a point in $\dcl(A)$.
\end{proposition}

\begin{proof}
\par The left-to-right implication is a consequence of \cref{typePasInvariant} and \cref{inclusionsIndepTriviales}. For the other direction, assume the condition on the right holds. Then the elements of $\dcl(AB)$ having the same cut as $c$ must all be either strictly smaller or strictly larger than $c$. Then, either $\ct_<(\ct(c/\dcl(A))/M)$ or $\ct_>(\ct(c/\dcl(A))/M)$ corresponds to a global (and complete by o-minimality) unary type that extends $\tp(c/AB)$. As $\ct(c/\dcl(A))$ is $A$-invariant, so are both those types, and we get the proposition.\qedhere

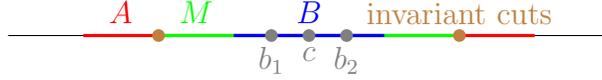
\begin{center}
\begin{tikzpicture}
\draw[black] (-4, 0) -- (4, 0);
\draw[red, very thick] (-3, 0) -- (3, 0);
\draw[green, very thick] (-2, 0) -- (2, 0);
\draw[blue, very thick] (-1, 0) -- (1, 0);
\draw[red, very thick] (-2.5, 0) node[anchor=south]{$A$};
\draw[green, very thick] (-1.5, 0) node[anchor=south]{$M$};
\draw[blue, very thick] (0, 0) node[anchor=south]{$B$};
\filldraw[gray, thick] (0, 0) circle(2pt) node[anchor=north]{$c$};
\filldraw[gray, thick] (-0.5, 0) circle(2pt) node[anchor=north]{$b_1$};
\filldraw[gray, thick] (0.5, 0) circle(2pt) node[anchor=north]{$b_2$};
\filldraw[brown, thick] (-2, 0) circle(2pt) node[anchor=north]{};
\filldraw[brown, thick] (2, 0) circle(2pt) node[anchor=south]{\textup{invariant cuts}};
\end{tikzpicture}
\end{center}
\end{proof}

\begin{definition}\label{LC}
Let $R$ be $\mathbb{Z}$ or $\mathbb{Q}$. Suppose $M$ is an expansion of some $R$-module, and let $f$ be some $n$-ary $\emptyset$-definable function in this structure. We will write $f\in \LC^n(R)$ when there exist $\lambda_1...\lambda_n\in R$ so that $f$ coincides with the function $x_1...x_n\longmapsto\sum\limits_i\lambda_i x_i$. We will write $\LC(R)$ when $n$ is implicit.
\end{definition}

\begin{assumptions}\label{DOAGDeclarationSEV}
Work in \DOAG, in the language of ordered $\mathbb{Q}$-vector spaces. In this language, every substructure is a definably closed $\mathbb{Q}$-vector-subspace. The independence notions we are looking at are insensitive to definable closure (i.e. we have $\indep{C}{A}{B}{}$ if and only if $\indep{\dcl(AC)}{\dcl(A)}{\dcl(AB)}{}$), so we fix $C\geqslant A\leqslant B$ $\mathbb{Q}$-vector subspaces of $M$.
\end{assumptions}

\par It is well-known that \DOAG~ is complete, has quantifier elimination and is o-minimal in this language. Moreover, for $c_1...c_n,d_1...d_n\in M$, we have :
$$(c_i)_i\equiv_A (d_i)_i\Longleftrightarrow \forall f\in \LC^n(\mathbb{Q}),\ \ct\left(f(c)/A\right)=\ct\left(f(d)/A\right)$$
\par By \cref{divisionTriviale0}, one can see that $\indep{C}{A}{B}{inv}$ if and only if for all $c_1...c_n\in C$, there exists $p$ a global extension of $\tp(c_1...c_n/B)$ so that for all $f\in \LC^n(\mathbb{Q})$, for all closed interval $I$ with bounds in $M$ not disjoint from $A$, we have $p(x)\models f(x)\not\in I$.
\par The above description of $\indep{}{}{}{inv}$ is still not geometric, because it is existential in $p\in S(M)$. There is a geometric description of a weaker version of this, where we swap the first two quantifiers:

\begin{definition}\label{defCutIndep}
Define $\indep{C}{A}{B}{cut}$ if the following equivalent conditions hold:

\begin{itemize}
\item $\forall c_1...c_n\in C,\ \forall f\in \LC^n(\mathbb{Q}),\ \indep{\left(f(c)\right)}{A}{B}{inv}$ (the quantifier over $p$ has been swapped with the one over $f$).
\item $\forall c\in C,\ \indep{c}{A}{B}{inv}$.
\item Every closed interval with bounds in $B$ that has a point in $C$ already has a point in $A$.
\end{itemize}
For $c_1...c_n\in M$, we define $\indep{c_1...c_n}{A}{B}{cut}$ when $\indep{(A+\mathbb{Q}c_1+...+\mathbb{Q}c_n)}{A}{B}{cut}$.
\end{definition}

The last item of \cref{defCutIndep} is a purely geometric description which depends only on $A$, $B$, $C$. This description is simple and uses very little algebra: we have to replace $A$ (resp. $B$, $C$) by the $\mathbb{Q}$-vector subspace generated by $A$ (resp. $AB$, $AC$) as done in the notation, then we take every point from $C$, we completely ignore the algebraic relations between those points, and we check for each point a condition that only depends on the linear order.

\begin{lemma}\label{divisionTriviale}
We have $\indep{}{}{}{d}\subset\indep{}{}{}{cut}$ in \DOAG.
\end{lemma}

\begin{proof}
See \cref{typePasInvariant}.
\end{proof}

It actually turns out that this nice, but weak independence notion $\indep{}{}{}{cut}$ is no weaker than $\indep{}{}{}{inv}$ in \DOAG. This is the fundamental result of this paper:

\begin{theorem}\label{thmTechniqueEnonce}
In \DOAG, we have $\indep{}{}{}{inv}=\indep{}{}{}{cut}$.
\end{theorem}

By \cref{inclusionsIndepTriviales} and \cref{divisionTriviale}, we just need to prove that $\indep{}{}{}{cut}\subset\indep{}{}{}{inv}$.

\begin{assumptions}\label{DOAGHypCutIndep}
Recall from \cref{DOAGDeclarationSEV} that $B\geqslant A\leqslant C$ are $\mathbb{Q}$-vector-subspaces of $M$. We assume $\indep{C}{A}{B}{cut}$. Let $c_1...c_n\in C$.
\end{assumptions}

\par Our goal is to build a global $A$-invariant extension of $\tp(c_1...c_n/B)$. This will be achieved in sections \ref{sectGlue}, \ref{secBaseInv}, \ref{sectionFormesNormales}, and we deal with \ROAG~ and examples in the remaining sections.

\begin{remark}\label{DOAGDichotomieDroiteGauche}
If $c$ is a singleton from $C\setminus A$, then we have $\indep{c}{A}{B}{\cut}$ if and only if at least one of the following conditions holds:
\begin{itemize}
\item Every bounded $B$-definable closed interval included in $\ct(c/A)$ (viewed as an $A$-type-definable set) has all its points smaller than $c$.
\item Every bounded $B$-definable closed interval included in $\ct(c/A)$ has all its points larger than $c$.
\end{itemize}
\end{remark}

\begin{definition}
If the first condition of \cref{DOAGDichotomieDroiteGauche} holds, then we  say that $c$ \textit{leans right} with respect to $A$ and $B$. If the second condition holds, then $c$ \textit{leans left}. The two conditions hold at the same time if and only if $\ct(c/\dcl(A))$ has no point in $\dcl(AB)$.

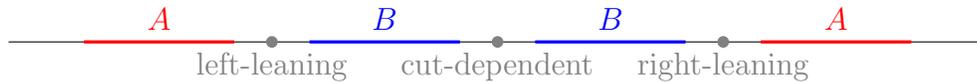
\begin{center}
\begin{tikzpicture}
\draw[black] (-4, 0) -- (9, 0);
\draw[red, very thick] (-3, 0) -- (-1, 0);
\draw[red, very thick] (-2, 0) node[anchor=south]{$A$};
\filldraw[gray] (-0.5, 0) circle (2pt) node[anchor=north]{left-leaning};
\draw[blue, very thick] (0, 0) -- (2, 0);
\draw[blue, very thick] (1, 0) node[anchor=south]{$B$};
\filldraw[gray] (2.5, 0) circle (2pt) node[anchor=north]{cut-dependent};
\draw[blue, very thick] (3, 0) -- (5, 0);
\draw[blue, very thick] (4, 0) node[anchor=south]{$B$};
\filldraw[gray] (5.5, 0) circle (2pt) node[anchor=north]{right-leaning};
\draw[red, very thick] (6, 0) -- (8, 0);
\draw[red, very thick] (7, 0) node[anchor=south]{$A$};
\end{tikzpicture}
\end{center}
\end{definition}

It is easy to see that $c$ leans right (resp. left) with respect to $A$ and $B$ if and only if there exists $A'\subset A$ so that $\ct(c/B)=\ct_>(A'/B)$ (resp. $\ct_<(A'/B)$).
\par This definition also makes sense (and will be used) in more abstract linear orders that do not come from ordered groups.


\subsection{Valuations}
The content of this subsection is well-known in the literature, and most of the proofs are omitted.

\begin{assumptions}
In this subsection, let $G$ be an Abelian group.
\end{assumptions}

\begin{definition}
Let $H$ be a subgroup of $G$, and $\Gamma$ a linearly-ordered set with a least element $-\infty$. A map $\val:\ G\longrightarrow \Gamma$ is an $H$\textit{-valuation} if it satisfies the following axioms for all $x, y\in G$:
\begin{itemize}
\item $\val(x-y)\leqslant \max(\val(x), \val(y))$
\item $\val(x)=-\infty\Longleftrightarrow x\in H$
\end{itemize}
A \textit{valuation} over $G$ is a $\{0\}$-valuation.
\par If $\val'$ is another $H$-valuation of $G$, we say that $\val'$ \textit{refines} $\val$ if there exists an increasing map $f:\ \val'(G)\longrightarrow \val(G)$ so that $\val=\val'\circ f$.
\end{definition}
Equivalently, an $H$-valuation can be seen as a valuation over $\faktor{G}{H}$.

\begin{remark}\label{triangleIsocele}
\par Let us remark for the readers unfamiliar with valuations that, if $\val(x)<\val(y)$, then it follows from the axioms that $\val(x+y)=\val(y)$.
\end{remark}

\par The next lemma gives us a canonical way to build an $H$-valuation from a preorder over $G$ satisfying certain conditions:

\begin{lemma}\label{valuationPreordre}
Let $P$ be a preorder over $G$. Let $\sim$ be the associated equivalence relation over $G$, $\pi$ the quotient map, and $<$ the associated order on $\faktor{G}{\sim}$. Suppose $<$ is linear, and we have $\pi(x)<\pi(y)\Longrightarrow \pi(x+y)=\pi(y)$ for all $x, y\in G$ (the same relation as in \cref{triangleIsocele}). Then the following conditions hold:
\begin{itemize}
\item $\faktor{G}{\sim}$ has a least element, which is $\pi(0)$.
\item The fiber $\pi^{-1}(\pi(0))$ is a subgroup of $G$.
\item $\pi$ is a $\pi^{-1}(\pi(0))$-valuation.
\end{itemize}
\end{lemma}

\par For example, if $P$ is the divisibility relation on an integral local ring $R$ ($P(x, y)$ when $y$ divides $x$), then $P$ is a preorder satisfying the axiom $\pi(x)<\pi(y)\Longrightarrow \pi(x+y)=\pi(y)$ of \cref{valuationPreordre}. In that case, $R/\sim$ is in natural correspondence with the poset of principal ideals of $R$, and $<$ is linear if and only if $R$ is a valuation ring. Then, the group valuation given by the proposition is the natural ring valuation on $R$.

\begin{assumptions}
Suppose now $G$ is an ordered Abelian group.
\end{assumptions}

\begin{proposition}\label{propArch}
Let $x, y\in G$. The following conditions are equivalent:
\begin{enumerate}
\item $x\in\bigcup\limits_{n<\omega}\left[-n|y|, n|y|\right]$.
\item For every convex subgroup $H$ of $G$, if $y\in H$, then $x\in H$.
\item The convex subgroup of $G$ generated by $x$ is included in the one generated by $y$.
\end{enumerate}
\end{proposition}

\begin{definition}\label{defArchVal}
These equivalent conditions define a relation on $G$ which is clearly a preorder that satisfies the necessary conditions of \cref{valuationPreordre}, with $\pi^{-1}(\pi(0))=\{0\}$. This gives a valuation over $G$, which we  call $\Delta$, the \textit{Archimedean valuation}. For $x\in G$, $\Delta^{-1}(\Delta(x))$ is called the \textit{Archimedean class} of $x$. By abuse of notation, we may identify the Archimedean value of an element with its Archimedean class.
\par The idea behind the statement $\Delta(x)<\Delta(y)$ is that $x$ is ``infinitesimal" compared to $y$, or that $y$ is ``infinitely greater" than $x$.
\par Equivalently:
$$
\begin{array}{rcl}
\Delta(x)<\Delta(y) & \Longleftrightarrow & x\in\bigcap\limits_{n>0}\left]-\dfrac{1}{n}|y|, \dfrac{1}{n}|y|\right[\\
& & \\
 & \Longleftrightarrow & |y|\in \bigcap\limits_{n>0}\left]n|x|, +\infty\right[
\end{array}
$$
\end{definition}

Let us draw a picture (not scaled correctly) which intuitively shows how the Archimedean classes look like. Suppose $G$ has exactly four Archimedean classes $\Delta(0)<\delta_1<\delta_2<\delta_3$. Then $G$ looks like this:
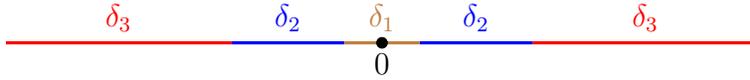
\begin{center}
\begin{tikzpicture}
\draw[red, very thick] (-5, 0) -- (-2, 0);
\draw[red, very thick] (2, 0) -- (5, 0);
\draw[blue, very thick] (-2, 0) -- (-0.5, 0);
\draw[blue, very thick] (0.5, 0) -- (2, 0);
\draw[brown, very thick] (-0.5, 0) -- (0.5, 0);
\draw[red, very thick] (-3.5, 0) node[anchor=south]{$\delta_3$};
\draw[red, very thick] (3.5, 0) node[anchor=south]{$\delta_3$};
\draw[blue, very thick] (-1.25, 0) node[anchor=south]{$\delta_2$};
\draw[blue, very thick] (1.25, 0) node[anchor=south]{$\delta_2$};
\draw[brown, very thick] (0, 0) node[anchor=south]{$\delta_1$};
\filldraw[black] (0, 0) circle(2pt) node[anchor=north]{$0$};
\end{tikzpicture}
\end{center}

There is no standard terminology in the modern literature when it comes to the Archimedean valuation and related notions. Some use the keyword “convex" to refer to all those notions, others use the very obscure names “$i$-extension" or “$i$-completeness". We draw our inspiration from older sources: we choose to use the terminology of \cite{Gravett1956ORDEREDAG}, which, on top of being a good introductory paper, sets up notations which we find more intuitive and easier to work with.

\begin{remark}\label{correspGrpConv_ClassesArch}
A convex subgroup $H\leqslant G$ is equal to the union of the convex subgroups of $G$ generated by each $x\in H$. So $H$ is entirely determined by its direct image $\Delta(H)$, which is an initial segment of $\Delta(G)$. We can note that the set $\mathcal{C}(G)$ of convex subgroups of $G$ is totally ordered by inclusion, and that $H\longmapsto\Delta(H)$ is an order isomorphism between $\mathcal{C}(G)$ and the set of non-empty initial segments of $\Delta(G)$ ordered by inclusion.
\par We can also note that the cosets of some convex subgroup $H\leqslant G$ are all convex, because the translations are strictly increasing. So the quotient of $G$ by some convex subgroup $H$ is naturally endowed with a total order, as defined in \cref{cutDef}.
\end{remark}

\begin{definition}\label{defIndiceClasseArch}
Let $H$ be a subgroup of $G$, and $\delta\in\Delta(G)$. We define $H_{<\delta}=\left\lbrace x\in H|\Delta(x)<\delta\right\rbrace$. We introduce a similar notation for the conditions $\leqslant\delta$, $=\delta$\ldots  More generally, if $P\subset\Delta(G)$, we define $H_{\in P}=\left\lbrace x\in H|\Delta(x)\in P\right\rbrace$.
\end{definition}

\subsection{Outline of the proof}
Although $\indep{}{}{}{\cut}$ and $\indep{}{}{}{\inv}$ are very combinatorial notions, we have to set up a rather technical algebraic machinery to prove that they coincide. Let us explain in this subsection what we do conceptually.

\begin{assumptions}\label{DOAGHypMonstre}
On top of \cref{DOAGHypCutIndep}, we assume that $M$ is $|B|^+$-saturated and strongly $|B|^+$-homogeneous.
\end{assumptions}
\par We wish to show $\indep{c_1\ldots c_n}{A}{B}{\inv}$. For this, we  partition our family $c$ into smaller subfamilies, that we call “blocks", and we  build a global $\Aut(M/A)$-invariant extension of the type of each block. We get a family of global $\Aut(M/A)$-invariant types, that we  call (in \cref{defBE}) a “block extension", and what is left for us is to “glue" those types into a global $\Aut(M/A)$-invariant extension of the full type of $c$.

\begin{definition}\label{defFaiblementOrth}
    Let $p_1\ldots p_n\in S(A)$. We say that the family $(p_i)_i$ is \textit{weakly orthogonal} if $\bigcup\limits_i p_i(x_i)$ is a complete type over $A$ in the variables $x_1\ldots x_n$.
\end{definition}

\par There is a link between orthogonality and convex subgroups that appears in (\cite{Mennuni2020TheDM}, Proposition 4.5). We state a variant:

\begin{lemma}\label{WOrthVSGrpConvexe}
    Let $G$ be some $\Aut(M/A)$-invariant convex subgroup of $M$, and $c, d\in M$. If $c\in (G+ A)\not\ni d$, then $\tp(c/A)$ and $\tp(d/A)$ are weakly orthogonal.
\end{lemma}

\begin{proof}
    Suppose by contradiction that they are not. Then $\tp(d/A)$ is not complete in $S(Ac)$, i.e. $\ct(d/A)$ has a point in $A+\mathbb{Q}\cdot c$, and this point is in $A+G$. By strong homogeneity, there exists an automorhpism $\sigma\in \Aut(M/A)$ such that $\sigma(d)\in A+G$. As $G$ is $\Aut(M/A)$-invariant, we have $d\in A+G$, a contradiction.
\end{proof}

Note that \cref{WOrthVSGrpConvexe} holds in the more general context of a relatively divisible subgroup $G$ in the expansion of an Abelian group which is in some sense \textit{algebraically bounded}, that is when the definable closure coincides with the relative divisible closure of the generated subgroup.

\begin{example}\label{exempleOrthDOAG}
    Suppose $A$ is the lexicographical product:
    $$\mathbb{Q}\times_{lex}\mathbb{Q}\times_{lex}\mathbb{Q}\leqslant\mathbb{R}\times_{lex}\mathbb{R}\times_{lex}\mathbb{R}$$
    let $c_1=(0, 0, \sqrt{2})$, $c_2=(\sqrt{2}, 0, 0)$, $c_3\in M$ such that:
    $$c_3\in\bigcap\limits_{N>0}\left](0, N, 0), \left(\dfrac{1}{N}, 0, 0\right)\right[$$
    let $G$ be the convex subgroup $\bigcap\limits_{N>0}\left]\left(-\dfrac{1}{N}, 0, 0\right), \left(\dfrac{1}{N}, 0, 0\right)\right[$. Then we can show by \cref{WOrthVSGrpConvexe} that $\tp(c_2/A)$ is weakly orthogonal to both $\tp(c_1/A)$ and $\tp(c_3/A)$ using $G$.
    \par Actually, we need not choose $G$ to show that the type of $c_1$ is weakly orthogonal to that of $c_2$, we may choose instead, for instance, the convex subgroup $\bigcap\limits_{N>0}\left]\left(0, -\dfrac{1}{N}, 0\right), \left(0, \dfrac{1}{N}, 0\right)\right[$. In fact, this works for any $\Aut(M/A)$-invariant convex subgroup of $G$ which properly extends the convex subgroup $\bigcap\limits_{N>0}\left]\left(0, 0, -\dfrac{1}{N}\right), \left(0, 0, \dfrac{1}{N}\right)\right[$.
    \par The same cannot be said for $c_3$, $G$ seems to be the only valid witness for orthogonality. Indeed, $c_2$ belongs to $H=\bigcup\limits_{N<\omega} \left[\left(-N, 0, 0\right), \left(N, 0, 0\right)\right]$, which is the least $\Aut(M/A)$-invariant convex subgroup properly extending $G$, whereas $c_3\not\in A+H'$, with $H'=\bigcup\limits_{N<\omega} \left[\left(0, -N, 0\right), \left(0, N, 0\right)\right]$ the largest $\Aut(M/A)$-invariant proper convex subgroup of $G$ (note that $H'$ witnesses orthogonality between $c_3$ and $c_1$).
    \par We conclude that the type of $c_2$ is weakly orthogonal to the two other types, but for very different reasons. As it is harder to prove weak orthogonality between $c_2$ and $c_3$, $c_2$ is, in some sense that will be made formal in the next section, more related to $c_3$ than it is to $c_1$.
\end{example}

Contrary to DLO, there seems to be “several layers" of weak orthogonality in DOAG. We understand them by defining several ways to partition our tuple, that refine each other. The finer the partition, the easier it becomes to find a block extension, the harder it is to “glue" the elements of the block extension.
\par In model theory, the standard way to “glue" global invariant types is via the tensor product:

\begin{definition}
\par Let $p,q$ be $\Aut(M/A)$-invariant global types. We define the \textit{tensor product} of $p$ by $q$ to be the following $\Aut(M/A)$-invariant complete global type:
$$
p(x)\otimes q(y)=\left\lbrace\varphi(x, y, m)|m\in M,\ \exists a\models q_{|Am},\ p(x)\models\varphi(x, a, m)\right\rbrace
$$
\end{definition}

The tensor product is associative, but not commutative in general. However, two types that are weakly orthogonal must necessarily commute.

\begin{fact}
In \DOAG~ (in fact in any distal theory, see \cite{SIMON2013294}, Proposition 2.17), two global $\Aut(M/A)$-invariant types that commute must be weakly orthogonal.
\end{fact}

\begin{lemma}\label{produitOrth}
Let $F_1, F_2$ be closed subspaces of $S(A)$. Suppose $p_1$ and $ p_2$ are weakly orthogonal for all $p_i\in F_i$. Then the set $F_1\times F_2$, seen as a topological subspace of $S(A)$, is exactly the topological direct product $F_1\times_{Top}F_2$, via the homeomorphism:
$$h:\ p_1(x_1)\cup p_2(x_2)\longmapsto (p_1(x_1), p_2(x_2))$$
\end{lemma}

\begin{proof}
The map $h$ is clearly a continuous bijection between compact separated spaces, thus it is a homeomorphism.
\end{proof}

The way we partition our tuple is with fibers of some ad-hoc valuations:

\begin{definition}\label{defBE}
Let $\val$ be a $B$-valuation over $M$. Given our family $c$, the $\val$\textit{-blocks} of $c$ are defined as the maximal subtuples of $c$ of elements of equal value. They form a partition of $c$.
\par By abuse of notation, if $c_i=(c_{ij})_j$ is a $\val$-block of $c$, we define $\val(c_i)=\val(c_{ij})$, which does not depend on the choice of $j$.
\par A \textit{weak $\val$-block extension} of $c$ is a family $(p_i)_i$ of global $\Aut(M/A)$-invariant extensions of the types over $B$ of each $\val$-block of $c$. Such a block extension is \textit{strong} when $\bigcup\limits_i p_i$ is consistent with $\tp(c/B)$.
\par In this paper, whenever we use the terminology of block extensions, it will always be for global $\Aut(M/A)$-invariant extensions of types over $B$, the parameter sets $A$, $B$ will not change.
\end{definition}

In particular, if $\val'$ refines $\val$, then the $\val'$-blocks of some family form a finer partition than its $\val$-blocks.

\begin{definition}\label{defSep}
Let $\val$ be an $A$-valuation over $M$. We say that the family $c$ is $\val$\textit{-separated} when the following conditions hold:
\begin{itemize}
\item $\forall (\lambda_i)_i\in\mathbb{Q},\ \val\left(\sum\limits_i\lambda_i c_i\right)=\max_i(\lambda_i c_i)$.
\item $\forall i,\ c_i\not\in A$.
\end{itemize}
\end{definition}

The notion of separatedness is an algebraic way to state that the $\val$-blocks of a tuple are independent from each other. One can note that $c$ is $\val$-separated if and only if each of its $\val$-blocks is. Moreover, a finite family that is separated with respect to some $A$-valuation must be a lift of some $\mathbb{Q}$-free family of $\faktor{M}{A}$, i.e. a family in $M$ which maps via the canonical surjection to a $\mathbb{Q}$-free family in $\faktor{M}{A}$. In particular, such a family is $\mathbb{Q}$-free.

\begin{remark}\label{indepInterdef}
The family $c$ is clearly $A$-interdefinable with the lift of a $\mathbb{Q}$-free family from $\faktor{C}{A}$. Moreover, if $d$ is a tuple that is $A$-interdefinable with $c$, then one can easily show that we have $\indep{c}{A}{B}{}$ if and only if $\indep{d}{A}{B}{}$ for $\indep{}{}{}{}\in\{\indep{}{}{}{\cut}, \indep{}{}{}{\inv}\}$.
\end{remark}

\begin{assumptions}\label{DOAGHypCLibre}
By \cref{indepInterdef}, on top of \cref{DOAGHypMonstre}, we may assume that $c$ is a lift of a $\mathbb{Q}$-free family from $\faktor{C}{A}$.
\end{assumptions}

\begin{remark}\label{operationsInterdef}
Now, if $d$ is another lift of a $\mathbb{Q}$-free family from $\faktor{C}{A}$, then $c$ and $d$ are $A$-interdefinable if and only if they have the same size $n$, and there exists $f\in \GL_n(\mathbb{Q})$ such that the $i$-th component of $d$ is an $A$-translate of the $i$-th component of $f(c)$ for each $i$.
\end{remark}

We just defined several ways to state that the blocks of some tuple are “independent" from each other and thus “easy to glue". The fact that the types of the blocks are weakly orthogonal is a model-theoretic way to define this idea, whereas the notion of separatedness with respect to some valuation is more algebraic. The core of our proofs relies on building valuations over $M$ for which these notions interact in a meaningful way.

\par More precisely, we  build $D$-valuations $\val^i_D$ on $M$, with $i\in\{1, 2, 3\}$, and with $D\leqslant M$, such that:

\begin{enumerate}[label=(C\arabic*), ref=(C\arabic*)]
\item $\val^3_D$ refines $\val^2_D$, which refines $\val^1_D$.
\item $\val^3_{\{0\}}=\Delta$ and $\val^2_{\{0\}}$ is trivial (it is refined by any valuation).
\item\label{existenceNormalisation} if $\indep{c}{A}{B}{\cut}$, then $c$ is $A$-interdefinable with a family $d$ that is simultaneously separated with respect to $\val^3_A$ and $\val^3_B$.
\item\label{existenceBaseInvariance} If $d$ is a $\val^3_B$-separated family, and $\indep{d}{A}{B}{\cut}$, then $d$ admits a strong $\val^3_B$-block extension.
\item\label{recollementVal3} Given $d$ a $\val^3_B$-separated family, and $(p_i)_i$ a strong $\val^3_B$-block extension of $d$, some tensor product of the $(p_i)_i$ is a global $\Aut(M/A)$-invariant extension of $\tp(d/B)$.
\item\label{recollementVal2} Suppose $d$ is $\val^2_D$-separated, and let $(b_i)_i$ be its $\val^2_D$-blocks. Then $(\tp(b_i/D))_i$ is weakly orthogonal. In particular, any weak $\val^2_B$-block extension is strong.
\item\label{recollementVal1} Given $d$ a $\val^1_B$-separated family, and $(p_i)_i$ a weak (thus strong by the previous point) $\val^1_B$-block extension of $d$, $(p_i)_i$ is weakly orthogonal, and thus they commute.
\end{enumerate}

\noindent Now, if we could build such valuations, then we would have $\indep{c}{A}{B}{\inv}$, which would finish the proof of \cref{thmTechniqueEnonce}. Indeed, if $d$ was given by \ref{existenceNormalisation}, then by \cref{indepInterdef} it would be enough to show that $\indep{d}{A}{B}{\inv}$. Now, again by \cref{indepInterdef}, we would have $\indep{d}{A}{B}{\cut}$, so \ref{existenceBaseInvariance} would give us a strong $\val^3_B$-block extension of $d$, and \ref{recollementVal3} would give us a way to glue it into a witness of $\indep{d}{A}{B}{\inv}$, which would conclude the proof. Moreover, \ref{recollementVal2}, and \ref{recollementVal1} would give us a fine understanding of the global $\Aut(M/A)$-invariant extensions of the type of $d$. In fact, we  establish in \cref{secBaseInv} an exhaustive classification of the Stone space of all these global $\Aut(M/A)$-invariant extensions.
\par Last but far from least, what remains for us to do is to actually define explicitly those valuations $\val^i_D$, and to show that all those nice properties hold.

\subsection{Relation to Dolich-independence}\label{sectDolich}

We make short comments on how our result on DOAG relates to the work of Dolich \cite{Dolich}, their Section 8 in particular. Dolich defines in their paper a geometric independence notion in terms of “halfway-definable cells", which we  call $\indep{}{}{}{\textup{Dolich}}$. It is shown that it coincides with $\indep{}{}{}{\f}$ in any o-minimal expansion of RCF. To motivate their results, Dolich also gives five axioms for independence notions (allegedly from unpublished notes of Shelah), four of which are always satisfied by non-forking in any theory, and the fifth, called “Chain Condition", is a weakening of the independence theorem. They also define an independence notion called “non-1-dividing", with a combinatorial definition which strengthens that of non-forking ; let us write it $\indep{}{}{}{1-\d}$. They claim that, in case some independence relation satisfies the five axioms in a given theory, non-1-dividing is the weakest relation satisfying those axioms, and they show that $\indep{}{}{}{\textup{Dolich}}$ satisfies the five axioms in any o-minimal theory. In particular, we have $\indep{}{}{}{\textup{Dolich}}\subset\indep{}{}{}{1-\d}\subset\indep{}{}{}{\f}$ in any o-minimal theory, and those inclusions are equalities in expansions of RCF.
\par Note that, if we generalize the independence notion $\indep{C}{A}{B}{\cut}$ to any o-minimal theory by the definition: “any closed bounded interval with bounds in $\dcl(AB)$ having a point in $\dcl(AC)$ already has a point in $\dcl(A)$", then it follows from \cref{typePasInvariant} that $\indep{}{}{}{\d}\subset\indep{}{}{}{\cut}$, hence in particular $\indep{}{}{}{\textup{Dolich}}\subset\indep{}{}{}{\cut}$. This inclusion is strict in general, for in RCF, one may find an example where $\indep{C}{A}{B}{\cut}$ holds, while $\indep{C}{A}{B}{\alg}$ fails, which implies $\notIndep{C}{A}{B}{\textup{Dolich}}$. So the main difference between the geometric independence notion introduced by Dolich and ours is that ours is easily shown to be weaker than non-forking in the general case, and the difficulties come when we prove that it is actually as strong as non-forking in DOAG, while $\indep{}{}{}{\textup{Dolich}}$ is clearly stronger than non-forking, and it is difficult to prove the other direction in RCF.
\par While we do not know whether $\indep{}{}{}{\textup{Dolich}}=\indep{}{}{}{\f}$ in DOAG (it would be a strengthening of our result), we  remark that $\indep{}{}{}{1-\d}=\indep{}{}{}{\f}$ in any theory (DOAG in particular) where $\indep{}{}{}{\f}=\indep{}{}{}{\inv}$. Indeed, one may easily show that $\indep{}{}{}{\inv}$ satisfies the Chain Condition, using the fact that it satisfies Extension. If $\indep{}{}{}{\inv}=\indep{}{}{}{\f}$, then it follows that $\indep{}{}{}{\f}$ satisfies the five axioms. By maximality of $\indep{}{}{}{1-\d}$, we have $\indep{}{}{}{\f}\subset\indep{}{}{}{1-\d}$, and the other inclusion always holds.

\section{How to glue types via ad-hoc valuations}\label{sectGlue}
In this section, we  define the valuations $\val^i_A$ ($i\in\{1, 2, 3\}$), and we  show that they satisfy the gluing properties \ref{recollementVal3}, \ref{recollementVal2}, \ref{recollementVal1}. We  also give a classification of the cuts over some parameter set. We work with \cref{DOAGHypMonstre}.

\subsection{Basic definitions and classification of the cuts}

\begin{lemma}\label{actionSurCoupures}
Let $c$, $d\in M$, and $a\in A$, such that $\ct(c/A)=\ct(d/A)$. Then $\ct(c+a/A)=\ct(d+a/A)$.
\end{lemma}

\begin{proof}
If not, then there exists $a'\in A$ which lies strictly between $c+a$ and $d+a$.
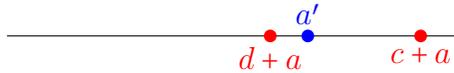
\begin{center}
\begin{tikzpicture}
\draw[black] (0, 0) -- (6, 0);
\filldraw[red, thick] (3.5, 0) circle(2pt) node[anchor=north]{$d+a$};
\filldraw[red, thick] (5.5, 0) circle(2pt) node[anchor=north]{$c+a$};
\filldraw[blue, thick] (4, 0) circle(2pt) node[anchor=south]{$a'$};
\end{tikzpicture}
\end{center}

\noindent then $a'-a\in A$ lies strictly between $c$ and $d$, contradicting the hypothesis.\qedhere
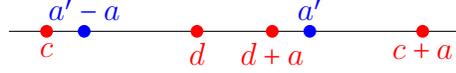
\begin{center}
\begin{tikzpicture}
\draw[black] (0, 0) -- (6, 0);
\filldraw[red, thick] (5.5, 0) circle(2pt) node[anchor=north]{$c+a$};
\filldraw[red, thick] (3.5, 0) circle(2pt) node[anchor=north]{$d+a$};
\filldraw[blue, thick] (4, 0) circle(2pt) node[anchor=south]{$a'$};
\filldraw[blue, thick] (1, 0) circle(2pt) node[anchor=south]{$a'-a$};
\filldraw[red, thick] (2.5, 0) circle(2pt) node[anchor=north]{$d$};
\filldraw[red, thick] (0.5, 0) circle(2pt) node[anchor=north]{$c$};
\end{tikzpicture}
\end{center}

\end{proof}

We hope the figures make the proofs easier to understand. However, we do not want them to be misleading, thus we would like to say that they may not cover all the possible cases. For instance, if $a$ was negative, then the correct picture would be reversed.

\begin{definition}
By \cref{actionSurCoupures} (and by saturation), $A$ acts by translation over the set of all the cuts of $M$ over $A$. For $c\in M$, denote by $\stab(c/A)$ the stabilizer of $\ct(c/A)$.
\end{definition}

Note that such a stabilizer is only a subgroup of $A$, which is clearly convex in $A$, but not in $M$.

\begin{definition}
Let $d\in M$. If $d\not\in A$, then we define: $$G(d/A)=\cap\left\lbrace\left]-|a|,|a|\right[\ :\ a\in A\setminus\stab(d/A)\right\rbrace$$ else we define $G(d/A)=\{0\}$. Either way, we also define: $$H(d/A)=\cup\left\lbrace\left[-|a|,|a|\right]\ :\ a\in \stab(d/A)\right\rbrace$$
\par We view $G(d/A)$ as an $A$-type-definable set, and $H(d/A)$ as an $A$-$\vee$-definable set. They have the same points in $A$, however they do not have the same points in $M$ when $d\not\in A$. They are $A$-(type/$\vee$)-definable convex subgroups.
\end{definition}

\begin{example}
In \cref{exempleOrthDOAG}, $G(c_2/A)=G(c_3/A)=G$, and $H(c_2/A)=H(c_3/A)=H'$. However, $G(c_1/A)$ is the group of elements which are infinitesimal with respect to $A$, and it is distinct from $G$. As for $H(c_1/A)$, it is trivial.
\par Likewise, for arbitrary $A$, if $\Delta(0)<\Delta(d)<\Delta(A\setminus\{0\})$ (i.e. $d$ is infinitesimal with respect to $A$), then $H(d/A)$ is trivial, and $G(d/A)$ is the type-definable group of elements which are infinitesimal with respect to $A$.
\par If $\Delta(d)>\Delta(A)$, then $H(d/A)$ is the convex subgroup generated by $A$, whereas $G(d/A)$ is the whole group.
\end{example}

\begin{remark}\label{DOAGComparaisonGH}
We always have $G(d/A)\geqslant H(d/A)$. By definition, $H(d/A)$ is the convex subgroup generated by $\stab(d/A)$, i.e. the least $A$-$\vee$-definable convex subgroup containing $\stab(d/A)$. In case $d\not\in A$, $G(d/A)$ is the largest $A$-type-definable convex subgroup disjoint from $A\setminus\stab(d/A)$. For all $d$, $d'\in M$, we have in fact :
$$
G(d/A)=G(d'/A)\Longrightarrow H(d/A)=H(d'/A)
$$
the only case where the implication is not an equivalence being when one point is in $A$, and the other is infinitesimal with respect to $A$. We also have:
$$
\stab(d/A)=\stab(d'/A)\Longleftrightarrow H(d/A)=H(d'/A).
$$
\end{remark}

\begin{proposition}\label{GPlusGrandDonneHPlusGrand}
Let $d,d'\in M$. If $G(d/A)<G(d'/A)$, then we have $G(d/A)<H(d'/A)$.
\end{proposition}

\begin{proof}
By definition of $G(d/A)$, there must exist $a\in A\setminus\stab(d/A)$ such that $a\in G(d'/A)$. By \cref{DOAGComparaisonGH}, $a\in H(d'/A)\setminus G(d/A)$.
\end{proof}

Let us now build $\val^1_A$.

\begin{lemma}\label{DOAGVal1BienDef}
Let $d_1,d_2\in M$. Then, with respect to inclusion, we have $G(d_1+d_2/A)\leqslant \max\left(G(d_1/A),G(d_2/A)\right)$.
\end{lemma}

\begin{proof}
Suppose by contradiction that $G(d_1+d_2/A)$ strictly contains $\max\left(G(d_1/A),G(d_2/A)\right)$. Then, there must exist $a\in \stab(d_1+d_2/A)$ such that $a\not\in\stab(d_i/A)$. As $\stab(d_i/A)$ is a convex subgroup of $A$, we also have $\dfrac{a}{2}\not\in \stab(d_i/A)$, so there exists $a_i\in A$ which lies strictly between $d_i$ and $d_i+\dfrac{a}{2}$.
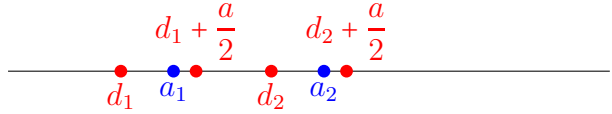
\begin{center}
\begin{tikzpicture}
\draw[black] (-0.5, 0) -- (7.5, 0);
\filldraw[red, thick] (1, 0) circle(2pt) node[anchor=north]{$d_1$};
\filldraw[red, thick] (2, 0) circle(2pt) node[anchor=south]{$d_1+\dfrac{a}{2}$};
\filldraw[red, thick] (3, 0) circle(2pt) node[anchor=north]{$d_2$};
\filldraw[red, thick] (4, 0) circle(2pt) node[anchor=south]{$d_2+\dfrac{a}{2}$};
\filldraw[blue, thick] (1.7, 0) circle(2pt) node[anchor=north]{$a_1$};
\filldraw[blue, thick] (3.7, 0) circle(2pt) node[anchor=north]{$a_2$};
\end{tikzpicture}
\end{center}

\noindent then $a_1+a_2$ lies strictly between $d_1+d_2$ and $d_1+d_2+a$, contradicting the fact that $a\in\stab(d_1+d_2/A)$.\qedhere

\begin{center}
\begin{tikzpicture}
\draw[black] (-0.5, 0) -- (7.5, 0);
\filldraw[red, thick] (1, 0) circle(2pt) node[anchor=north]{$d_1$};
\filldraw[red, thick] (2, 0) circle(2pt) node[anchor=south]{$d_1+\dfrac{a}{2}$};
\filldraw[red, thick] (3, 0) circle(2pt) node[anchor=north]{$d_2$};
\filldraw[red, thick] (4, 0) circle(2pt) node[anchor=south]{$d_2+\dfrac{a}{2}$};
\filldraw[red, thick] (5, 0) circle(2pt) node[anchor=north]{$\ \ d_1+d_2$};
\filldraw[red, thick] (7, 0) circle(2pt) node[anchor=south]{$d_1+d_2+a$};
\filldraw[blue, thick] (1.7, 0) circle(2pt) node[anchor=north]{$a_1$};
\filldraw[blue, thick] (3.7, 0) circle(2pt) node[anchor=north]{$a_2$};
\filldraw[blue, thick] (6.4, 0) circle(2pt) node[anchor=north]{$a_1+a_2$};
\end{tikzpicture}
\end{center}
\end{proof}

\begin{definition}
By \cref{DOAGVal1BienDef}, the map $x\longmapsto G(x/A)$ is an $A$-valuation, so we set $\val^1_A(x)=G(x/A)$.
\end{definition}

Note that, by definition, we always have $G(d/B)\leqslant G(d/A)$ for all $d\in M$. Furthermore:

\begin{proposition}\label{HPlusGrand}
For all $d\in M$, if $\indep{d}{A}{B}{\cut}$, then $H(d/B)\geqslant H(d/A)$.
\end{proposition}

\begin{proof}
Suppose $H(d/B)<H(d/A)$. Then there exists $a\in\stab(d/A)$ such that $a\not\in\stab(d/B)$. Thus we can find $b\in B$ which lies strictly between $d$ and $d+a$.
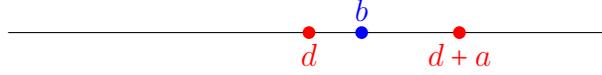
\begin{center}
\begin{tikzpicture}
\draw[black] (-4, 0) -- (4, 0);
\filldraw[red, thick] (2, 0) circle(2pt) node[anchor=north]{$d+a$};
\filldraw[red, thick] (0, 0) circle(2pt) node[anchor=north]{$d$};
\filldraw[blue, thick] (0.7, 0) circle(2pt) node[anchor=south]{$b$};
\end{tikzpicture}
\end{center}

\noindent on one hand, $b-a$  lies strictly between $d$ and $d-a$.
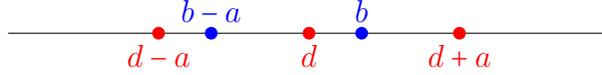
\begin{center}
\begin{tikzpicture}
\draw[black] (-4, 0) -- (4, 0);
\filldraw[red, thick] (-2, 0) circle(2pt) node[anchor=north]{$d-a$};
\filldraw[red, thick] (2, 0) circle(2pt) node[anchor=north]{$d+a$};
\filldraw[red, thick] (0, 0) circle(2pt) node[anchor=north]{$d$};
\filldraw[blue, thick] (-1.3, 0) circle(2pt) node[anchor=south]{$b-a$};
\filldraw[blue, thick] (0.7, 0) circle(2pt) node[anchor=south]{$b$};
\end{tikzpicture}
\end{center}

As $\stab(d/A)$ is a group, we have on the other hand $-a\in\stab(d/A)$, i.e. $\ct(d+a/A)=\ct(d/A)=\ct(d-a/A)$. In particular, no point from $A$ lies in the closed interval with bounds $b$ and $b-a$. On the other hand, $d$ lies strictly between $b$ and $b-a$, which implies $\notIndep{d}{A}{B}{\cut}$.\qedhere
\end{proof}

In the cut-independent setting, one may intuitively see $G(c/A)$ as a distance that approximates the position of $c$ “from the top", while $H(c/A)$ approximates $c$ “from the bottom". Then, when we go to a larger parameter set $B$, we get a finer approximation.

\begin{proposition}\label{DOAGPropRamificateur}
Let $a\in A$, and $d\in M$. Then the following conditions are equivalent:
\begin{itemize}
\item $d-a\in G(d/A)$ and $d\not\in A$.
\item $\Delta(d-a)\not\in \Delta(A)$.
\end{itemize}
\end{proposition}

\begin{proof}
Suppose $d-a\in G(d/A)$. If we had $\Delta(d-a)\in\Delta(A)$, then there would be $a'\in A$ such that $\Delta(d-a)=\Delta(a')$. This would imply $a'\in (G(d/A)\cap A)=\stab(d/A)$. Since, for some $n$, $d-n|a'|\leqslant a\leqslant d+n|a'|$, as $na'\in\stab(d/A)$, it follows that $\ct(d/A)=\ct(a/A)$, thus $d=a$, proving the first direction.
\par Conversely, suppose $\Delta(d-a)\not\in\Delta(A)$. By definition of $G(d/A)$, in order to show that $a\in (d\mod G(d/A))$, it suffices to show that $a\in\left]d-|a'|,d+|a'|\right[$ (i.e. $|d-a|\leqslant |a'|$) for every $a'\in A\setminus \stab(d/A)$. Suppose by contradiction this fails for some $a'$. Then $\Delta(d-a)\geqslant \Delta(a')$, and this inequality is strict by the hypothesis. As $a'\not\in\stab(d/A)$, let $a''\in A$ which lies strictly between $d$ and $d+a'$.
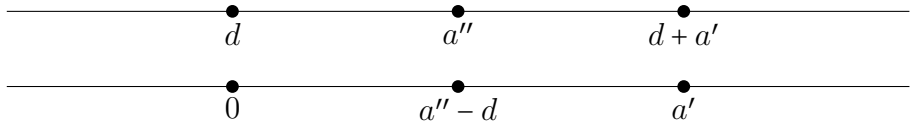
\begin{center}
\begin{tikzpicture}
\draw[black] (0, 0) -- (12, 0);
\draw[black] (0, -1) -- (12, -1);
\filldraw[black, thick] (3, 0) circle(2pt) node[anchor=north]{$d$};
\filldraw[black, thick] (6, 0) circle(2pt) node[anchor=north]{$a''$};
\filldraw[black, thick] (9, 0) circle(2pt) node[anchor=north]{$d+a'$};
\filldraw[black, thick] (3, -1) circle(2pt) node[anchor=north]{$0$};
\filldraw[black, thick] (6, -1) circle(2pt) node[anchor=north]{$a''-d$};
\filldraw[black, thick] (9, -1) circle(2pt) node[anchor=north]{$a'$};
\end{tikzpicture}
\end{center}

\noindent Then $|d-a''|\leqslant |a'|$, thus $\Delta(d-a'')<\Delta(d-a)$. We then apply \cref{triangleIsocele} to get $\Delta(d-a)=\Delta(d-a''-d+a)=\Delta(a-a'')\in\Delta(A)$, a contradiction.
\end{proof}

\begin{definition}
If $a$ satisfies the conditions of \cref{DOAGPropRamificateur}, then we say that $a$ is a \textit{ramifier of $d$ over $A$}. We  write $\ram(d/A)$ for the set of those ramifiers, and we  say that $d$ is \textit{ramified over $A$} if this set is non-empty. We  say that $d$ is \textit{Archimedean over $A$} whenever it is not ramified over $A$.
\end{definition}

In \cref{exempleOrthDOAG}, $c_1$ and $c_2$ are Archimedean over $A$, $c_3$ is ramified over $A$, and $0$ is a ramifier.

\begin{remark}\label{conditionsRam}
If $d\in M$, then $d$ is ramified over $A$ if and only if it satisfies the following equivalent conditions:
\begin{enumerate}
\item The coset $d\mod G(d/A)$ has a point in $A$, and $d\not\in A$.
\item $\Delta(A+\mathbb{Q}d)\supsetneq\Delta(A)$.
\end{enumerate}
Moreover, as $A\cap G(d/A)=\stab(d/A)$, $\ram(d/A)$ is a coset from the quotient $\faktor{A}{\stab(d/A)}$, thus all its elements lie in the same coset $\mod H(d/A)$. It turns out that $d$ does not belong to that coset:
\end{remark}

\begin{lemma}\label{ramificateurLoin}
Let $d\in M\setminus A$ and $a\in A$. Then $d-a\not\in H(d/A)$.
\end{lemma}

\begin{proof}
Suppose by contradiction that there exists $a'\in \stab(d/A)$ such that $|d-a|\leqslant |a'|$. Then $a\in \ram(d/A)$, thus $\Delta(d-a)\not\in \Delta(A)$, and in particular $\Delta(d-a)<\Delta(a')$. By \cref{triangleIsocele}, $\Delta(d-(a+a'))=\Delta(a')\in\Delta(A)$, thus $a+a'\not\in \ram(d/A)$. This contradicts the fact that $a\in \ram(d/A)$, and the fact that $a'\in\stab(d/A)$.
\end{proof}

\begin{corollary}\label{tmpDefPetitDelta}
Let $d\in M$ be ramified over $A$, as well as two elements $a$ and $a'$ in $\ram(d/A)$. Then $\Delta(d-a)=\Delta(d-a')$.
\end{corollary}

\begin{proof}
We have $d-a\not\in H(d/A)\ni a-a'$, thus $\Delta(d-a)>\Delta(a-a')$, and by \cref{triangleIsocele} we have $\Delta(d-a)=\Delta(d-a+a-a')=\Delta(d-a')$.
\end{proof}

\begin{definition}\label{defPetitDelta}
Let $d\in M$ be ramified over $A$. We define $\delta(d/A)$ as $\Delta(d-a)$ for some $a\in \ram(d/A)$. This definition does not depend on the choice of the ramifier $a$ by \cref{tmpDefPetitDelta}.
\end{definition}

Note that $\delta(d/A)$ is the unique element of $\Delta(A+\mathbb{Q}\cdot d)\setminus\Delta(A)$.

\begin{remark}\label{classeUnique}
Let $d$ be ramified over $A$, and $a\in \ram(d/A)$. Then, as $G(d/A)$ and $H(d/A)$ have the same points in $A$, one can easily compute $\ct(\delta/\Delta(A))$. A convex subgroup is induced by its Archimedean classes, so by definition of the $\vee$-definable convex subgroup $H=H(d/A)$, for any model $N$ containing $A$, $\Delta(H(N))$ is the least initial segment of $\Delta(N)$ containing $\Delta(\stab(d/A))$. Now, as $\delta\in\Delta(G(d/A))\setminus\Delta(H)$, we clearly have $\ct(\delta/\Delta(A))=\ct_>(\Delta(\stab(d/A))/\Delta(A))$.
\end{remark}
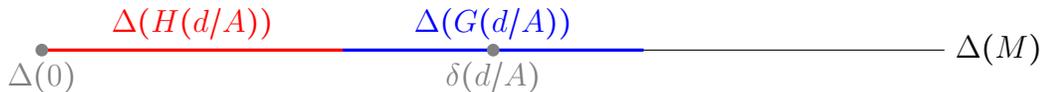
\begin{center}
\begin{tikzpicture}
\draw[black] (0, 0) -- (12, 0);
\draw[red, very thick] (0, 0) -- (4, 0);
\draw[blue, very thick] (4, 0) -- (8, 0);
\filldraw[gray, thick] (0, 0) circle(2pt) node[anchor=north]{$\Delta(0)$};
\filldraw[gray, thick] (6, 0) circle(2pt) node[anchor=north]{$\delta(d/A)$};
\filldraw[black] (12, 0) node[anchor=west]{$\Delta(M)$};
\filldraw[red] (2, 0) node[anchor=south]{$\Delta(H(d/A))$};
\filldraw[blue] (6, 0) node[anchor=south]{$\Delta(G(d/A))$};
\end{tikzpicture}
\end{center}

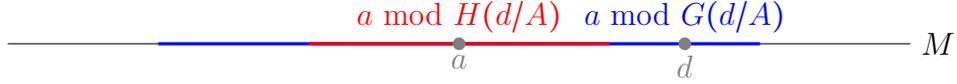
\begin{center}
\begin{tikzpicture}
\draw[black] (-6, 0) -- (6, 0);
\draw[blue, very thick] (-4, 0) -- (4, 0);
\draw[red, very thick] (-2, 0) -- (2, 0);
\filldraw[black] (6, 0) node[anchor=west]{$M$};
\filldraw[gray, thick] (0, 0) circle(2pt) node[anchor=north]{$a$};
\filldraw[gray, thick] (3, 0) circle(2pt) node[anchor=north]{$d$};
\filldraw[red] (0, 0) node[anchor=south]{$a\mod H(d/A)$};
\filldraw[blue] (3, 0) node[anchor=south]{$a\mod G(d/A)$};
\end{tikzpicture}
\end{center}

Many notions that we manipulate behave very differently with ramified and Archimedean points, it will often lead to case disjunctions. For instance, the following statements show us that cuts do not look the same for Archimedean and ramified points:

\begin{lemma}\label{coupureIncluseDansG}
Let $d\in M$. Then $\ct(d/A)\subset (d\mod G(d/A))$ (as type-definable sets).
\end{lemma}

\begin{proof}
By contraposition, let $d'\in M$ be such that $d-d'\not\in G(d/A)$. Then, by definition, there exists $a\in A\setminus\stab(d/A)$ such that $|d-d'|\geqslant|a|$. Let $a'\in A$ which lies strictly between $d$ and $d+a$. As either $d+a$ or $d-a$ lies strictly between $d$ and $d'$, either $a'$ or $a'-a$ lies strictly between $d$ and $d'$, thus $\ct(d/A)\neq\ct(d'/A)$.
\end{proof}

\begin{proposition}\label{coupureArch}
Let $d\in M$ be Archimedean over $A$. Then the $A$-type-definable set $\ct(d/A)$ coincides with $d\mod G(d/A)$.
\end{proposition}

\begin{proof}
Suppose $d\not\in A$, else the proposition is trivial.
\par We have the inclusion $d\mod G(d/A)\subset\ct(d/A)$, as $d\mod G(d/A)$ is a convex set containing $d$ and disjoint from $A$. The other direction follows from \cref{coupureIncluseDansG}.
\end{proof}

\begin{proposition}\label{coupureRam}
Let $d\in M$ be ramified over $A$. Then $\ct(d/A)$ can be written as an $A$-translate of one of the two connected components of $G(d/A)\setminus H(d/A)$.
\end{proposition}
\noindent These translates are exactly the two blue segments in the last figure of \cref{classeUnique}. The connected components of some set will always refer to its maximal convex subsets.
\begin{proof}
Let $a\in \ram(d/A)$. By \cref{coupureIncluseDansG}, we have:
$$\ct(d/A)\subset d\mod G(d/A)=a\mod G(d/A)$$
By \cref{ramificateurLoin}, $d\not\in (a\mod H(d/A))$, which is the convex closure of the set $\ram(d/A)$, which is a subset of $A$, thus $a\mod H(d/A)$ must be disjoint from $\ct(d/A)$. It follows that $\ct(d/A)$ is included in the connected component of $(a+G(d/A))\setminus (a+H(d/A))$ containing $d$.
\par Let $\mathcal{C}$ be that connected component, and $d'\in \mathcal{C}$. There does not exist $a'\in A$ lying between $d$ and $d'$, for such $a'$ would be in $(a+ G(d/A))\cap A=\ram(d/A)\subset a+H(d/A)$, and $a'$ would also belong to $\mathcal{C}$ as $\mathcal{C}$ is convex, contradicting the fact that $\mathcal{C}$ is disjoint from $a+H(d/A)$.
\end{proof}

\begin{corollary}\label{critereEquRam}
With the hypothesis from \cref{coupureRam}, if $a$ is in $\ram(d/A)$, and $d'$ is another point from $M$, then $d'\equiv_A d$ if and only if the following conditions hold:
\begin{itemize}
\item $\ct(\Delta(d'-a)/\Delta(A))=\ct(\Delta(d-a)/\Delta(A))$ (hence $\Delta(d'-a)$ is not in $\Delta(A)$).
\item $d'<a\Longleftrightarrow d<a$.
\end{itemize}
\end{corollary}

\begin{proof}
The first condition is equivalent to $d'-a\in G(d/A)\setminus H(d/A)$. The second condition ensures that $d'$ lies in the correct connected component.
\end{proof}

The two statements \cref{coupureArch} and \cref{coupureRam} yield a simple classification of the cuts over $A$. In order to consider non-forking extensions of types, we now have to deal with ways to refine those cuts to cuts over a larger parameter set $B$. As we classified the cuts over $A$ by looking at $\Aut(M/A)$-invariant convex subgroups, we have to compute what are the $\Aut(M/B)$-invariant convex subgroup corresponding to their refinements.

\begin{definition}
If $G$ is an $A$-type-definable convex subgroup, then we define the following $A$-$\vee$-definable convex subgroup:
$$\underline{G}_A=\bigcup\limits_{a\in A\cap G}\left[-|a|,|a|\right]$$
\par If $H$ is an $A$-$\vee$-definable convex subgroup, then we define the following $A$-type-definable convex subgroup:
$$\overline{H}^A=\bigcap\limits_{a\in A\setminus H}\left]-|a|,|a|\right[$$
\end{definition}

\begin{remark}\label{inegalitesRafinementsGrp}
Let $d\in M$. We have:
$$ H(d/A)\leqslant{\overline{H(d/A)}^M}\leqslant{\overline{H(d/A)}^B}\leqslant{\underline{G(d/A)}_B}\leqslant {\underline{G(d/A)}_M}\leqslant G(d/A)$$
and :
$$
\underline{G(d/A)}_A=H(d/A),\ {\overline{H(d/A)}^A}=G(d/A)
$$
moreover, if $\indep{d}{A}{B}{\cut}$, then by \cref{HPlusGrand} we recall:
$$H(d/A)\leqslant H(d/B)\leqslant G(d/B)\leqslant G(d/A).$$
\end{remark}

\subsection{Weak orthogonality}
Now we start to deal with global invariant extensions, in order to prove \ref{recollementVal1}.

\begin{lemma}\label{coupInvArch}
Let $d\in M\setminus A$ be Archimedean over $A$, with $\indep{d}{A}{B}{\cut}$. Let $p$ be a global $\Aut(M/A)$-invariant extension of $\tp(d/B)$. Then $G(d/B)=G(d/A)$ (as $B$-type-definable sets), and:
$$p(x)\models x-d\in G(d/A)\setminus{\underline{G(d/A)}_M}$$
\end{lemma}

\noindent Note that the statement holds even if the element $d$ is ramified over $B$.

\begin{proof}
Suppose by contradiction $G(d/B)<G(d/A)$. Then there exists $b$ in $B\setminus\stab(d/B)$ such that $b\in G(d/A)$. Let $b_1\in B$ which lies strictly between $d$ and $d+b$. Like wise, as $-b$ is also in $G(d/A)\setminus\stab(d/B)$, we find $b_2\in B$ which lies strictly between $d$ and $d-b$. Then the closed interval with bounds $b_1, b_2$ contains $d$ and is strictly included in $d\mod G(d/A)=\ct(d/A)$, which contradicts cut-independence.
\par By \cref{coupureArch}, as $p$ extends $\tp(d/A)$, we necessarily have:
$$p(x)\models x-d\in G(d/A)$$
\par Let $c\models p$ in some elementary extension. As $p$ is $\Aut(M/A)$-invariant, we have $\indep{c}{A}{M}{\cut}$, thus $G(c/M)=G(d/A)$ by the above paragraph ($c$ is of course Archimedean over $A$). As $c-d\in G(c/M)$, and $d\in M\not\ni c$, we have $d\in\ram(c/M)$, thus $c-d\not\in H(c/M)={\underline{G(c/M)}_M}={\underline{G(d/A)}_M}$. We conclude that $p(x)\models x-d\not\in {\underline{G(d/A)}_M}$.
\end{proof}

\begin{lemma}\label{coupInvRam}
Let $d\in M$ be ramified over $A$, with $\indep{d}{A}{B}{\cut}$, and let $a\in\ram(d/A)$. Then $\Delta(d-a)\not\in\Delta(B)$.
\par Moreover, if $p$ is a global $\Aut(M/A)$-invariant extension of $\tp(d/B)$, then, either ($p(x)\models x-a\in G(d/A)\setminus{\underline{G(d/A)}_M}$ and $G(d/B)=G(d/A)$) or ($p(x)\models x-a\in {\overline{H(d/A)}^M}\setminus H(d/A)$ and $H(d/B)=H(d/A)$).
\end{lemma}

\begin{proof}
Suppose we have $\Delta(d-a)=\Delta(b)$ for some $b\in B$. Let $n>0$ such that $\dfrac{1}{n}|b|\leqslant |d-a|\leqslant n|b|$. Then we have $d\in \left[a-n|b|,a-\dfrac{1}{n}|b|\right]\cup\left[a+\dfrac{1}{n}b,a+n|b|\right]$. As $b\in G(d/A)\setminus H(d/A)$, none of those two intervals has a point in $A$, which contradicts $\indep{d}{A}{B}{\cut}$.
\par As $d$ is ramified over $A$, and $p$ extends $\tp(d/A)$, it follows from \cref{coupureRam} that $p(x)\models x-a\in G(d/A)\setminus H(d/A)$.
\par Suppose by contradiction $p(x)\models x-a\in {\underline{G(d/A)}_M}\setminus{\overline{H(d/A)}^M}$. Then there exists $m_1,m_2\in M$ such that $m_1\not\in H(d/A)$, $m_2\in G(d/A)$, and:
$$p(x)\models |m_1|\leqslant |x-a|\leqslant |m_2|$$
i.e. $p(x)\models x\in \left[a-|m_2|,a-|m_1|\right]\cup\left[a+|m_1|,a+|m_2|\right]$. As $m_1\not\in H(d/A)$ and $m_2\in G(d/A)$, each of those two intervals is included in each connected component of $a+(G(d/A)\setminus H(d/A))$. As a result, none of those intervals belongs to $p$ by \cref{typePasInvariant}, a contradiction.
\par By \cref{coupInvArch}, and by \cref{inegalitesRafinementsGrp}, if we had $p(x)\models x-a\in {\underline{G(d/A)}_B}$ (resp. $p(x)\models x-a\not\in {\overline{H(d/A)}^B}$), then we would have:
$$p(x)\models x-a\in {\overline{H(d/A)}^M}\setminus H(d/A)$$
and respectively: $p(x)\models x-a\in G(d/A)\setminus{\underline{G(d/A)}_M}$. To conclude, it now suffices to show the two following properties:
\begin{enumerate}
\item If $G(d/B)<G(d/A)$, then $\tp(d/B)\models x-a\in {\underline{G(d/A)}_B}$.
\item If $H(d/B)>H(d/A)$, then $\tp(d/B)\models x-a\not\in {\overline{H(d/A)}^B}$.
\end{enumerate}
Let us proceed:
\begin{enumerate}
\item By hypothesis, let $b\in B\setminus\stab(d/B)$ such that $b\in G(d/A)$. Let $b'\in B$ which lies strictly between $d$ and $d+b$.

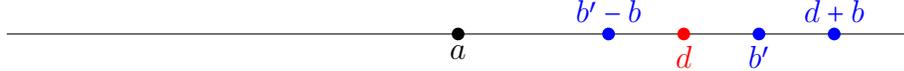
\begin{center}
\begin{tikzpicture}
\draw[black] (-6, 0) -- (6, 0);
\filldraw[black, thick] (0, 0) circle(2pt) node[anchor=north]{$a$};
\filldraw[blue, thick] (2, 0) circle(2pt) node[anchor=south]{$b'-b$};
\filldraw[red, thick] (3, 0) circle(2pt) node[anchor=north]{$d$};
\filldraw[blue, thick] (4, 0) circle(2pt) node[anchor=north]{$b'$};
\filldraw[blue, thick] (5, 0) circle(2pt) node[anchor=south]{$d+b$};
\end{tikzpicture}
\end{center}

\noindent then the interval $I$ with bounds $b'-a$, $b'-b-a$ is included in $G(d/A)$, and $\tp(d/B)\models x-a\in I$.

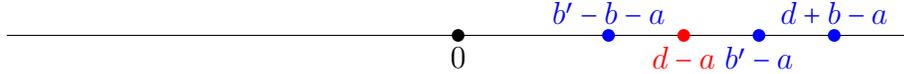
\begin{center}
\begin{tikzpicture}
\draw[black] (-6, 0) -- (6, 0);
\filldraw[black, thick] (0, 0) circle(2pt) node[anchor=north]{$0$};
\filldraw[blue, thick] (2, 0) circle(2pt) node[anchor=south]{$b'-b-a$};
\filldraw[red, thick] (3, 0) circle(2pt) node[anchor=north]{$d-a$};
\filldraw[blue, thick] (4, 0) circle(2pt) node[anchor=north]{$b'-a$};
\filldraw[blue, thick] (5, 0) circle(2pt) node[anchor=south]{$d+b-a$};
\end{tikzpicture}
\end{center}

\item By hypothesis, let $b\in \stab(d/B)$ such that $b\not\in H(d/A)$. Then we have $\ct(d-b/B)=\ct(d/B)=\ct(d+b/B)$, which implies $a\not\in \left[d-|b|,d+|b|\right]$. As a result, $\tp(d/B)\models x-a\not\in\left[-|b|, |b|\right]$.\qedhere
\end{enumerate}
\end{proof}

The two statements \cref{coupInvArch}, \cref{coupInvRam} imply:

\begin{corollary}\label{coupuresInvPossibles}
Let $d\in M$ such that $\indep{d}{A}{B}{\cut}$. Let $p\in S(M)$ be a global $\Aut(M/A)$-invariant extension of $\tp(d/B)$. Then there exists $f$ an $A$-definable map such that one of the three following conditions hold:
\begin{itemize}
\item $d\in A $, $f=id$.
\item The cut over $M$ induced by $p$ is one of the two connected components of the translation by $f(d)$ of $G(d/A)\setminus {\underline{G(d/A)}_M}$. In that case, $G(d/B)=G(d/A)$.
\item The cut over $M$ induced by $p$ is one of the two connected components of the translation by $f(d)$ of ${\overline{H(d/A)}^M}\setminus H(d/A)$. In that case, $H(d/B)=H(d/A)$.
\end{itemize}
\end{corollary}

\begin{proof}
    If $d\in A$, then the statement is trivial. If $d\not\in A$ and $d$ is Archimedean over $A$, then the second condition holds with $f=id$, as stated in \cref{coupInvArch}. If $d$ is ramified over $A$, then either the second or the third condition holds with $f:x\longmapsto x-a$, by \cref{coupInvRam}.
\end{proof}

\begin{remark}
If $d\not\in A$, while only one of the conditions from \cref{coupuresInvPossibles} holds, we might still have $G(d/B)=G(d/A)$ and $H(d/B)=H(d/A)$ at the same time. For instance, this is the case whenever $\Delta(d)<\Delta(B\setminus\{0\})$, and $B$ has no element that is infinitesimal compared to $A$.
\par In fact, one can see that $p$ is the type of an $M$-ramified point which adds an Archimedean class $\delta$ which is cut-independent over $\Delta(A)$: the two conditions of the statement expresses that $\delta$ either leans left or right with respect to $\Delta(A)$ and $\Delta(M)$.
\end{remark}

\begin{remark}\label{rqueCpInvP}
For all $d\in M$, recall from \cref{inegalitesRafinementsGrp} that $G(d/B)={\overline{H(d/B)}^B}$. As a result, if $H(d/B)=H(d/A)$, then $G(d/B)={\overline{H(d/A)}^B}$. In particular we see that, in either case of \cref{coupuresInvPossibles}, $p(x)$ implies $x-f(d)\in G(d/B)$. Likewise:
\end{remark}

\begin{lemma}\label{rqueCpInvP2}
    Let $d\in M$ such that $\indep{d}{A}{B}{\cut}$, and $d\not\in A$. Let $p$ be a global $\Aut(M/A)$-invariant extension of $\tp(d/B)$. Then $p(x)$ implies $x\not\in (M+H(d/B))$.
\end{lemma}

\begin{proof}
    As $d\not\in A$, the first condition of \cref{coupuresInvPossibles} fails. Let $f$ be the $A$-definable map witnessing \cref{coupuresInvPossibles}.
    \par As $G(d/A)\geqslant G(d/B)$, we have: $$\underline{G(d/A)}_M\geqslant\underline{G(d/B)}_M\geqslant H(d/B)$$
    \par If the second condition holds, let $G=G(d/B), H=\underline{G(d/B)}_M$, else let $G=\overline{H(d/B)}^M, H=H(d/B)$. Either way, $G$ and $H$ have the same points in $M$, $p(x)\models x-f(d)\in G\setminus H$, and $H\geqslant H(d/B)$. Suppose by contradiction that there is $m\in M$ such that $p(x)\models x-m\in H(d/B)$. Then, as $p(x)\models x-f(d)\in G$, we have $m-f(d)\in G$, yet $m-f(d)\in M$, therefore $m-f(d)\in H$. This is a contradiction, as we cannot have at the same time $p(x)\models x-f(d)\not\in H$, $p(x)\models x-m\in H$ and $f(d)-m\in H$.
\end{proof}

We established that realizations of unary global invariant types belong to some cosets of convex subgroups, and not others. Just as in \cref{WOrthVSGrpConvexe}, we may relate those statements to orthogonality, in order to prove \ref{recollementVal1}.

\begin{lemma}\label{inductionWOrth}
    Let $d_1,\ldots ,d_n$ be a family of tuples of $M$ such that, for all $i$, $\tp(d_i/A)$ and $\tp(d_{<i}/A)$ are weakly orthogonal. Then the family of types $(\tp(d_1/A),\ldots ,\tp(d_n/A))$ is weakly orthogonal with respect to \cref{defFaiblementOrth}.
\end{lemma}

\begin{proof}
    We show by induction on $i$ that $(\tp(d_1/A),\ldots ,\tp(d_i/A))$ is weakly orthogonal. The statement is trivial if $i=1$. By induction hypothesis, we have $\bigcup\limits_{j<i}\tp(d_j/A)\models\tp(d_{<i}/A)$. By hypothesis, $\tp(d_i/A)\cup\tp(d_{<i}/A)\models\tp(d_{\leqslant i}/A)$, and we conclude by induction.
\end{proof}

\begin{lemma}\label{tupleWOrth}
    Let $d_1, d_2$ be two tuples from $M$. Assume for all $f$ in $\LC(\mathbb{Q})$ that $\tp(f(d_1)/A)$ is weakly orthogonal to $\tp(d_2/A)$. Then $\tp(d_1/A)$ is weakly orthogonal to $\tp(d_2/A)$.
\end{lemma}

\begin{proof}
    Let $d\equiv_A d_1$. It suffices to show that $d\equiv_{Ad_2}d_1$. By quantifier elimination, this holds if and only if, for every $f\in\LC(\mathbb{Q})$, we have $f(d)\equiv_{Ad_2}f(d_1)$, which is exactly the hypothesis of the statement.
\end{proof}

\begin{lemma}\label{tupleWOrth2}
    Let $d_1\in M$ be a singleton, and let $d_2\in M$ be a tuple. Let $V$ be the $\mathbb{Q}$-vector space $\dcl(Ad_2)$, and let $G$ be some $\Aut(M/A)$-invariant convex subgroup of $M$. Suppose we have $V\subset (A+G)\not\ni d_1$. Then $\tp(d_1/A)$ and $\tp(d_2/A)$ are weakly orthogonal.
\end{lemma}

\begin{proof}
    The proof is the same as that of \cref{WOrthVSGrpConvexe}.
\end{proof}

Let us now prove \ref{recollementVal1}:

\begin{proposition}\label{WOrthVal1}
Let $c=(c_{ij})_{ij}$ be a finite tuple from $M$ that is $\val^1_B$-separated, such that its $\val^1_B$-blocks are the $(c_i)_i=((c_{ij})_j)_i$. Suppose we have $(p_i(x_i))_i$ a weak $\val^1_B$-block extension of the $c_{ij}$. Then $(p_i)_i$ is weakly orthogonal.
\end{proposition}

\noindent Recall from \cref{defBE} that by definition, types in a block extension are $\Aut(M/A)$-invariant. In particular, we have $\indep{c_i}{A}{B}{\cut}$ for each $i$.

\begin{proof}
Enumerate the $c_i$ in increasing order of their values. Let $N$ be some $|M|^+$-saturated, strongly $|M|^+$-homogeneous elementary extension of $M$, and let $d_i\in N$ be a realization of $p_i$. By \cref{inductionWOrth}, and \cref{tupleWOrth}, it suffices to show that for all $i$, for all $f\in \LC(\mathbb{Q})$, $\tp(f(c_i)/M)$ is weakly orthogonal to $\tp(c_{<i}/M)$. This is trivial if $f=0$.
\par Suppose $f\neq 0$. By $\val^1_B$-separatedness, $\val^1_B(f(c_i))=\val^1_B(c_i)$. Let $H=H(f(c_i)/B)$. As $H$ is a $B$-$\vee$-definable group, $H(N)$ is $\Aut(M/N)$-invariant. As $\tp(f(d_i)/M)$ is a global $\Aut(M/A)$-invariant extension of $\tp(f(c_i)/M)$, we have $f(d_i)\not\in (M+H(N))$ by \cref{rqueCpInvP2}. Let $V=\dcl(Md_{<i})$. By \cref{tupleWOrth2}, it suffices to show that $V\subset (M+H(N))$. Let $v\in V$. Let $g\in\LC(\mathbb{Q})$ and $m\in M$ be such that $v=m+g(d_{<i})$. Let $h$ be the $A$-definable map witnessing \cref{coupuresInvPossibles} applied to $g(d_{<i})$. As $\val^1_B$ is a $B$-valuation, and the values of the blocks before $c_i$ are strictly smaller than that of $c_i$, we have $G(g(c_{<i})/B)<G(c_i/B)$, thus $G(g(c_{<i})/B)<H$ by \cref{GPlusGrandDonneHPlusGrand}. By \cref{rqueCpInvP}:
$$g(d_{<i})-h\circ g(c_{<i})\in G(g(c_{<i})/B)<H$$ It follows that $v=m+h\circ g(c_{<i})+(g(d_{<i})-h\circ g(c_{<i}))\in (M+H(N))$, concluding the proof.
\end{proof}

Let us now build $\val^2_A$:

\begin{lemma}\label{newArchVsRam}
Let $d,d'\in M$ be such that $G(d/A)=G(d'/A)$, $d$ is Archimedean over $A$ and $d'$ is ramified over $A$. Then $G(d+d'/A)=G(d/A)$ and $d+d'$ is Archimedean over $A$.
\end{lemma}

\begin{proof}
Let $a\in\ram(d'/A)$, and let $G=G(d/A)=G(d'/A)$. Then the (convex) coset $d+a\mod G$ has no point in $A$ and contains $d+d'$, in particular $d+d'\mod G\subset\ct(d+d'/A)$. By ultrametric inequality for the valuation $v^1_A$, $G(d+d'/A)\leqslant G$. By \cref{coupureIncluseDansG}, $\ct(d+d'/A)$ is contained in $d+d'\mod G(d+d'/A)$, therefore we have:
$$d+d'\mod G\subset\ct(d+d'/A)\subset d+d'\mod G(d+d'/A)\subset d+d'\mod G$$
As a result, $G(d+d'/A)=G$. This concludes the proof, as the coset $d+d'\mod G(d+d'/A)$ has no point in $A$.
\end{proof}

\begin{corollary}
The following satisfies the conditions of \cref{valuationPreordre}:
\begin{center}
$x\leqslant y$\\
$\Longleftrightarrow$\\
either $\val^1_A(x)<\val^1_A(y)$, or they have the same value and $x$ is ramified over $A$, or they have the same value and $y$ is Archimedean over $A$.
\end{center}
\end{corollary}

\begin{definition}\label{defVal2}
We define $\val^2_A$ the valuation induced by the above preorder. It is clearly an $A$-valuation which refines $\val^1_A$.
\end{definition}

This new valuation splits the $\val^1_A$-blocks in two, putting the $A$-ramified points below the $A$-Archimedean ones. For instance, in \cref{exempleOrthDOAG}, we have $\val^2_A(c_1)<\val^2_A(c_3)<\val^2_A(c_2)$.

Let us now prove \ref{recollementVal2}:

\begin{proposition}\label{WOrthVal2}
Let $c=(c_{ij})_{ij}$ be a finite tuple from $M$ that is $\val^2_A$-separated, such that its $\val^2_A$-blocks are the $(c_i)_i=((c_{ij})_j)_i$. Then $(\tp(c_i/A))_i$ is weakly orthogonal.
\end{proposition}

\noindent Note that the types which are weakly orthogonal are types over $A$ here, as opposed to the types in \cref{WOrthVal1} which were global types. The respective restrictions over $A$ of non-weakly orthogonal global types can very well be weakly orthogonal.

\begin{proof}
Just like in \cref{WOrthVal1}, enumerate the blocks in increasing order of their value by $\val^2_A$. Choose $i$, and $f\in \LC(\mathbb{Q})$ with $f\neq 0$. As in the proof of \cref{WOrthVal1}, it suffices to find $G$ an $\Aut(M/A)$-invariant convex subgroup of $M$ such that $f(c_i)\not\in (A+G)$, and, for all $g\in\LC(\mathbb{Q})$, $g(c_{<i})\in (A+G)$.

\begin{itemize}
    \item Suppose the $(c_{ij})_j$ are ramified over $A$, then we choose $G=H(c_i/A)$. By separatedness, $G=H(f(c_i)/A)$. Now, for any $g\in\LC(\mathbb{Q})$, we have by \cref{GPlusGrandDonneHPlusGrand} $G(g(c_{<i})/A)<G$, therefore there exists $a$ in $A\setminus \stab(g(c_{<i})/A)$ such that $a\in G$. Let $a'\in A$ be such that $a'\in\left]g(c_{<i})-|a|, g(c_{<i})+|a|\right[$. Then $g(c_{<i})\in \left]a'-|a|, a'+|a|\right[\subset(A+G)$.
    \item Suppose the $(c_{ij})_j$ are Archimedean over $A$, and choose $g$. This time we let $G=G(c_i/A)$. By definition of being Archimedean, $f(c_i)\not\in (A+G)$. Either $G(g(c_{<i})/A)<G$, in which case we have $g(c_{<i})\in (A+H(c_i/A))\subset (A+G)$ just as in the above item, or $G(g(c_{<i})/A)=G$, in which case $g(c_{<i})$ is ramified over $A$ as $\val^2_A(g(c_{<i}))<\val^2_A(c_i)$, therefore $g(c_{<i})\in (A+G)$.
\end{itemize}
This concludes the proof.
\end{proof}

\subsection{Tensor product}
We saw that the blocks of a block extension for the coarsest valuation are weakly orthogonal, thus can be glued by just taking the union. It will no longer be the case for finer valuations, the tensor product of the blocks will no longer commute, and we will have to choose the order of the factors. For the third valuation, which is yet undefined, this choice will have to be taken with care, because some choices may give a global type which does not extend our type over $B$.
\par The following proposition shows us how to properly glue the elements of a $\val^2_B$-block extension.

\begin{proposition}\label{produitTensVal2}
Let $c$ be a finite tuple from $M$ such that $\indep{c}{A}{B}{\cut}$. Suppose $c=(c_{i})_{i\in\{1, 2\}}$ is a $\val^1_B$-block (i.e. a family having only one block) that is $\val^2_B$-separated, and $c_1$ (resp. $c_2$) is the block of $B$-ramified (resp. Archimedean) elements of $c$. Suppose $p_1,p_2$ is a weak $\val^2_B$-block extension of $c$. Then $c$ satisfies the restriction to $B$ of the tensor product of $p_1, p_2$ in any order.
\end{proposition}

\noindent Note that, although the restrictions to $B$ of $p_1\otimes p_2$ and $p_2\otimes p_1$ coincide, these two global types are distinct, i.e. $(p_1, p_2)$ is not weakly orthogonal. This will be explained in the next section.

\begin{proof}
By \cref{WOrthVal2}, $(\tp(c_i/B))_{i\in\{1, 2\}}$ is weakly orthogonal, thus $\tp(c_1/B)$ is complete in $S(B+\mathbb{Q}c_2)$, and it coincides in particular with $(p_1)_{|B+\mathbb{Q}c_2}$. It follows that $c\models (p_1\otimes p_2)_{|B}$, and the same reasoning can be applied to $p_2\otimes p_1$.
\end{proof}

Let us now build $\val^3_B$:

\begin{lemma}\label{newRamVsRam}
Let $d_1,d_2\in M$ be two points which are ramified over $A$, and have the same $\val^2_A$-value. Suppose $\delta(d_1/A)<\delta(d_2/A)$. Then $\val^2_A(d_1+d_2)=\val^2_A(d_2)$, and $\delta(d_1+d_2/A)=\delta(d_2/A)$ (as defined in \cref{defPetitDelta}).
\end{lemma}
\begin{proof}
Let $a_i\in\ram(d_i/A)$. Then, by \cref{triangleIsocele}, $\Delta(d_1+d_2-a_1-a_2)=\delta(d_2/A)$, therefore $d_1+d_2$ is ramified over $A$, and $\delta(d_1+d_2/A)=\delta(d_2/A)$. By \cref{classeUnique}, we have:
$$\ct_>(\Delta(\stab(d_2/A))/\Delta(A))=\ct_>(\Delta(\stab(d_1+d_2/A))/\Delta(A))$$
\par However, those stabilizers are convex subgroups of $A$, thus the sets of their Archimedean classes are initial segments of $\Delta(A)$, so this equality of cuts implies $\Delta(\stab(d_2/A))=\Delta(\stab(d_1+d_2/A))$. It follows that $\stab(d_2/A)=\stab(d_1+d_2/A)$, thus $H(d_2/A)=H(d_1+d_2/A)$, thus $\val^1_A(d_2)=\val^1_A(d_1+d_2)$. This concludes the proof, as $d_2$, $d_1+d_2$ are both ramified over $A$.
\end{proof}

\begin{corollary}
The following satisfies the conditions of \cref{valuationPreordre}:
\begin{center}
$x\leqslant y$\\
$\Longleftrightarrow$\\
either $\val^2_A(x)<\val^2_A(y)$, or they have the same value, they are both ramified over $A$, and $\delta(x/A)\leqslant\delta(y/A)$.
\end{center}
\end{corollary}

\begin{definition}
We define $\val^3_A$ the valuation induced by the above preorder. It is clearly an $A$-valuation which refines $\val^2_A$.
\end{definition}

The valuation splits the ramified $\val^2_A$-blocks in a way which depends on the Archimedean classes $\delta$ added by their elements. Note that, while for $i\in\{1, 2\}$, $\val^i_A$ factors through the quotient by $A$-elementary equivalence (these are model-theoretic objects), $\val^3_A$ is an \textit{algebraic} object that is no longer refined by types.
\par We  need to define additional notions in order to better understand $\val^3_B$-block extensions. Contrary to $\val^2_B$ and $\val^1_B$, a weak $\val^3_B$-block extension of some tuple might not be strong. We need to find necessary and sufficient conditions for such a block extension to be strong.

\begin{proposition}\label{completionRam}
Let $c=(c_{ij})_{ij}$ be a finite $\val^2_B$-block of $B$-ramified points from $M$ that is $\val^3_B$-separated, and whose $\val^3_B$-blocks are the $(c_i)_i=((c_{ij})_j)_i$. Suppose the indices $i$ are ordered such that: $i<k\Longleftrightarrow\delta(c_{i}/B)<\delta(c_{k}/B)$. Let $d=(d_{ij})_{ij}$ be another tuple from $M$. Then $c\equiv_B d$ if and only if the following conditions hold:
\begin{enumerate}
\item For each $i$, we have $d_i\equiv_B c_i$. In particular, each $d_i$ is a $\val^3_B$-block of $B$-ramified points that has the same $\val^2_B$-value as $c$.
\item For every $i$ and $k$, $i<k$ if and only if $\delta(d_{i}/B)<\delta(d_{k}/B)$.
\end{enumerate}
\end{proposition}

\begin{proof}
For each $i$ and $j$, let $b_{ij}\in \ram(c_{ij}/B)$.
\par Suppose $c\equiv_B d$ (then condition $1$ holds). Let $i,k,j,l$ be indices. If $i<k$, then $\tp(c/B)\models |x_{ij}-b_{ij}|<\dfrac{1}{n}|x_{kl}-b_{kl}|$ for every $n$, therefore we have $\Delta(d_{ij}-b_{ij})<\Delta(d_{kl}-b_{kl})$. Moreover, by \cref{critereEquRam}, we have $\delta(d_{ij}/B)=\Delta(d_{ij}-b_{ij})$ (and similarly for $kl$), and condition $2$ holds.
\par Conversely, suppose that the conditions of the list hold. Let $f_i$ be in $\LC^{|c_i|}(\mathbb{Q})$, and let $b'=\sum\limits_i f_i(b_i)$. We have to prove that $c'=\sum\limits_i f_i(c_i)$ and $d'=\sum\limits_i f_i(d_i)$ have the same cut over $B$. This is trivial if each $f_i$ is zero, else let $k$ be the maximal index for which $f_k$ is non-zero. By $\val^3_B$-separatedness, $\val^3_B(f_k(c_k))=\val^3_B(c_{kj})$ for any $j$. Moreover, $f_k((c_{kj}-b_{kj})_j)$ is a linear combination of elements of $G(c_k/B)$, hence it belongs to $G(c_k/B)$, and $f_k(b_k)$ is a ramifier of $f_k(c_k)$. Likewise, we have $f_i(c_i)-f_i(b_i)\in G(c_i/B)$ for any $i$. It follows by maximality of $k$ that:
\begin{equation}\label{tmpCompletionRam}\tag{3} \delta(c_{k}/B)=\Delta(f_k(c_k)-f_k(b_k))>\Delta(f_i(c_i)-f_i(b_i))
\end{equation}
 for each $i\neq k$. By the conditions of the list, this also holds for $d$. Choose an arbitrary index $j$. There exists $n\in\omega_{>0}$ and $\epsilon\in\{\pm 1\}$ such that $f_k(c_k)-f_k(b_k)$ lies between $\dfrac{\epsilon}{n}(c_{kj}-b_{kj})$ and $\epsilon n (c_{kj}-b_{kj})$. By the first condition, this must hold for $c$ replaced with $d$ (with the same $\epsilon, n$). Moreover, as \ref{tmpCompletionRam} holds for $d$, for all $i\neq k$ and $N>0$, we have $|f_i(d_i)-f_i(b_i)|<\dfrac{1}{N}|d_{kj}-b_{kj}|$. By choosing $N$ large enough, and by summing everything, we get that $d'-b'$ lies between $\dfrac{\epsilon}{N+1}(d_{kj}-b_{kj})$ and $\epsilon(N+1)(d_{kj}-b_{kj})$, and we get the same property with $d$ replaced by $c$. Then we have:

$$\begin{array}{rcl}\ct(\Delta(c'-b')/\Delta(B)) & = & \ct(\Delta(c_{kj}-b_{kj})/\Delta(B))\\
 & = & \ct(\Delta(d_{kj}-b_{kj})/\Delta(B))\\
 & = & \ct(\Delta(d'-b')/\Delta(B))
\end{array}
$$
 \par The second equality is a consequence of \cref{critereEquRam} applied to $c_{kj}$, $d_{kj}$ (by the hypothesis $c_k\equiv_B d_k$). The other equalities follow from the inequalities of the previous paragraph. We also have:
$$
d'<b' \Longleftrightarrow \epsilon = -1\Longleftrightarrow c'<b'
$$
as a result, by \cref{critereEquRam}, we have $c'\equiv_B d'$, concluding the proof.
\end{proof}

\begin{remark}\label{abstractionGH}
We recall some facts from previous subsections. Let $c=(c_{ij})_{ij}$ be a finite $\val^2_B$-block of $B$-ramified points from $M$ that is $\val^3_B$-separated, whose $\val^3_B$-blocks are the $(c_i)_i=((c_{ij})_j)_i$. Let $(p_i)_i$ be a weak $\val^3_B$-block extension of $c$. By \cref{coupuresInvPossibles}, and by separatedness, there exists $G$ (resp. $H$) a unique $A$-type-definable(resp. $A$-$\vee$-definable) convex subgroup such that $H<G$, and exactly one of the following conditions holds for each $i$:
\begin{enumerate}
\item For every non-zero $f\in \LC^{|c_i|}(\mathbb{Q})$, there exists $b\in B$ such that we have $p_i(x_i)\models f(x_i)-b\in G\setminus {\underline{G}_M}$.
\item For every non-zero $f\in \LC^{|c_i|}(\mathbb{Q})$, there exists $b\in B$ such that we have $p_i(x_i)\models f(x_i)-b\in {\overline{H}^M}\setminus H$.
\end{enumerate}
Of course we know exactly what $G$ and $H$ are, but this abstraction will simplify the definitions and proofs.
\end{remark}

The following notion will mostly be used to understand strong $\val^3_B$-block extensions, and to prove \ref{recollementVal3}:

\begin{definition}\label{defIO}
If the first condition of \cref{abstractionGH} holds, then we say that $p_i$ is \textit{outer}, else $p_i$ is \textit{inner}.

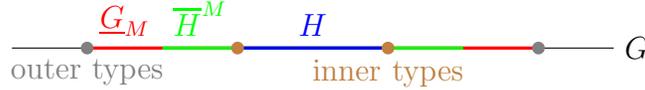
\begin{center}
\begin{tikzpicture}
\draw[black] (-4, 0) -- (4, 0);
\draw[black] (4, 0) node[anchor=west]{$G$};
\draw[red, very thick] (-3, 0) -- (3, 0);
\draw[green, very thick] (-2, 0) -- (2, 0);
\draw[blue, very thick] (-1, 0) -- (1, 0);
\draw[red, very thick] (-2.5, 0) node[anchor=south]{$\underline{G}_M$};
\draw[green, very thick] (-1.5, 0) node[anchor=south]{$\overline{H}^M$};
\draw[blue, very thick] (0, 0) node[anchor=south]{$H$};
\filldraw[gray, thick] (-3, 0) circle(2pt) node[anchor=north]{outer types};
\filldraw[gray, thick] (3, 0) circle(2pt) node[anchor=north]{};
\filldraw[brown, thick] (1, 0) circle(2pt) node[anchor=north]{inner types};
\filldraw[brown, thick] (-1, 0) circle(2pt) node[anchor=south]{};
\end{tikzpicture}
\end{center}
\end{definition}

\begin{lemma}\label{forceBaseInv}
Let $c=(c_{ij})_{ij}$ be a finite $\val^2_B$-block of $B$-ramified points from $M$ that is $\val^3_B$-separated, such that its $\val^3_B$-blocks are the $(c_i)_i=((c_{ij})_j)_i$. Suppose the indices $i$ are ordered such that: $i<k\Longleftrightarrow\delta(c_{i}/B)<\delta(c_{k}/B)$. Let $(p_i)_i$ be a weak $\val^3_B$-block extension of $c$. Let $I$ (resp. $O$) be the set of indices $i$ such that $p_i$ is inner (resp. outer). Then the following are equivalent:
\begin{itemize}
\item $(p_i)_i$ is a strong $\val^3_B$-block extension of $c$.
\item For every $i\in I,o\in O$, we have $i<o$.
\end{itemize}
\end{lemma}

\begin{proof}
The realizations of $\bigcup\limits_i p_i$ satisfy the first condition of \cref{completionRam}, so we are trying to characterize when some of them satisfy the second condition. Let $H=H(c/A)<G(c/A)=G$ be the convex subgroups witnessing \cref{abstractionGH}.
\par Suppose we have $i\in I$ and $o\in O$ such that $o<i$. Let $d\models\bigcup\limits_j p_j$. Then, as $p_i$ is inner, we have $\delta(d_i/M)\in\Delta({\overline{H}^M})$. By the same argument, we have $\delta(d_o/M)\not\in \Delta({\underline{G}_M})$, therefore $\delta(d_i/M)<\delta(d_o/M)$, thus $d\not\equiv_B c$, and we proved the top-to-bottom implication.
\par Conversely, suppose we have $i<o$ for every $i\in I$, $o\in O$. In order to satisfy the second condition of \cref{completionRam}, it suffices to prove that the following partial type is consistent:

$$
q(x)=\bigcup\limits_i p_i(x_i)\cup\left\lbrace |x_{i,f(i)}-b_i|<\dfrac{1}{n}|x_{k,f(k)}-b_k|\ :\ i<k, n>0\right\rbrace
$$
with $f(i)$ an arbitrary choice of a coordinate of $(c_{ij})_j$ for each $i$, and with $b_i\in\ram(c_{i,f(i)}/B)$. Let us build a realization $\beta$ of $q$.
\par Let $(\alpha_i)_{i\in I}\models\bigcup\limits_{i\in I} p_i$. We  define by induction, for each $i\in I$, a tuple $\beta_i\models p_i$ such that, for $i<k$ in $I$, we have $\Delta(\beta_{if(i)}-b_i)<\Delta(\beta_{kf(k)}-b_k)$. Let $N$ be an $|M|^+$-saturated, strongly $|M|^+$-homogeneous elementary extension of $M$ containing $(\alpha_i)_{i\in I}$. Let $k\in I$, and suppose we defined $\beta_i\in N$ for every $i<k$. By saturation, let $\gamma\in N$ such that $\gamma\in {\overline{H}^M}\setminus H$, and $\Delta(\gamma)>\Delta(\beta_{if(i)}-b_i)$ for all $i<k$, i.e. $\gamma\in {\overline{H}^M}\setminus {\underline{({\overline{H}^M})}_{M+\sum\limits_{i<k}\mathbb{Q}\beta_{i, f(i)}}}$. Such a $\gamma$ exists, as we are only dealing with indices of inner types. Then, by \cref{critereEquRam}, there exists $\epsilon\in\{\pm 1\}$ such that $b_k+\epsilon\gamma\equiv_M \alpha_{k,f(k)}$. By strong homogeneity, there exists $\sigma\in \Aut(N/M)$ such that $\sigma(\alpha_{k, f(k)})=b_k+\epsilon\gamma$. We set $\beta_k=\sigma(\alpha_k)$.
\par We can similarly define $(\beta_o)_{o\in O}$. Then $(\beta_i)_{i\in I\cup O}$ satisfies $q$, concluding the proof.\qedhere
\end{proof}

Now we can prove \ref{recollementVal3}:

\begin{proposition}\label{tmpRecollementVal3}
Let $c=(c_{ij})_{ij}$ be a finite $\val^2_B$-block of $B$-ramified points from $M$ that is $\val^3_B$-separated, whose $\val^3_B$-blocks are the $(c_i)_i=((c_{ij})_j)_i$. Let $(p_i)_i$ be a strong $\val^3_B$-block extension of $c$. Then some tensor product of the $p_i$ is consistent with $\tp(c/B)$.
\end{proposition}

\noindent We will see much later, in \cref{finCompletionRam}, that the tensor product in the following proof is in fact the only completion of $(p_i)_i$ (be it $\Aut(M/A)$-invariant or not) which is consistent with $\tp(c/B)$.

\begin{proof}
Suppose the indices $i$ are ordered such that $i<k$ if and only if we have $\delta(c_{i}/B)<\delta(c_{k}/B)$. Define $I$, $O$, $G$, $H$, just as in \cref{forceBaseInv}. Let $\alpha_I\models p_I=\bigcup\limits_{i\in I}p_i$, $\alpha_O\models p_O=\bigcup\limits_{o\in O}p_o$. As the conditions of \cref{forceBaseInv} are satisfied, if we had $\alpha_I\equiv_B (c_i)_{i\in I}$ and $\alpha_O\equiv_B (c_o)_{o\in O}$, then by \cref{completionRam} we would clearly have $\alpha_I\alpha_O\equiv_B c$. Moreover:
$$G(\alpha_I/M)={\overline{H}^M}<{\underline{G}_M}=H(\alpha_O/M)$$
thus one clearly sees that $\alpha_I$ and $\alpha_O$ are $\val^2_M$-blocks of distinct values, thus their types over $M$ are weakly orthogonal by \cref{WOrthVal2}. Therefore, if the respective types over $M$ of $\alpha_I$ and $\alpha_O$ are tensor products of the $p_i$, then so is the type of $\alpha_I\alpha_O$. As a result, we can assume that either $I$ or $O$ is empty.
\par Suppose $O=\emptyset$. For $k\in I$, and $c'=(c_i)_{i>k}$, let us show that $c_k\models p_{k|D}$, with $D$ the $\mathbb{Q}$-vector subspace generated by $Bc'$. This implies that $\tp(c/B)$ is consistent with the tensor product of the $p_i$ in increasing order. Let $f$ be non-zero in $\LC^{|c_k|}(\mathbb{Q})$. By $\val^3_B$-separatedness, we have $\delta(f(c_k)/B)=\delta(c_{k}/B)$ (call this Archimedean class $\delta$). Let $b$ be in $\ram(f(c_k)/B)$. As $i\in I$, we have $p_k(x_k)\models f(x_k)-b\in{\overline{H}^M}\setminus H$. It is enough to show that $f(c_k)-b\in {\overline{H}^D}$, for the cut over $M$ corresponding to the pushforward of $p_k$ by $f$ would be included in $\ct(f(c_k)/D)$, and we could conclude. For $i>k$, let $\delta_i=\delta(c_{i}/B)$. By $\val^3_B$-separatedness, we clearly have $\Delta(D)=\Delta(B)\cup\left\lbrace\delta_i|i>k\right\rbrace$. By definition of $I$, we have $\delta<\delta_i$ for every $i>k$. Moreover, we have by \cref{classeUnique}:
$$\ct(\delta_i/\Delta(B))=\ct_>(\Delta(H)\cap\Delta(B)/\Delta(B))$$
as a result, $\Delta(H)\cap\Delta(D)=\Delta(H)\cap\Delta(B)$, and:
$$\ct(\delta/\Delta(D))=\ct_>(\Delta(H)\cap\Delta(D)/\Delta(D))$$
thus $\delta\in\Delta({\overline{H}^D})$, concluding the proof.

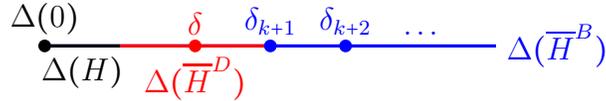
\begin{center}
\begin{tikzpicture}
\draw[blue, very thick] (0, 0) -- (6, 0);
\draw[black, very thick] (0, 0) -- (1, 0);
\draw[red, very thick] (1, 0) -- (3, 0);
\filldraw[black, thick] (0, 0) circle(2pt) node[anchor=south]{$\Delta(0)$};
\filldraw[red, thick] (2, 0) circle(2pt) node[anchor=south]{$\delta$};
\filldraw[blue, thick] (3, 0) circle(2pt) node[anchor=south]{$\delta_{k+1}$};
\filldraw[blue, thick] (4, 0) circle(2pt) node[anchor=south]{$\delta_{k+2}$};
\draw[blue] (5, 0) node[anchor=south]{$\ldots $};
\draw[black] (0.5, 0) node[anchor=north]{$\Delta(H)$};
\draw[red] (2, 0) node[anchor=north]{$\Delta(\overline{H}^D)$};
\draw[blue] (6, 0) node[anchor=west]{$\Delta(\overline{H}^B)$};
\end{tikzpicture}
\end{center}

\par Now, if instead $I=\emptyset$, then a similar proof would show that the tensor product of the $p_o$ in decreasing order would be consistent with $\tp(c/B)$.\qedhere
\end{proof}

\begin{corollary}\label{preuveRecollementVal3}
The property \ref{recollementVal3} holds.
\end{corollary}

\begin{proof}
Proposition \ref{tmpRecollementVal3} deals with the particular case of a $\val^2_B$-block. Then, given the global invariant extensions of the types of each $\val^2_B$-block given by \cref{tmpRecollementVal3}, some tensor product of those global types witnesses the statement, by \cref{produitTensVal2} and \cref{WOrthVal1}.
\end{proof}

\section{How to build a block extension}\label{secBaseInv}

In the previous section, we defined the key valuations $\val^i_A$, and we proved the properties \ref{recollementVal3}, \ref{recollementVal2}, \ref{recollementVal1}. In this section, we  give a proof of \ref{existenceBaseInvariance}. The main technical goal is to show that $\indep{}{A}{B}{\cut}$ coincides with $\indep{}{A}{B}{\inv}$, when what we put on the left is a $\val^3_B$-separated $\val^3_B$-block. This will show the existence of a weak block extension, and then it will not be hard to show that such a block extension can be chosen strong. In fact, given some $\val^3_B$-separated tuple $c$, we  describe explicitly the homeomorphism type of the space of global $\Aut(M/A)$-invariant extensions of $\tp(c/B)$.
We adopt in this section \cref{DOAGHypMonstre}.

\subsection{The Archimedean case}\label{existenceBaseInvArch}
In this subsection, we deal with the case where $c$ is a $\val^3_B$-separated block of $B$-Archimedean points from $M$.

\begin{remark}\label{ArchBImpliqueArchA}
Suppose $c$ is a $\val^3_B$-separated block of $B$-Archimedean points from $M$. As $\indep{c}{A}{B}{\cut}$, the $c_i$ must be Archimedean over $A$. Indeed, by contraposition, if $c_i$ is ramified over $A$ for some $i$, then we cannot have $\delta(c_i/A)\in\Delta(B)$, otherwise there would exist $b\in B$, $a\in A$ such that $\delta(c_i/A)=\Delta(c_i-a)=\Delta(b)$, thus $c_i$ would lie between $a+\dfrac{1}{n}b$ and $a+nb$ for some $n\in \mathbb{Z}$, which would contradict $\indep{c}{A}{B}{\cut}$. Now, as $\Delta(c_i-a)\not\in\Delta(B)$ we have a contradiction with the hypothesis that all the $c_i$ are Archimedean over $B$.
\end{remark}
\begin{assumptions}
On top of \cref{DOAGHypMonstre}, we assume that $c$ is a $\val^3_B$-separated block of $B$-Archimedean points from $M$. By \cref{ArchBImpliqueArchA} and \cref{coupInvArch}, let us also fix $G=G(c/A)=G(c/B)$.
\end{assumptions}

\begin{lemma}\label{tmpClassifArch}
Let $p_i\in S(M)$ be a global $\Aut(M/A)$-invariant extension of $\tp(c_i/B)$. Suppose $q$ is a complete global extension of $\tp(c/B)$ which extends $\bigcup\limits_i p_i$. If we have $q(x)\models f(x)-f(c)\not\in {\underline{G}_M}$ for all non-zero $f\in \LC^n(\mathbb{Q})$, then $q$ is $\Aut(M/A)$-invariant.
\end{lemma}

\begin{proof}
Suppose by contradiction $q$ is not $\Aut(M/A)$-invariant. Then there exists a non-zero $f\in \LC^n(\mathbb{Q})$, and a cut $X$ over $M$, such that $q(x)\models f(x)\in X$, and the global $1$-type induced by $X$ is not $\Aut(M/A)$-invariant. As $\indep{c}{A}{B}{\cut}$, we have $\indep{f(c)}{A}{B}{\inv}$. By \cref{coupInvArch}, we clearly see that the two global $\Aut(M/A)$-invariant extensions of $\tp(f(c)/B)$ correspond to the translates by $f(c)$ of the two connected components of $G\setminus{\underline{G}_M}$. We  obtain a contradiction by showing that $X$ coincides with one of those. It is enough to show that $q(x)\models f(x)-f(c)\in G\setminus {\underline{G}_M}$. By hypothesis, we just need to have $q(x)\models f(x)-f(c)\in G$, which follows from the fact that $c$ is a $\val^3_B$-separated block. Indeed, $\val^3_B(f(c))=\val^3_B(c)$, therefore $f(c)$ is Archimedean over $B$, and $G(f(c)/B)=G$. This implies by \cref{coupureArch} that $\tp(f(c)/B)\models y-f(c)\in G$, concluding the proof.
\end{proof}

It turns out that the space of global $\Aut(M/A)$-invariant extensions of $\tp(c/B)$ in the Archimedean case has a very simple description:

\begin{proposition}\label{classificationArch}
Let $N$ be some $|M|^+$-saturated and strongly $|M|^+$-homogeneous elementary extension of $M$. Let $H=({\underline{G}_M})({N})$ be the convex subgroup of ${N}$ generated by $G(M)$, and let ${N}'$ be the ordered Abelian group $\faktor{{N}}{H}$, which is a model of \DOAG~ by saturation. Let $F_1\subset S^{|c|}(M)$ be the closed set of $\Aut(M/A)$-invariant extensions of $\tp(c/B)$, and $F_2\subset S^{|c|}(\{0\})$ be the closed set of types of $\mathbb{Q}$-free families. Then the map:
$$g:\ \tp^{{N}}(\alpha/M)\longmapsto \tp^{{N}'}((\alpha_i-c_i\mod H)_i/\{0\})$$
is a well-defined homeomorphism $F_1\longrightarrow F_2$.
\end{proposition}

\begin{proof}
Let $f\in \LC^n(\mathbb{Q})$ be non-zero, and let $\alpha\in N$ realize a global $\Aut(M/A)$-invariant extension of $\tp(c/B)$. Then, by separatedness, $f(c)$ is Archimedean over $A$, and $G(f(c)/A)=G$, thus $f(c)-f(\alpha)\not\in {\underline{G}_M}$ by \cref{coupInvArch}, i.e. $f(c)-f(\alpha)\not\in H$. No non-trivial linear combination of $(c_i-\alpha_i)_i$ belongs to $H$, hence $(c_i-\alpha_i\mod H)_i$ is $\mathbb{Q}$-free.
\par Moreover, if $\beta\equiv_M\alpha$, then we have $f((\beta_i-c_i\mod H)_i)>0$ in $N'$ if and only if $f((\beta_i-c_i)_i)\in\bigcap\limits_{m\in H}]m, +\infty[$ in $N$, if and only if the same holds for $\alpha$, therefore $(\beta_i-c_i\mod H)_i\equiv_{\{0\}}(\alpha_i-c_i\mod H)_i$ in $N'$. It follows that $g$ is well-defined, and it is continuous by definition (and quantifier elimination).
\par Let $p_1,p_2\in F_1$ be distinct. Then there exists two distinct cuts $X_1,X_2$ over $M$, and $f\in \LC^n(\mathbb{Q})$ non-zero, such that $p_i(x)\models f(x)\in X_i$. However, by \cref{coupInvArch}, $X_i$ must be the translate by $f(c)$ of a connected component of $G\setminus {\underline{G}_M}$. As a result, one of the $X_i$ (say, $X_1$) is the leftmost connected component, and the other is the rightmost one. We have:
$$p_1(x)\models f(x)-f(c)< {\underline{G}_M}$$
thus $g(p_1)\models f(x)<0$. With the same reasoning, $g(p_2)\models f(x)>0$, and we conclude that $g(p_1)\neq g(p_2)$, proving that $g$ is injective.
\par Now let us prove surjectivity. Let $\beta'$ be a $\mathbb{Q}$-free family from ${N}'$. Let $\beta$ be a preimage of $\beta'$ by $\pi:\ {N}\longrightarrow{N}'$. Let $d_1\in{N}$ such that $\Delta(d_1)>\Delta(\beta_i)$ for every $i$. Let $d_2\in{N}$ such that $\Delta(d_2)\in\Delta(G)\setminus\Delta({\underline{G}_M})$. Then $|d_1|$ and $|d_2|$ have the same cut over $G(M)$: $+\infty$, the cut of positive elements larger than all the elements of $G(M)$. By strong homogeneity, there is an automorphism $\sigma\in \Aut({N}/G(M))$ such that $\sigma(|d_1|)=|d_2|$. Since $\sigma$ fixes $G(M)$, we have $\sigma(H)=H$, thus $\sigma':\ \pi(x)\longmapsto \pi(\sigma(x))$ is a well-defined automorphism of the ordered group ${N}'$. As $\beta'$ is $\mathbb{Q}$-free, so is $\sigma'(\beta')$. As $\sigma(\beta)$ is a preimage of $\sigma'(\beta')$, none of its elements is in $H$. Moreover, $|\sigma(\beta_i)|\leqslant \sigma(|d_1|)=|d_2|\in G$ for every $i$, thus $\sigma(\beta_i)\in G$. As $\tp^{N'}(\beta'/\{0\})=\tp^{N'}(\sigma'(\beta')/\{0\})$, and we are looking for a preimage of this type by $g$, we may assume $\sigma=id$, and thus, by $\mathbb{Q}$-freeness of $\beta'$, $f(\beta)\in G\setminus{\underline{G}_M}$ for every non-zero $f\in \LC^n(\mathbb{Q})$. Let $p=\tp((c_i+\beta_i)_i/M)$. As $\beta_i\in G$ for all $i$,  $p$ extends $\tp(c/B)$, and it is $\Aut(M/A)$-invariant by \cref{tmpClassifArch}. Thus $p\in F_1$, and it is a preimage of $\tp(\beta'/\{0\})$ by $g$.
\par We showed that $g$ is a continuous bijection, and we conclude by compactness and separation.
\end{proof}

Since $F_2$ is non-empty, we have in particular proved the following:

\begin{corollary}\label{conclusionCasArchBaseInv}
If $c$ is a $\val^3_B$-separated block of $B$-Archimedean points from $M$, then we have $\indep{c}{A}{B}{\inv}$.
\end{corollary}

The next subsection gives tools to understand what $F_2$ looks like. For the most part, it is used to deal with the ramified blocks.

\subsection{Archimedean groups}

We need to define some algebraic notions about Archimedean groups in order to well-understand the global invariant extensions of the types of ramified blocks. Let us start with well-known facts:

\begin{proposition}\label{propUnivR}
Let $G$ be an ordered Abelian group. Then $G$ has at most two Archimedean classes (including $\Delta(0)$) if and only if it embeds into $(\mathbb{R}, +, <)$. In that case, for all $x\in G_{>0}$, for all $\mu\in\mathbb{R}_{>0}$, there exists a unique embedding of ordered groups $\sigma:G\longrightarrow\mathbb{R}$ sending $x$ to $\mu$.
\end{proposition}

The idea of the proof is to send each $y\in G$ to $\mu\cdot\left( \sup\left\lbrace\lambda\in\mathbb{Q}|\lambda y<x\right\rbrace\right)^{-1}$.

\begin{remark}
In the special case $G=\mathbb{R}$, we see that an embedding of ordered groups $\sigma:\ \mathbb{R}\longrightarrow\mathbb{R}$ is uniquely determined by the choice of $\sigma(1)$, hence $\sigma$ actually coincides with the automorphism $x\longmapsto \sigma(1)\cdot x$. This establishes an easy description of the group of ordered group automorphisms of $\mathbb{R}$, which is naturally isomorphic to the multiplicative group $\mathbb{R}_{>0}$.
\end{remark}

\begin{definition}\label{defPArch}
We define $\mathbb{P}^+_n$, the space of half-lines in dimension $n$ (plus the origin) as the set of orbits of $\mathbb{R}^n$ under the automorphisms of the ordered group $\mathbb{R}$. We call $\mathbb{P}^+$ the canonical map $\mathbb{R}^n\longrightarrow\mathbb{P}^+_n$.
\par Suppose $G$ is an Archimedean ordered Abelian group, and let $u$ be a finite tuple from $G$. By the universal property of \cref{propUnivR}, we see that $\mathbb{P}^+(\sigma(u))$ is always the same for any ordered group embedding $\sigma:\ G\longrightarrow \mathbb{R}$. As a result, we can define $\mathbb{P}^+(u)\in\mathbb{P}^+_{|u|}$ as this unique class.
\end{definition}

\begin{remark}\label{DOAGFPPlusCommutent}
If $f\in \LC^n(\mathbb{Q})$, then $f$ commutes with any automorphism, thus one can easily show that $f(\mathbb{P}^+(u))=\mathbb{P}^+(f(u))\in\mathbb{P}^+_1$ for any $u\in\mathbb{R}^n$. However, $\mathbb{P}^+_1$ is not very complicated to describe: $f(\mathbb{P}^+(u))$ is either $\mathbb{R}_{>0}$, $\mathbb{R}_{<0}$ or $\{0\}$.
\end{remark}

\begin{definition}
Let $u\in\mathbb{R}^n$. We  say that $\mathbb{P}^+(u)$ is $\mathbb{Q}$\textit{-free} if $u$ is $\mathbb{Q}$-free. By \cref{DOAGFPPlusCommutent}, this definition does not depend on the choice of the representative $u$.
\end{definition}

\begin{proposition}\label{DOAGComp}
Let $G$ be an ordered Abelian group. Then for any subgroup $H$ of $G$, and $\Delta(0)\neq\delta\in\Delta(G)$, the quotient: 
$$H_\delta=\faktor{H_{\leqslant\delta}}{H_{<\delta}}$$
as defined in \cref{defIndiceClasseArch}, is Archimedean.
\end{proposition}
In the literature, the $(H_\delta)_{\delta\in\Delta(H)}$ are called the ``components" of $H$.

\begin{definition}\label{defPArch2}
Let $\delta\in\Delta(M)$, and $u$ a tuple from $M_{=\delta}$. By \cref{DOAGComp}, we can define $\mathbb{P}^+(u)=\mathbb{P}^+(u\mod M_{<\delta})$ with respect to the definition \ref{defPArch} in the Archimedean group $\faktor{M_{\leqslant\delta}}{M_{<\delta}}$.
\par Suppose $d=(d_1\ldots d_n)$ is a $\val^3_B$-block of $B$-ramified elements of $M$, and let $\delta=\delta(d_i/B)\in\Delta(M)$. Let $b_i\in\ram(d_i/B)$. Then we can define $\mathbb{P}^+(d/B)=\mathbb{P}^+((d_i-b_i)_i)$. As $H(d_i/B)\leqslant M_{<\delta}$, by \cref{conditionsRam}, this definition does not depend on the choice of the $b_i$.
\end{definition}

\begin{proposition}\label{val3BSep-Libre}
Let $d=(d_1\ldots d_n)$ be a finite $\val^3_B$-block of $B$-ramified points from $M$. Then $d$ is $\val^3_B$-separated if and only if $\mathbb{P}^+(d/B)$ is $\mathbb{Q}$-free.
\end{proposition}

\begin{proof}
Let $b_i\in\ram(d_i/B)$, and let $\delta=\Delta(d/B)$. \par Then $d$ is not $\val^3_B$-separated if and only if $\val^3_B(f(d))<\val^3_B(d)$ for some non-zero $f\in \LC^n(\mathbb{Q})$, if and only if there exists such an $f$ with $\Delta(f(d)-f(b))<\delta$, if and only if we have $f$ such that $\{0\}= f(\mathbb{P}^+((d_i-b_i)_i))=f(\mathbb{P}^+(d/B))$, if and only if $\mathbb{P}^+(d/B)$ is not $\mathbb{Q}$-free.
\end{proof}

\subsection{The ramified case}\label{existenceBaseInvRam}
\begin{assumptions}
In this subsection, on top of \cref{DOAGHypMonstre} suppose $c=(c_i)_i$ is a $\val^3_B$-separated $\val^3_B$-block of $B$-ramified points. Fix $G=G(c/B)$, $H=H(c/B)$, $\delta=\delta(c/B)$. For each $i$, let $b_i$ be in $\ram(c_i/B)$, and write $b=(b_i)_i$.
\end{assumptions}
\par Recall that \cref{critereEquRam} characterizes the type of a ramified singleton in terms of cuts of Archimedean classes. We can now extend this characterization to the type of our $\val^3_B$-separated ramified $\val^3_B$-block:

\begin{proposition}\label{completionRamSep}
Let $d=(d_i)_i$ be a tuple from $M$. Then we have $d\equiv_B c$ if and only if the following conditions hold:
\begin{itemize}
\item $d_i\equiv_B c_i$ for every $i$.
\item $d$ is a $\val^3_B$-block, i.e. $\delta(d_i/B)=\delta(d_j/B)$ for every $i$, $j$.
\item $\mathbb{P}^+(d/B)=\mathbb{P}^+(c/B)$.
\end{itemize}
\end{proposition}

\begin{proof}
\par Suppose the conditions of the list hold. Let $f\in \LC^n(\mathbb{Q})$ be non-zero. Then, as $c$ is $\val^3_B$-separated, $\mathbb{P}^+(c/B)$ is $\mathbb{Q}$-free, therefore $\Delta(f(c)-f(b))=\delta$. As a result, $\ct(f(c)/B)$ is the translate by $f(b)$ of one of the connected components of $G\setminus H$. The same holds for $d$, but we need to make sure that $f(c)$ and $f(d)$ lie in the same connected component. The thing is, $f(c)$ lies in the rightmost one if and only if $(f(c)-f(b)\mod M_{<\delta})>0$ in the ordered group $\faktor{M_{\leqslant\delta}}{M_{<\delta}}$, if and only if $\mathbb{P}^+(f(c)-f(b))=\mathbb{R}_{>0}$, if and only if $\mathbb{R}_{>0}=f(\mathbb{P}^+((c_i-b_i)_i))=f(\mathbb{P}^+(c/B))$. Now, by the hypothesis $\mathbb{P}^+(d/B)=\mathbb{P}^+(c/B)$, those conditions are equivalent to $f(d)$ lying in the rightmost connected component, and we get the bottom-to-top direction.
\par Conversely, suppose now $d\equiv_B c$, in particular the first condition holds. As for the second, $d$ must be a $\val^3_B$-block, otherwise it would be witnessed in its type by formulas $(|x_i-b_i|>n|x_j-b_j|)_n$ for some $i\neq j$, which would contradict the hypothesis on $c$. By strong homogeneity of $M$, let $\sigma\in \Aut(M/B)$ be such that $\sigma(c)=d$. The equivalence relation $\Delta(x)=\Delta(y)$ is an $\vee$-definable subset of $S^2(\emptyset)$, therefore it is invariant by $\sigma$. As a result, the map:
$$\sigma':\ (x\mod M_{<\delta})\longmapsto (\sigma(x)\mod M_{<\delta(d/B)})$$
is a well-defined automorphism of ordered groups from the quotient $\faktor{M_{\leqslant\delta}}{M_{<\delta}}$ to $\faktor{M_{\leqslant\delta(d/B)}}{M_{<\delta(d/B)}}$, which sends each $c_i-b_i$ to $d_i-b_i$. Let $\tau$ be an embedding of ordered groups $\faktor{M_{\leqslant\delta(d/B)}}{M_{<\delta(d/B)}}\rightarrow \mathbb{R}$. Then $\mathbb{P}^+(c/B)\ni\tau\circ\sigma'((c_i-b_i)_i)=\tau((d_i-b_i)_i)\in\mathbb{P}^+(d/B)$, thus the last condition of the list holds, concluding the proof.
\end{proof}

\begin{lemma}\label{lemmeCompletionRamGlobal}
Choose an arbitrary index $i$. Suppose that $p$ is a global $\Aut(M/A)$-invariant extension of $\tp(c_i/B)$. Then $q=p\cup\tp(c/B)$ is a complete global type in $S(M)$.
\end{lemma}

\noindent Note that $q$ may not be $\Aut(M/A)$-invariant. We will see an example soon.

\begin{proof}
Proving that $q$ is consistent is easy: choose $d$ a realization of $p$ in some strongly $|B|^+$-homogeneous elementary extension $N$ of $M$, choose $\sigma$ in $\Aut(N/B)$ such that $\sigma(c_i)=d$, then $\sigma(c)$ is a realization of $q$.
\par Let $\alpha\models q$, and $\delta'=\Delta(\alpha_i-b_i)$. As $\alpha_i\models p$, we have:
$$\delta'\in(\Delta(G)\setminus\Delta({\underline{G}_M}))\cup(\Delta({\overline{H}^M})\setminus\Delta(H))$$
in particular it does not belong to $\Delta(M)$. For every $j\neq i$, not only do we have $\Delta(\alpha_j-b_j)\not\in\Delta(B)$ as $\alpha_j\equiv_B c_j$, but in fact we have $\Delta(\alpha_j-b_j)=\delta'$ by the second condition of \cref{completionRamSep}. In particular $\alpha$ is a $\val^3_M$-block, and we have:
$$\mathbb{P}^+(\alpha/M)=\mathbb{P}^+((\alpha_j-b_j)_j)=\mathbb{P}^+(\alpha/B)=\mathbb{P}^+(c/B)$$
\par The last equality follows from \cref{completionRamSep}, and the others follow from the definitions. We established that, for all $\beta\models q$, $\beta$ is a $\val^3_M$-block, and $\mathbb{P}^+(\beta/M)=\mathbb{P}^+(\alpha/M)$. By \cref{completionRamSep} (with $c,B$ replaced by $\alpha,M$), in order to check that $q$ is complete, it is enough to show that for every $j$, the pushforward $q_j$ of $q$ by the $j$-th projection, is a complete type in $S(M)$.
\par Let $f\in \LC^N(\mathbb{Q})$ be the projection on the $j$-th coordinate.
\begin{itemize}
\item On one hand, $q_j(x_j)\models x_j>b_j$ if and only if $f(\mathbb{P}^+(c/B))>0$.
\item On the other hand, we have:
$$q(x)\models \ct(\Delta(x_j-b_j)/\Delta(M))=\ct(\Delta(x_i-b_i)/\Delta(M))=\ct(\delta'/\Delta(M))$$
\par The first equality follows from the fact that the realizations of $q$ are $\val^3_M$-blocks, and the last equality follows from \cref{critereEquRam}, as $p$ is complete. In particular:
$$q_j(x_j)\models\ct(\Delta(x_j-b_j)/\Delta(M))=\ct(\delta'/\Delta(M))$$
\end{itemize}
Then, we can apply \cref{critereEquRam} to show that $q_j$, and thus $q$, is complete.
\end{proof}

\begin{proposition}
Let $i$ (if it exists) be such that $c_i$ is Archimedean over $A$. Then $\tp(c_i/B)$ has exactly one global $\Aut(M/A)$-invariant extension $p$, and $p\cup\tp(c/B)$ is a complete, $\Aut(M/A)$-invariant global type in $S(M)$.
\end{proposition}

\begin{proof}
Suppose for simplicity $c_i>b_i$ (else apply the proposition with $c_i, b_i$ replaced by $-c_i,-b_i$). Recall that $\indep{c}{A}{B}{\cut}$, thus $\tp(c_i/B)$ must have at least one global $\Aut(M/A)$-invariant extension. By \cref{coupInvArch}, $G(c_i/A)=G$, and the types $p_>,p_<$ over $M$ corresponding respectively to the cuts:
$$\ct_>(\ct(c_i/A)/M)=\ct_>(b_i\mod G(M)/M)$$
$$\ct_<(\ct(c_i/A)/M)=\ct_<(b_i\mod G(M)/M)$$
are clearly global $\Aut(M/A)$-invariant extensions of $\tp(c_i/A)$. These two types are the only ones implying that $x-b_i\in G\setminus {\underline{G}_M}$. As $c_i>b_i$, only $p_>$ extends $\tp(c_i/B)$. By \cref{coupInvArch}, there are no other global $\Aut(M/A)$-invariant extension of $\tp(c_i/B)$.
\par Let $q=p_>\cup\tp(c/B)$, which is complete by \cref{lemmeCompletionRamGlobal}. Let us show that $q$ is $\Aut(M/A)$-invariant. As $\mathbb{P}^+(c/B)$ is $\mathbb{Q}$-free, the realizations of $q$ are $\val^3_M$-separated by \cref{val3BSep-Libre}. As a result, for every non-zero $f\in \LC^N(\mathbb{Q})$, $q(x)\models f(x)-f(b)\in G\setminus {\underline{G}_M}$. As $G=G(c_i/A)$ and ${\underline{G}_M}$ are $\Aut(M/A)$-invariant, we only have to show that $X=f(b)\mod {\underline{G}_M}$ is $\Aut(M/A)$-invariant as a subset of $M$, because:
$$q(x)\models\ct(f(x)/M)\in\{\ct_>(X/M),\ct_<(X/M)\}$$
\par However, as a subset of $M$, $X$ coincides with $f(b)\mod G(M)$. If $f(c)$ is Archimedean over $A$, then $X=\ct(f(c)/A)(M)$, which is invariant under $\Aut(M/A)$. Else, by definition, there exists $a\in A$ such that $f(c)-a\in G(M)$. However, we also have $f(c)-f(b)\in G(M)$, therefore $X= a\mod G(M)$, which is $\Aut(M/A)$-invariant. This concludes the proof.
\end{proof}

Note that $p_>$, and thus $q$, is outer (with respect to \cref{defIO}).

\begin{remark}
If $j\neq i$, and $c_j$ is ramified over $A$, then $\tp(c_j/B)$ might have two global $\Aut(M/A)$-invariant extensions, and the union of each of those extensions with $\tp(c/B)$ would be a consistent, complete type in $S(M)$. However, one of the two will not be $\Aut(M/A)$-invariant, as it will not be consistent with $p_>$. In the above proof, we would encounter a problem where we would want to show that $X$ is $\Aut(M/A)$-invariant, because we would have to deal with the case where $f(c)$ is Archimedean over $A$, and $X=f(b)\mod H$.
\par This is exactly what happens in the following example:
\end{remark}

\begin{example}
Let $A=\mathbb{Q}$, $B=\mathbb{Q}+\mathbb{Q}\sqrt{2}$, $\epsilon$ a positive infinitesimal element, $c_1=\sqrt{2}+\epsilon$, and $c_2=\epsilon\cdot\sqrt{2}$. Then $\indep{c_1c_2}{A}{B}{\cut}$, and $c_1c_2$ is a $\val^3_B$-separated block. Moreover, $\tp(c_1/B)$ has exactly one global $\Aut(M/A)$-invariant extension, and $\tp(c_2/B)$ has two, one of which is $p$, the type of a positive element that is infinitesimal with respect to $M$. Then $q=p\cup\tp(c_1c_2/B)$ is complete and consistent in $S(M)$, but it is not $\Aut(M/A)$-invariant. The reason is that $q(x, y)\models \sqrt{2}\leqslant x\leqslant \sqrt{2}+\epsilon$.
\end{example}

\begin{proposition}
Suppose $c_i$ is ramified over $A$ for every $i$. Choose an arbitrary index $i$. Let $H'=H(c/A),G'=G(c/A)$. If $H\neq H'$ or $G\neq G'$, then $\tp(c_i/B)$ has exactly one global $\Aut(M/A)$-invariant extension, else it has exactly two. Moreover, if $p$ is such an extension, then $p\cup\tp(c/B)$ is a complete $\Aut(M/A)$-invariant type in $S(M)$.
\end{proposition}

\begin{proof}
We can assume $b_j\in A$ for every $j$. Suppose $c_i>b_i$. As $\indep{c}{A}{B}{\cut}$, recall that $H'\leqslant H\leqslant G\leqslant G'$. Let $p_<,p_>$ be the global $1$-types corresponding to the respective cuts $\ct_<(\ct(c_i/A)/M)=\ct_>(b_i\mod H'/M)$, $\ct_>(\ct(c_i/A)/M)=\ct_>(b_i\mod {\underline{G'}_M}/M)$. One can see that these two types are the only global $\Aut(M/A)$-invariant extensions of $\tp(c_i/A)$. Moreover $p_<$ (resp. $p_>$) is consistent with $\tp(c_i/B)$ if and only if $H=H'$ (resp. $G=G'$).
\par Now, let $p$ be a global $\Aut(M/A)$-invariant extension of $\tp(c_i/B)$, and $q=p\cup\tp(c/B)$. Then $q$ is complete by \cref{lemmeCompletionRamGlobal}. Let us show that $q$ is $\Aut(M/A)$-invariant.
\par Let $f\in \LC^n(\mathbb{Q})$ be non-zero. One just has to show that both the cosets $f(b)\mod H'(M)$ and $f(b)\mod G'(M)$ are $\Aut(M/A)$-invariant. However, $H'(M)$ and $G'(M)$ are both $\Aut(M/A)$-invariant, and the $b_i$ were now chosen in $A$. This concludes the proof.
\end{proof}

Note that $p_<$ (thus the extension $p_<\cup\tp(c/B)$) is inner, while $p_>$ (thus $p_>\cup\tp(c/B)$) is outer.

\begin{remark}
The groups $H'$ and $G'$ have the same points in $A$. By quantifier elimination, it is not hard to show that $G'$ does not admit an $A$-($\vee$/type)-definable proper convex subgroup that strictly contains $H'$. Recall that $H'\leqslant H\leqslant G\leqslant G'$ by cut-independence. As a result, if $G\neq G'$, then $G$ cannot be $A$-type-definable. Similarly, if $H\neq H'$, then $H$ is not $A$-$\vee$-definable.
\end{remark}

\begin{corollary}\label{DOAGDescription1ou2ExtInvRam}
If one of the $c_i$ is Archimedean over $A$, or $H$ is not $A$-$\vee$-definable, or $G$ is not $A$-type-definable, then $\tp(c/B)$ has exactly one global $\Aut(M/A)$-invariant extension. Else, it has exactly two. More precisely, if $G$ is $A$-type-definable, then $\tp(c/B)$ admits a global $\Aut(M/A)$-invariant extension that is outer ; while if $H$ is $A$-$\vee$-definable and none of the $c_i$ is Archimedean over $A$, then $\tp(c/B)$ has a global $\Aut(M/A)$-invariant extension that is inner.
\end{corollary}

\subsection{Gluing everything together}\label{homeoDOAG}
We recall that we adopt \cref{DOAGHypMonstre}. We now have enough tools to prove \ref{existenceBaseInvariance}. We can actually prove the following more precise statement:

\begin{proposition}\label{classificationRam}
Suppose $c=(c_{ik})_{ik}$ is a finite $\val^2_B$-block of $B$-ramified points from $M$ that is $\val^3_B$-separated, such that its $\val^3_B$-blocks are the $(c_i)_{i\in E}=((c_{ik})_k)_{i\in E}$. Let $G=G(c/B)$, $H=H(c/B)$. Suppose the indices $i\in E$ are ordered such that $i<l$ if and only if we have $\delta(c_{i}/B)<\delta(c_{l}/B)$. Define $I,O\subset E$ as follows:
\begin{itemize}
\item If $H$ is not $A$-$\vee$-definable, then $I=\emptyset,O=E$.
\item If $G$ is not $A$-type-definable, then $I=E,O=\emptyset$.
\item Else, $I=\emptyset$, and $O$ is the least final segment of $E$ containing the set of indices $o\in E$ such that one of the $(c_{ok})_k$ is Archimedean over $A$.
\end{itemize}
Either way, let $J=E\setminus(I\cup O)$. Then the following conditions hold:
\begin{enumerate}
\item For all $i\in I$, $\tp(c_i/B)$ has exactly one $\Aut(M/A)$-invariant global extension $p_i$, and it is inner.
\item For all $o\in O$, $\tp(c_o/B)$ has at least one $\Aut(M/A)$-invariant global extension, of which exactly one is outer, denote it by $q_o$.
\item For all $j\in J$, $\tp(c_j/B)$ has exactly two $\Aut(M/A)$-invariant global extensions: $r_j$, which is inner, and $s_j$, which is outer.
\item The map $K\longmapsto (p_i)_{i\in I},(r_j)_{j\in K},(s_j)_{j\in J\setminus K},(q_o)_{o\in O}$ is a bijection from the set of initial segments of $J$ to the set of strong $\val^3_B$-block extensions of $c$. In particular, $c$ admits exactly $|J|+1$ strong $\val^3_B$-block extensions.
\end{enumerate}
\end{proposition}

\begin{proof}
Note that $I,O,J$ are pairwise-disjoint convex subsets of $E$ which cover $E$. In fact, $I$ is an initial segment of $E$, $O$ is a final segment, and $J$ is in-between. The first two conditions are easy consequences of \cref{DOAGDescription1ou2ExtInvRam}. The third condition also easily follows from the fact that for all $j\in J$, none of the $(c_{jk})_k$ is Archimedean over $A$. There remains to prove the last condition.
\par Suppose we have $o\in O$ such that $\tp(c_o/B)$ has a global $\Aut(M/A)$-invariant extension $q'$ which is inner. Let us show that $q'$ cannot be extended to a strong $\val^3_B$-block extension of $c$. Let $(q'_e)_{e\in E}$ be some strong block extension of $c$. By definition of $O$, there exists $o'$, $k$ such that $o'\leqslant o$ and $c_{o'k}$ is Archimedean over $A$. As a result, $q_{o'}$ is the unique global $\Aut(M/A)$-invariant extension of $\tp(c_{o'}/B)$ by \cref{DOAGDescription1ou2ExtInvRam}, therefore $q'_{o'}=q_{o'}$. As $o'\leqslant o$, by \cref{forceBaseInv}, $q'_{o}$ must be outer, thus $q'_{o}\neq q'$ and we are done.
\par As a result, all the strong block extensions of $c$ extend $(p_i)_{i\in I},(q_o)_{o\in O}$. Then, each such extension $(q'_e)_e$ identifies with a subset of $J$: the set of all $j$ such that $q'_j=r_j$. This concludes the proof, as \cref{forceBaseInv} tells us that the valid choices are exactly the initial segments of $J$.
\end{proof}

\begin{corollary}\label{preuveExistenceStrongExt}
If $c$ is a finite $\val^3_B$-separated family from $M$ such that $\indep{c}{A}{B}{\cut}$, then $c$ admits a strong $\val^3_B$-block extension.
\end{corollary}

\begin{proof}
By \cref{WOrthVal2}, it is enough to show that each $\val^2_B$-block of $c$ has a $\val^3_B$-block extension. The Archimedean $\val^2_B$-blocks of $c$ coincide with its Archimedean $\val^3_B$-blocks, by definition of $\val^3_B$. We find a $\val^3_B$-block extension of those blocks by \cref{conclusionCasArchBaseInv} (the valid choices are built explicitly in \cref{classificationArch}), and we find a block extension of the $B$-ramified $\val^2_B$-blocks by \cref{classificationRam}.
\end{proof}

We just proved the property \ref{existenceBaseInvariance}, but we are actually very close to get a full classification of the space of global $\Aut(M/A)$-invariant extensions of $\tp(c/B)$. We achieve this in the remainder of this section, but these results are not necessary for our main results about forking.

\begin{lemma}\label{finCompletionRam}
Suppose $c=(c_{ij})_{ij}$ is a finite $\val^2_B$-block of $B$-ramified points from $M$ that is $\val^3_B$-separated, whose $\val^3_B$-blocks are the $(c_i)_i=((c_{ij})_j)_i$. Let $(p_i)_i$ be a strong $\val^3_B$-block extension of $c$. Then $q=\bigcup\limits_i p_i\cup\tp(c/B)$ is a complete type in $S(M)$.
\end{lemma}

\begin{proof}
Suppose the indices $i$ are ordered such that $i<k$ if and only if we have $\delta(c_{i}/B)<\delta(c_{k}/B)$. Let $b_{ij}\in\ram(c_{ij}/B)$. We know from \cref{tmpRecollementVal3} that $q$ is consistent. Let $\alpha,\beta\models q$. Then each $\alpha_i,\beta_i$ is a $\val^3_M$-block of $M$-ramified points, and:
$$\Delta(\alpha_{ij}-b_{ij})=\delta(\alpha_i/M),\Delta(\beta_{ij}-b_{ij})=\delta(\beta_i/M)$$
for all $i$ and $j$. Moreover, for $i<k$, and arbitrary indices $jl$, we have $\tp(c/B)\models |x_{kj}-b_{kj}|>n|x_{il}-b_{kl}|$ for all $n>0$. It follows that we have $\delta(\alpha_i/M)<\delta(\alpha_k/M)$ and $\delta(\beta_i/M)<\delta(\beta_k/M)$. We can apply \cref{completionRam} to get $\alpha\equiv_M\beta$, which concludes the proof.
\end{proof}

\begin{remark}\label{espaceFini}
In the setting of \cref{finCompletionRam}, by \cref{classificationRam}, the Stone space of all the global $\Aut(M/A)$-invariant extensions of $\tp(c/B)$ is finite Hausdorff, hence it is discrete.
\end{remark}

\begin{remark}
Let $X$ be a topological space. Suppose $(O_i)_i$ is an open partition of $X$. Then the open subsets of $X$ are exactly the subsets that can be written $\bigcup\limits_i O'_i$, with $O'_i$ an open subset of $O_i$. As a result, $X=\coprod\limits_i O_i$.
\end{remark}

\begin{corollary}\label{coproduit}
Let $X$ be a topological space, and $Y$ a discrete topological space such that we have a continuous map $X\longrightarrow Y$. Then $X$ is the coproduct of the fibers.
\end{corollary}

\begin{lemma}\label{tmpAInvComplet}
Suppose $c=(c_{ij})_{ij}(c'_k)_k$ is a finite, $\val^3_B$-separated, $\val^1_B$-block such that $c'$ is the family of its $B$-Archimedean points, and its $B$-ramified $\val^3_B$-blocks are the $(c_i)_i=((c_{ij})_j)_i$. Let $(p_i)_i$ be a strong $\val^3_B$-block extension of $(c_i)_i$, and suppose all the $p_i$ are outer. Let $G=G(c/B)=G(c'/A)$. For each $i$ and $j$, choose $b_{ij}\in\ram(c_{ij}/B)$. Let $d=(b_{ij})_{ij}(c'_k)_k$. Let $q_k$ be a global $\Aut(M/A)$-invariant extension of $\tp(c'_k/B)$, and let $q$ be a complete global extension of $\tp(c/B)$ which extends $\bigcup\limits_i p_i$ and $\bigcup\limits_k q_k$. If $q(x)\models f(x)-f(d)\not\in {\underline{G}_M}$ for all non-zero $f\in \LC^n(\mathbb{Q})$, then $q$ is $\Aut(M/A)$-invariant.
\end{lemma}

\begin{proof}
Let $f\in \LC^n(\mathbb{Q})$ be non-zero. As is the proof of \cref{tmpClassifArch}, it suffices to show that $q(x)\models f(x)\in X$, with $X$ some $\Aut(M/A)$-invariant cut over $M$. By $\val^3_B$-separatedness, we have $G(f(c)/B)=G$. Write $f=g+h$, with $g\in \LC^{|(c_{ij})_{ij}|}(\mathbb{Q})$, $h\in \LC^{|(c'_k)_k|}(\mathbb{Q})$. By \cref{forceBaseInv}, $(p_i)_i$ is a strong block extension of $(c_{ij})_{ij}$. By \cref{finCompletionRam}, the pushforward of $q$ by the projection over the coordinates $(x_{ij})_{ij}$ is $\Aut(M/A)$-invariant, thus there is $X$ an $\Aut(M/A)$-invariant cut over $M$ such that $q(x)\models g((x_{ij})_{ij})\in X$. If $h=0$, then $q(x)\models f(x)\in X$. Now suppose $h\neq 0$. If $g=0$, then we are done as the proof is the same as in \cref{tmpClassifArch}, so suppose $g\neq 0$. By $\val^3_B$-separatedness, $f(c)$ is Archimedean over $B$, and $G(f(c)/B)=G$. By \cref{coupInvArch}, the two global $\Aut(M/A)$-invariant extensions of $\tp(h((c'_k)_k)/B)$ are the translates by $h((c'_k)_k)$ of the connected components of $G\setminus{\underline{G}_M}$. By \cref{tmpClassifArch}, the pushforward of $q$ by the projection on the coordinates $(x'_k)_k$ is $\Aut(M/A)$-invariant, therefore there is $Y$ one of those two translates so that $q(x)\models h((x'_k)_k)\in Y$. As the $p_i$ are outer, and by $\val^3_B$-separatedness, $X$ is the translate by $g((b_{ij})_{ij})$ of some connected component of the type-definable set $G\setminus{\underline{G}_M}$. Then, we have $q(x)\models f(x)-f(d)\in G\setminus{\underline{G}_M}$, which means that $q(x)\models f(x)\in L\cup R$, with:
$$L= \ct_<\left(\left[g((b_{ij})_{ij})+h((c'_k)_k)\mod G(M)\right]/M\right)$$
$$R= \ct_>\left(\left[g((b_{ij})_{ij})+h((c'_k)_k)\mod G(M)\right]/M\right)$$
By $A$-invariance of $X$ and $Y$, both $L$ and $R$ are $\Aut(M/A)$-invariant, and concluding the proof.
\end{proof}

\begin{lemma}\label{classifVal1Block}
Suppose $c=(c_{ij})_{ij}(c'_k)_k$ is a finite, $\val^3_B$-separated, $\val^1_B$-block such that $c'$ is the family of its $B$-Archimedean points, and its $B$-ramified $\val^3_B$-blocks are the $(c_i)_i=((c_{ij})_j)_i$. Let $(p_i)_i$ be a strong $\val^3_B$-block extension of $(c_i)_i$. Let $q$ be the global $\Aut(M/A)$-invariant extension of $\tp((c_{ij})_{ij}/B)$ which extends $\bigcup\limits_i p_i$ ($q$ exists and is unique, by \cref{finCompletionRam}).
\par Then the following description yields a full classification of the Stone space $F_1$ of global $\Aut(M/A)$-invariant types extending $q\cup\tp(c/B)$:
\par Suppose $c'$ is not empty (or else $q\cup\tp(c/B)$ is already complete). Let $O$ be the set of indices $o$ such that $p_o$ is outer. Let $G=G(c/B)=G(c'/A)$. Define ${N}$, $H$, ${N}'$ as in the statement of \cref{classificationArch}. For each $i$ and $j$, choose $b_{ij}\in\ram(c_{ij}/B)$, and choose $\alpha$ a realization of the projection of $q$ onto the coordinates $(x_o)_{o\in O}$. Let $F_2$ be the closed subspace of $S_{(x_o)_{o\in O},x'}(\{0\})$ corresponding to the types of $\mathbb{Q}-free$ families satisfying $\tp^{{N}'}((\alpha_{oj}-b_{oj}\mod H)_{oj}/\{0\})$. Then the function mapping $\tp^{{N}}((\gamma_{ij})_{i\not\in O, j}\alpha(\beta_k)_k/M)$ to
$$\tp^{{N}'}((\alpha_{oj}-b_{oj}\mod H)_{o\in O, j}(\beta_k-c'_k\mod H)_k/\emptyset)$$
is a well-defined homeomorphism $F_1\longrightarrow F_2$.
\end{lemma}

\begin{proof}
\par Let $d=((b_{oj})_{o\in O, j}(c'_k)_k)$, and let $\pi$ be the projection:
$$(x_i)_i,x'\longmapsto (x_o)_{o\in O},x'$$
\par Let us show first that the map from the statement factors through a homeomorphism $\pi(F_1)\longrightarrow F_2$. The proof is very similar to that of \cref{classificationArch}.

\begin{itemize}[]
\item Let $\beta\in N$ be such that $\tp(\alpha\beta/M)\in\pi(F_1)$. Let $f_1$ in $\LC^{|\alpha|}(\mathbb{Q})$, $f_2$ in $ \LC^{|\beta|}(\mathbb{Q})$ be such that $f_1+f_2\neq 0$. Just as in \cref{classificationArch}, in order to show that the map $g:\ \pi(F_1)\longrightarrow F_2$ given by the statement is well-defined, it suffices to show that we have $f_1(\alpha)+f_2(\beta)-(f_1+f_2)(d)\not\in {\underline{G}_M}$. If $f_2\neq 0$, then $(f_1+f_2)(d)$ is Archimedean over $A$, which concludes the proof by \cref{coupInvArch} just as in \cref{classificationArch}. As all the $p_o$ are outer, and $\alpha$ is $\val^3_B$-separated (as $(c_{oj})_{oj}\equiv_B\alpha$ is), we have $f_1(\alpha)-(f_1+0)(d)\not\in {\underline{G}_M}$, proving that $g$ is well-defined.

\item Let us show that $g$ is injective. Let $p_1$, $p_2\in \pi(F_1)$ be distinct, and $f_1$, $f_2\in \LC(\mathbb{Q})$ be as in the previous item. We use the same case disjunction ($f_2\neq 0$ or $f_2=0$) as in the previous item to show that the pushforward of $p_i$ by $f_1+f_2$ is (as a cut over $M$) a translate by $(f_1+f_2)(d)$ of one of the two connected components of $G\setminus {\underline{G}_M}$. Then, it follows from the same reasoning as in \cref{classificationArch} that $g$ is injective.

\item Let us now prove surjectivity. Let $\beta'$ be a tuple from $N'$ such that $r'=\tp((\alpha_{oj}-b_{oj}\mod H)_{oj}\beta'/\{0\})\in F_2$. Let $\beta$ be a preimage of $\beta'$ by the projection $N\longrightarrow N'$. Let $d_1, d_2\in N$ be such that we have $\Delta(d_1)>\Delta(\beta_i)$, $\Delta(d_k)>\Delta(\alpha_{oj}-b_{oj})$, and $d_2\not\in{\underline{G}_M}$ for all $i, j, k, o$. Let $\sigma\in \Aut\left(N/H+\sum\limits_{oj}\mathbb{Q}\cdot(\alpha_{oj}-b_{oj})\right)$ be such that $\sigma(|d_1|)=|d_2|$, and let $r=\tp(\alpha(c'_k+\sigma(\beta_k))_k/M)$. Then $r$ is $\Aut(M/A)$-invariant by \cref{tmpAInvComplet}, and, just as in \cref{classificationArch}, $r$ is a preimage of $r'$ by $g$.

\item We conclude that $g$ is a homeomorphism $\pi(F_1)\longrightarrow F_2$.
\end{itemize}
\par It remains to show that $F_1\longrightarrow \pi(F_1)$ is homeo. By \cref{finCompletionRam}, recall that $p=\tp((c_i)_{i\not\in O}/B)\cup\bigcup\limits_{i\not\in O} p_i$ is a complete global type. Let $q'$ be in $\pi(F_1)$. Then, if $\gamma_1\models p,\gamma_2\models q'$, one can note that $G(\gamma_2/M)=G$, while $G(\gamma_1/M)={\overline{H(c/B)}^M}<G$ (this is the difference between being inner and outer). As a result, the $\val^1_M$(and thus $\val^2_M$)-values of the elements of $\gamma_1$ (there is only one such value, but we do not care) are different from the values of the elements of $\gamma_2$. By \cref{WOrthVal2}, $p$ and $q$ are weakly orthogonal, which shows that $F_1\longrightarrow\pi(F_1)$ is homeo.
\end{proof}

When put together with \cref{espaceFini} and \cref{coproduit}, we get that the space of global $\Aut(M/A)$-invariant extensions of the type of any finite, $\val^3_B$-separated $\val^1_B$-block $(c_i)_i, c'$ that is cut-independent is the finite coproduct of the topological spaces described in \cref{classifVal1Block}. The continuous map we consider for \cref{coproduit} is the projection $\pi:\ ((x_i)_i, x')\longmapsto (x_i)_i$. Note that, by \ref{classifVal1Block}, any global $\Aut(M/A)$-invariant extension of the type of $(c_i)_i$ over $B$ extends to one of the type of $((c_i)_i, c')$ (as one can always extend the type of some $\mathbb{Q}$-free family to an arbitrary type of a larger $\mathbb{Q}$-free family), so $\pi$ is onto the finite space described in \cref{espaceFini}. Now, given a finite, $\val^3_B$-separated tuple $c$ that is cut-independent, we just described $(S_i)_i$, the Stone spaces of the global $\Aut(M/A)$-invariant extensions of the type of each $\val^1_B$-block of $c$. By \ref{recollementVal1} and \cref{produitOrth}, the Stone space of every global $\Aut(M/A)$-invariant extension of the type of the full tuple $c$ is exactly the finite direct product of the $S_i$. Let us rephrase all of these conclusions (with their hypothesis) in a single statement for good measure:

\begin{remark}\label{DOAGDescriptionFinaleHomeo}
    In DOAG, let $M$ be a sufficiently saturated and strongly homogeneous model, and $C\geqslant A\leqslant B$ be $\mathbb{Q}$-vector subspaces, such that\\ \noindent $\indep{C}{A}{B}{\cut}$. Let $c=(c_{ij})_{ij}$ be some $\val^3_B$-separated family of $C$, with $(c_i)_i=((c_{ij})_j)_i$ its $\val^1_B$-blocks. For each $i$, let $d_i$ be the tuple of $B$-ramified points of $c_i$. Let $F_i$ be the space of all global $\Aut(M/A)$-invariant extensions of $\tp(d_i/B)$, which is finite, and fully described by \cref{classificationRam} and \cref{finCompletionRam}. For each $i$, and for each $q\in F_i$, let $S_q$ be the space (described in \cref{classifVal1Block}) of global $\Aut(M/A)$-invariant extensions of $\tp(c_i/B)$ which extend $q$. We recall that $S_i$ is homeomorphic to some closed set of parameter-free DOAG-types. Then the space $S_i$ of all global $\Aut(M/A)$-invariant extensions of $\tp(c_i/B)$ is of course the finite disjoint union $\bigcup\limits_{q\in F_i}S_q$, which is naturally homeomorphic to the coproduct $\coprod\limits_{q\in F_i}S_q$. Lastly, the whole space of global $\Aut(M/A)$-invariant extensions of $\tp(c/B)$ is naturally homeomorphic to $\prod\limits_i S_i$.
    \par In conclusion, we established a natural and explicit homeomorphism between the space of global $\Aut(M/A)$-invariant extensions of $\tp(c/B)$, and a topological space of the form $\prod\limits_i\coprod\limits_j F_{ij}$, with $(F_{ij})_{ij}$ some closed sets of parameter-free DOAG-types.
\end{remark}

Moreover, as definable bijections are homeomorphisms in $S(M)$, this homeomorphism type is kept when replacing $c$ by an $A$-interdefinable tuple. Note that we show in the next section that every cut-independent tuple is $A$-interdefinable with a $\val^3_B$-separated, cut-independent family. As a result, this will yield a complete description of the Stone space of global $\Aut(M/A)$-invariant extensions of any non-forking type.


\section{Normal forms}\label{sectionFormesNormales}
We built in the last two sections a nice framework to get a fine understanding of a particular class of tuples, the $\val^3_B$-separated families. What remains for us to do in order to prove \cref{thmTechniqueEnonce} is to prove \ref{existenceNormalisation}, which allows us to always reduce to the case where we deal with such a nice family.
\par We adopt \cref{DOAGHypCLibre} in this section.

\par We have to show that $c$ is $A$-interdefinable with a $\val^3_B$-separated family $d$, which is cut-independent by \cref{indepInterdef}. In fact, $d$ will also be $\val^3_A$-separated. This could be useful, as we saw in \cref{classificationRam} that we have to keep track of which points from $d$ are Archimedean or ramified over $A$ in order to have a good understanding of the global $\Aut(M/A)$-invariant extensions of its type. By \cref{operationsInterdef}, $d$ is the image of $c$ by a composition of $A$-translations, and a map from \GL$_n(\mathbb{Q})$. In order to build $d$, we start with a particular basis of $\dcl(Ac)$ (remember that the definable closure of a set is the $\mathbb{Q}$-vector subspace generated by said set in \DOAG), then we go through a process of several steps, where we apply $A$-translations and maps from \GL$_n(\mathbb{Q})$ to this basis. At the end of each step, our current tuple satisfies an additional property from a list of conditions, the conjunction of which is a sufficient condition to be a $\val^3_B$ and $\val^3_A$-separated family.

\begin{remark}\label{rappelArchAB}
Recall that, if $d\in M$ is a $B$-Archimedean singleton such that $\indep{d}{A}{B}{\cut}$, then $d$ must be Archimedean over $A$ (see \cref{ArchBImpliqueArchA}). In particular, $G(d/B)=G(d/A)$.
\end{remark}

The following definition allows us to enumerate our families in a way that helps us to simultaneously consider their $\val^3_B$-blocks and their $\val^3_A$-blocks. The elements of our families are indexed by the subgroups $G(-/B)$, and the eventual Archimedean classes $\delta(-/B)$ that they bring. We use the single quote $'$ to refer to elements that are Archimedean over $A$ and ramified over $B$. We use the tilde $\tilde{\ }$ to refer to elements that are ramified over $A$ (and thus over $B$ if the family is cut-independent from $B$ over $A$). By elimination, the other elements are those that are Archimedean over $B$ and $A$.

\begin{definition}
We say that the family $(d_{Gi})_{Gi}(d'_{G'\delta'i'})_{G'\delta'i'}(\tilde{d}_{\tilde{G}\tilde{\delta}\tilde{i}})_{\tilde{G}\tilde{\delta}\tilde{i}}$ is \textit{under normal enumeration} (with respect to $A$ and $B$) if the following conditions hold:
\begin{itemize}
\item All the $d_{Gi}$ are Archimedean over $B$.
\item All the $d'_{G'\delta'i'}$ are Archimedean over $A$ and ramified over $B$.
\item All the $\tilde{d}_{\tilde{G}\tilde{\delta}\tilde{i}}$ are ramified over $A$.
\item For all $G$, $i$, we have $G(d_{Gi}/B)=G$.
\item For all $G'$, $\delta'$, $i'$, we have $G(d'_{G'\delta'i'}/B)=G'$ and $\delta(d'_{G'\delta'i'}/B)=\delta'$.
\item For all $\tilde{G}$, $\tilde{\delta}$, $\tilde{i}$, we have $G(\tilde{d}_{\tilde{G}\tilde{\delta}\tilde{i}}/B)=\tilde{G}$, and $\delta(\tilde{d}_{\tilde{G}\tilde{\delta}\tilde{i}}/A)=\tilde{\delta}$.
\end{itemize}
\end{definition}

Given a normal enumeration, the ramified $\val^3_B$-blocks of the family are $((d'_{G'\delta'i'})_{i'}(\tilde{d}_{G'\delta'\tilde{i}})_{\tilde{i}})_{G'\delta'}$, its Archimedean $\val^3_B$-blocks are $((d_{Gi})_i)_G$, its ramified $\val^3_A$-blocks are $((\tilde{d}_{\tilde{G}\tilde{\delta}\tilde{i}})_{\tilde{i}})_{\tilde{G}\tilde{\delta}}$, and its Archimedean $\val^3_A$-blocks are $((d_{Gi})_i(d'_{G\delta'i'})_{\delta'i'})_G$. One can note that the presence of the elements of $d'$ is the reason why none of the valuations $\val^3_A$, $\val^3_B$ refines the other.
\par Let us now give the list of properties that we want to satisfy:

\begin{definition}
Let $dd'\tilde{d}$ be a family under normal enumeration with respect to $A$ and $B$. We  say that $dd'\tilde{d}$ is \textit{under normal form} with respect to $A$ and $B$ if it is a lift of a $\mathbb{Q}$-free family from $\faktor{M}{B}$, and the following conditions hold:

\begin{enumerate}[label=(P\arabic*), ref=(P\arabic*)]
\item \label{P1} For all $\tilde{G}$ and $\tilde{\delta}$, $(\tilde{d}_{\tilde{G}\tilde{\delta}\tilde{i}})_{\tilde{i}}$ is a lift of a $\mathbb{Q}$-free family from $\faktor{M_{\leqslant\tilde{\delta}}}{M_{<\tilde{\delta}}}$.
\item \label{P2} Non-trivial linear combinations of $dd'$ are Archimedean over $A$.
\item \label{P3} For all $G$, non-trivial linear combinations $e$ of $(d_{Gi})_{i}(d'_{G\delta'i'})_{\delta'i'}$ satisfy $G(e/A)=G$.
\item \label{P4} For all $G$, any non-trivial linear combination $c$ of the tuple $(d_{Gi})_i$ is Archimedean over $B$ (which implies $G(c/B)=G$ if the family is cut-independent, by \ref{P3} and \cref{rappelArchAB}).
\item \label{P5} For all $G'$ and $\delta'$, non-trivial linear combinations $e'$ of the tuple $(d'_{G'\delta'i'})_{i'}(\tilde{d}_{G'\delta'\tilde{i}})_{\tilde{i}}$ are ramified over $B$, and satisfy $\delta(e'/B)=\delta'$.
\end{enumerate}
\end{definition}

\begin{remark}
Condition \ref{P1} implies that each ramified $\val^3_A$-block is $\val^3_A$-separated. The conjunction of \ref{P2} and \ref{P3} is equivalent to each Archimedean $\val^3_A$-block being $\val^3_A$-separated. Likewise, \ref{P4} (resp. \ref{P5}) is equivalent to each Archimedean (resp. ramified) $\val^3_B$-block being $\val^3_B$-separated. As a result, if we manage to prove that $c$ is $A$-interdefinable with a family under normal form, then this family would be simultaneously $\val^3_A$-separated and $\val^3_B$-separated, and \ref{existenceNormalisation} would follow.
\end{remark}

\begin{remark}
From the definition of the normal form, if $dd'\tilde{d}$ is under normal form with respect to $A$ and $B$, then one can note that any subfamily of $dd'\tilde{d}$ is also under normal form with respect to $A$ and $B$.
\end{remark}

\begin{definition}
Define $C$ to be the $\mathbb{Q}$-vector space generated by $Ac$. We  define a \textit{basis} (resp. \textit{free family}) \textit{of $C$ over $A$} as a lift of a basis (resp. free family) of $\faktor{C}{A}$.
\end{definition}

\begin{remark}
Remember that $dim\left(\faktor{C}{A}\right)$ is finite. It is easy to see that any concatenation of lifts of bases of $\faktor{C_{\leqslant\tilde{\delta}}}{C_{<\tilde{\delta}}}$, with $\tilde{\delta}\in\Delta(C)\setminus\Delta(A)$, is a free family over $A$, hence a finite family. In particular, $\Delta(C)\setminus\Delta(A)$ is finite.
\end{remark}

\begin{lemma}\label{premierP1}
For each $\tilde{\delta}\in\Delta(C)\setminus\Delta(A)$, let $\tilde{d}_{\tilde{\delta}}$ be a lift of a basis of $\faktor{C_{\leqslant\tilde{\delta}}}{C_{<\tilde{\delta}}}$, and let $\tilde{d}$ be the concatenation of all the $\tilde{d}_{\tilde{\delta}}$. Then, for all $e$ in $C\setminus \dcl(A\tilde{d})$, there exists $\tilde{e}\in \dcl(A\tilde{d})$ such that $e+\tilde{e}$ is Archimedean over $A$.
\par Moreover, if $e$ is ramified over $A$, then we can choose $\tilde{e}$ such that there exists $\tilde{\delta}\in\Delta(C)\setminus\Delta(A)$ for which $\Delta(e+\tilde{e})<\tilde{\delta}\leqslant\Delta(e)$.
\end{lemma}

\begin{proof}
Let us build $\tilde{e}$ by induction. Let $\tilde{e}_0=0\in \dcl(A\tilde{d})$. If $e+\tilde{e}_n$ is not Archimedean over $A$, then let $a\in\ram(e+\tilde{e}_n/A)$, and $\tilde{\delta}_n=\delta(e+\tilde{e}_n/A)$. By definition of $\tilde{d}$, there exists $v$ a linear combination of $ \tilde{d}_{\tilde{\delta}_n}$ for which we have $\Delta(e+\tilde{e}_n-a-v)<\tilde{\delta}_n$, so let us set $ \tilde{e}_{n+1}=\tilde{e}_n-a-v$.
\par By finiteness of $\Delta(C)\setminus\Delta(A)$, the sequence of the $\tilde{\delta}_n$, and hence that of the $\tilde{e}_n$, must eventually stop. The ultimate value $\tilde{e}$ of the $\tilde{e}_n$ witnesses the first part of the statement. If $e$ is ramified over $A$, then $\tilde{e}\neq 0= \tilde{e}_0$, so $\tilde{\delta}_0$ exists and witnesses the second part of the statement.
\end{proof}

\begin{lemma}\label{deuxiemeP1}
For each $\tilde{\delta}\in\Delta(C)\setminus\Delta(A)$, let $\tilde{d}_{\tilde{\delta}}$ be a lift of a basis of $\faktor{C_{\leqslant\tilde{\delta}}}{C_{<\tilde{\delta}}}$, and let $\tilde{d}$ be the concatenation of all the $\tilde{d}_{\tilde{\delta}}$. Let $(v_n)_n$ be a sequence such that $(\tilde{d},v_n)$ is a basis of $C$ over $A$ for each $n$. Suppose that, for all $n$, there exists $v_1'$ a term from $v_n$, $v_2'$ a term from $v_{n+1}$, and $\delta\in\Delta(C)\setminus\Delta(A)$, such that $\Delta(v_2')<\delta\leqslant\Delta(v_1')$, and $v_{n+1}$ is obtained from $v_n$ by replacing $v_1'$ by $v_2'$.
\par Then the sequence $(v_n)_n$ must eventually stop.
\end{lemma}

\begin{proof}
Let $N=|\Delta(C)\setminus\Delta(A)|<\omega$, $(\delta_n)_{n<N}$ the strictly increasing enumeration of $\Delta(C)\setminus\Delta(A)$, and $\mathcal{I}$ be the set of intervals: $$\left\lbrace \left]-\infty, \delta_0\right[, \left[\delta_0, \delta_1\right[, \left[\delta_1, \delta_2\right[, \ldots \left[\delta_{N-2}, \delta_{N-1}\right[, \left[\delta_{N-1}, +\infty\right[\right\rbrace$$
\par Define on $\mathcal{I}$ the only total order that makes the canonical projection $f:\ \Delta(C)\longrightarrow\mathcal{I}$ an order-preserving map. With respect to that order, let $g$ be the unique order-isomorphism $\mathcal{I}\longrightarrow \{0, 1\ldots  N\}$. For each $i$ between $0$ and $N$, let $F(n, i)$ be the number of terms $v$ from $v_n$ such that $g(f(\Delta(v)))=i$. Finally, let $\rho$ be the following function:
$$
n\longmapsto F(n, N)\times\omega^N + F(n, N-1)\times\omega^{N-1}+\ldots +F(n, 1)\times\omega+F(n, 0)
$$
then $\rho$ has its values in the well-ordered set $\omega^{N+1}$. By hypothesis, for each $n$, there exists $m<m'$ such that $F(n, m')=F(n+1, m')+1$, $F(n, m)+1=F(n+1, m)$, and $F(n, i)=F(n+1, i)$ for every $i\neq m$, $m'$. As a result, $\rho$ is strictly decreasing, so the well-ordering forces the sequence $(v_n)_n$ to stop eventually.
\end{proof}

\begin{lemma}\label{DOAGP2}
For each $\tilde{\delta}\in\Delta(C)\setminus\Delta(A)$, let $u_{\tilde{\delta}}$ be a lift of a basis of $\faktor{C_{\leqslant\tilde{\delta}}}{C_{<\tilde{\delta}}}$, and let $u$ be the concatenation of all the $u_{\tilde{\delta}}$. Then there exists $dd'\tilde{d}$ a basis of $C$ over $A$, under normal enumeration with respect to $A$ and $B$, for which \ref{P2} holds, and $\tilde{d}=u$ (so \ref{P1} holds as well).
\end{lemma}

\begin{proof}
Let $w_0$ be some basis of $C$ over $A$ containing $u$. Note that $w_n=(u,v_n)$ does not witness the statement if and only if there exists a non-trivial linear combination $v'$ of $v_n$ that is ramified over $A$. In that case, let $v$ be some term of $v_n$ that appears in $v'$, chosen such that $\Delta(v)$ is maximal. By valuation inequality, we must have $\Delta(v)\geqslant\Delta(v')$. Apply \cref{premierP1} to find $u'\in\dcl(Au)$ and $\delta\in\Delta(C)\setminus\Delta(A)$ for which $\Delta(v'+u')<\delta\leqslant\Delta(v')\leqslant\Delta(v)$. Replace $v$ by $v'+u'$ in $w_n$, let $w_{n+1}$ be the new family obtained. Now, $w_{n+1}$ is also a basis of $C$ over $A$ containing $u$.
\par The sequence $(v_n)_n$ witnesses the hypothesis of \cref{deuxiemeP1}, so it must eventually stop, and its last element witnesses the statement.
\end{proof}

\begin{remark}\label{rqueConservationP2}
If $w=(u, v)$ witnesses \cref{DOAGP2}, and $M$ is an invertible matrix of size $|v|$ with coefficients in $\mathbb{Q}$, then $(u, (Mv))$ also witnesses \cref{DOAGP2}.
\end{remark}

\begin{lemma}
Let $dd'\tilde{d}=(v, \tilde{d})$ be a free family over $A$ that witnesses \ref{P1} and \ref{P2}, and $N=|v|$. Then there exists $M\in \GL_N(\mathbb{Q})$ such that $((Mv)$, $\tilde{d})$ witnesses \ref{P3} (and thus \ref{P1}, \ref{P2} by \cref{rqueConservationP2}).
\end{lemma}

\begin{proof}
This is by induction over $N$. \ref{P3} trivially holds if $N=0, 1$.
\par For every $G$, let $v_G$ be the $\val^1_A$-block of $v$ of value $G$. Let $G$ be maximal such that $v_G$ is not $\val^3_A$-separated. As long as $v_G$ is not $\val^3_A$-separated, there must exist $e$ a non-trivial linear combination of $v_G$ such that $G(e/A)<G$, then replace some term of $v_G$ that appears in $e$ by $e$ (making that replacement corresponds to replacing $v$ by $Mv$ for some $M\in \GL_N(\mathbb{Q})$). As we have $G(e/A)<G$, these replacements always decrease the size of $v_G$, and leave $v_{>G}$ untouched as the value of $e$ is smaller, so the $\val^1_A$-blocks of value larger than $G$ remain separated, thus we keep the maximality of $G$. As a result, these replacements eventually stop, and they stop before we removed all the terms from $v_G$, because if there is only one term left in $v_G$, then $v_G$ is clearly separated.
\par By induction hypothesis, we can replace $v_{<G}$ by $Mv_{<G}$ such that the family $((Mv_{<G}),\tilde{d})$ witnesses \ref{P3}. It is now clear that $((Mv_{<G}),v_{\geqslant G},\tilde{d})$ witnesses \ref{P3}. The only changes we made to the original family $v$ are operations from $\GL_N(\mathbb{Q})$, concluding the proof.
\end{proof}

\begin{remark}\label{DOAGConservationP3}
Let $w=dd'\tilde{d}=(v, \tilde{d})$ be a free family over $A$ that witnesses \ref{P1}, \ref{P2}, \ref{P3}. Let $v_G$ be the $\val^1_A$-block of $v$ of value $G$. If we replace $v_G$ by $Mv_G$ in $w$ for some invertible matrix $M$, then the new family $w'$ still witnesses \ref{P1}, \ref{P2} and \ref{P3}.
\end{remark}

\begin{lemma}
There exists a basis of $C$ over $A$ that witnesses \ref{P1}, \ref{P2}, \ref{P3}, \ref{P4}.
\end{lemma}

\begin{proof}
Let $w=dd'\tilde{d}$ be a basis of $C$ over $A$ that witnesses \ref{P1}, \ref{P2} and \ref{P3}. For each $G$, as long as there is a non-trivial linear combination $e$ of the $(d_{Gi})_i$ that is not Archimedean over $B$, then replace some term of $(d_{Gi})_i$ that appears in $e$ by $e$. Each of these replacements removes a term from $(d_{Gi})_i$ and adds a new term to $(d'_{G\delta'i'})_{\delta'i'}$. This eventually stops as the number of terms of $(d_{Gi})_i$ is finite and strictly decreases at each step.
\par By \cref{DOAGConservationP3}, the new family $w'$ that we obtain after all these replacements still witnesses \ref{P1}, \ref{P2}, \ref{P3}. We perform these replacements for each $G$, and the new family obtained clearly witnesses \ref{P4}.
\end{proof}

\begin{remark}\label{P5OpEl}
Let $w=dd'\tilde{d}$ be a free family over $A$ that witnesses \ref{P1}, \ref{P2}, \ref{P3}, \ref{P4}, and $G'$ some index. Then, if we replace $(d'_{G'\delta'i'})_{\delta'i'}$ by $M(d'_{G'\delta'i'})_{\delta'i'}$ in $w$ for some invertible matrix $M$, it is clear that the new family still witnesses \ref{P1}, \ref{P2}, \ref{P3}, \ref{P4}.
\end{remark}

\begin{lemma}\label{P5Translation}
Let $w=dd'\tilde{d}$ be a basis of $C$ over $A$ that witnesses \ref{P1}, \ref{P2}, \ref{P3}, \ref{P4}. Let $d'_{G'\delta'k'}$ be a term from $d'$, and $\tilde{e}$ be a linear combination of the $(\tilde{d}_{G'\tilde{\delta}\tilde{i}})_{\tilde{\delta}\tilde{i}}$. If we replace $d'_{G'\delta'k'}$ by $e=d'_{G'\delta'k'}+\tilde{e}$ in $w$, then the new family obtained is still  a basis of $C$ over $A$ that witnesses \ref{P1}, \ref{P2}, \ref{P3}, \ref{P4}.
\end{lemma}

\begin{proof}
First of all, as $\val^2_A(d'_{G'\delta'k'})>\val^2_A(\tilde{e})$, $e$ is Archimedean over $A$, and $G(e/A)=G'$. As $\val^2_B(e)\leqslant \max(\val^2_B(d'_{G'\delta'k'}), \val^2_B(\tilde{e}))$, $e$ cannot be Archimedean over $B$. This implies that the $A$-ramified and $B$-Archimedean blocks of our new family are left untouched, hence we keep the properties \ref{P1}, \ref{P4}. Let $u'$ be the new block of $A$-Archimedean and $B$-ramified points in the normal enumeration. Then $e$ lies in the block $u'_{G'}$. In order to show that \ref{P2} and \ref{P3} hold in the new family $du'\tilde{d}$, we just have to show that the $\val^3_A$-block $d_{G'}u_{G'}'$ is $\val^3_A$-separated. Indeed, it would follow that the whole $du'$ is $\val^3_A$-separated, which is clearly equivalent to \ref{P2} and \ref{P3} being true in $du'\tilde{d}$.
\par Any non-trivial linear combination of $d_{G'}u'_{G'}$ can be written $u+\lambda\cdot\tilde{e}$, with $u$ a non-trivial linear combination of $d_{G'}d'_{G'}$, and $\lambda\in\mathbb{Q}$. As \ref{P2} and \ref{P3} hold in $dd'\tilde{d}$, we have $G(u/A)=G'$, and $u$ is Archimedean over $A$. Then, we have $\val^3_A(u)>\val^3_A(\tilde{e})$, hence $\val^3_A(u+\lambda\cdot\tilde{e})=\val^3_A(u)$. This value equals that of the corresponding block, $\val^3_A(d_{G'}d'_{G'})$, as \ref{P2} and \ref{P3} hold in $dd'\tilde{d}$, thus it equals $\val^3_A(d_{G'}u'_{G'})$, concluding the proof.
\end{proof}

\begin{lemma}\label{P5Final}
Let $w=dd'\tilde{d}$ be a basis of $C$ over $A$ that witnesses \ref{P1}, \ref{P2}, \ref{P3}, \ref{P4}, and not \ref{P5}. Then there exists $G'$, $\delta'$, and two $\mathbb{Q}$-linear maps $f, g\in \LC(\mathbb{Q})$, such that $v=f((d'_{G'\delta'i'})_{i'})+g((\tilde{d}_{G'\delta'\tilde{i}})_{\tilde{i}})$ is ramified over $B$, and $\delta(v/B)<\delta'$.
\par Moreover, if $d'_{G'\delta'k'}$ appears in $v$, then replacing it by $v$ in $w$ gives a new family $w'$ which is still a basis of $C$ over $A$ that witnesses \ref{P1}, \ref{P2}, \ref{P3}, \ref{P4}.
\end{lemma}

\begin{proof}
By definition of \ref{P5}, there exists $G'$, $\delta'$, and $\mathbb{Q}$-linear maps $f, g$ in $\LC(\mathbb{Q})$ such that, either $v=f((d'_{G'\delta'i'})_{i'})+g((\tilde{d}_{G'\delta'\tilde{i}})_{\tilde{i}})$ is not ramified over $B$, or $\delta(v/B)\neq\delta'$.
\par Note that $f$ must be non-zero, otherwise the above hypothesis fails by \ref{P1}. By \ref{P1}, \ref{P2}, \ref{P3}, $dd'\tilde{d}$ is $\val^3_A$-separated. As $f$ is non-zero, $\val^3_A(v)=\val^3_A(d'_{G'})$, thus $v$ is Archimedean over $A$, and $G(v/A)=G'=G(v/B)$. For all $i'$, $\tilde{i}$, we have $d'_{G'\delta'i'},\tilde{d}_{G'\delta'\tilde{i}}\in B+ G'$ (by the definition of being ramified over $B$), so $v\in B+ G'=B+ G(v/B)$, therefore $v$ must be ramified over $B$, thus by hypothesis we have $\delta(v/B)\neq \delta'$. By valuation inequality (for $\val^3_B$), we must have $\delta(v/B)<\delta'$, and we get the first part of the statement.
\par Suppose that $d'_{G'\delta'k'}$ appears in $v$. Then, replacing $d'_{G'\delta'k'}$ by $v'=f((d'_{G'\delta'i'})_{i'})$ is just a replacement of $(d'_{G'\gamma'i'})_{\gamma'i'}$ by $M(d'_{G'\gamma'i'})_{\gamma'i'}$ for some invertible matrix $M$. By \cref{P5OpEl}, the new family is still a basis of $C$ over $A$ that witnesses \ref{P1}, \ref{P2}, \ref{P3}, \ref{P4}. By \ref{P2} and \ref{P3}, $v'$ is Archimedean over $A$, and $G(v'/A)=G'$. Moreover, as all the $(d'_{G'\delta'i'})_{i'}$ are in $B+ G'$, $v'$ must be ramified over $B$. Then, $w'$ is obtained by replacing $v'$ by $v$ in (a normal enumeration of) $w''$. This is exactly the operation described in \cref{P5Translation}, concluding the proof.
\end{proof}

\begin{theorem}\label{existenceFormeNormale}
There exists $w=dd'\tilde{d}$ a basis of $C$ over $A$ that  is under normal form with respect to $A$ and $B$.
\end{theorem}

\begin{proof}
Let $w_0=dd'^0\tilde{d}$ be a basis of $C$ over $A$ that witnesses \ref{P1}, \ref{P2}, \ref{P3}, \ref{P4}.
\par Suppose $w_n=dd'^n\tilde{d}$ does not witness \ref{P5}. Then the hypothesis of \cref{P5Final} hold in $w_n$, so we do the replacement from this lemma. Let $w_{n+1}$ be the new family obtained.
\par To prove the theorem, we have to prove that the sequence $(w_n)_n$ must stop.
\par To do this, we once again build a strictly decreasing map with values in $\omega^\omega$. Let $N=|\Delta(C+B)\setminus\Delta(B)|$, and $(\delta_i)_{i<N}$ be the strictly increasing enumeration of $\Delta(CB)\setminus\Delta(B)$. Define $F(n, i)$ as the number of terms $c$ from $d'^n$ for which $\delta(c/B)=\delta_i$. We define $g$ as the map:
$$
n\longmapsto F(n, N)\times\omega^N + F(n, N-1)\times\omega^{N-1}+\ldots +F(n, 1)\times\omega+F(n, 0)
$$
\par For each $n$, the family $w_{n+1}$ is obtained from $w_n$ by replacing some term $v$ from $d'^n$ by a vector $v'$ from $d'^{n+1}$ for which $\delta(v'/B)<\delta(v/B)$, so $g$ is strictly decreasing. This concludes the proof.
\end{proof}

With that, we proved \cref{thmTechniqueEnonce}, the main technical theorem of this paper stating that $\indep{}{}{}{\cut}=\indep{}{}{}{\inv}$ in DOAG. The reasoning is explained in the end of \cref{sectOAG}, and the different steps of the proof are carried out in \cref{WOrthVal1}, \cref{WOrthVal2}, \cref{tmpRecollementVal3}, \cref{preuveExistenceStrongExt} and \cref{existenceFormeNormale}.


\section{Quantifier-free types in the Presburger language}

\begin{definition}
The \textit{Presburger language} is defined as :
$$\mathcal{L}_P=\left\lbrace +, -, <, 0, \boldone, (\mathfrak{d}_{l^n})_{l\textup{ prime}, n>0}\right\rbrace$$
\par Given an ordered Abelian group $G$, we see $G$ as an $\mathcal{L}_P$-structure by setting $\mathfrak{d}_{l^n}(x)$ as the predicate $\exists y,\ l^n y=x$, and setting $\boldone=\min(]0, +\infty[)$ if $G$ is discrete, else we interpret $\boldone$ as the only $\emptyset$-definable choice, $\boldone=0$ (this is the standard interpretation of the language). A \textit{special} subgroup of $G$ is a pure subgroup (that is a relatively divisible subgroup) $A\leqslant G$ containing $\boldone$. Special subgroups will correspond to definably closed sets in the groups we are interested in (see \cref{specialImpliqueDefClos}).
\par Define $S^m_{\qf}(A)$ to be the Stone space of quantifier-free types over $A$ in $m$ variables. Given an $m$-tuple $c$, we write $\tp_l(c/A)$ as the partial type generated by the set of quantifier-free formulas satisfied by $c$, with parameters in $A$, which only involve the predicates $(\mathfrak{d}_{l^n})_{n>0}$ (equality does not appear in these formulas either). We write $S^m_l(A)$ for the Stone space of all such partial types, and we call its elements the \textit{$l$-types}. We define similarly $\tp_<(c/A)$, $S^m_<(A)$ for the quantifier-free formulas which only involve the predicates $=$, $<$.
\par For ease of notation, define the set of indices $J$ as the union of the set of primes with $\{<\}$. For every $j\in J$, we have natural restriction maps $S^m_{\qf}(A)\longrightarrow S^m_j(A)$, and they are all surjective. We  write $(\pi_j)_j$ for those maps.
\end{definition}

\begin{assumptions}\label{ROAGMMonstreA}
We fix an ordered Abelian group $M$, and a special subgroup $A\leqslant M$, such that ${M}$ is $|A|^+$-saturated and strongly $|A|^+$-homogeneous.
\end{assumptions}

\begin{lemma}[Standard variant of the Chinese remainder theorem]\label{thmChinois}
Let $L$ be a finite set of primes, and $N>0$. Then for all $(a_l)_l\in M^L$, there exists $b\in M$ such that $M\models\mathfrak{d}_{l^N}(a_l-b)$ for all $l\in L$.
\end{lemma}

\begin{corollary}\label{surjPremier}
The natural map $S^m_{\qf}(A)\longrightarrow \prod\limits_{l\textup{ prime}} S^m_l(A)$ is surjective.
\end{corollary}

\begin{proof}
For each $l$, let $a_l=(a_{i, l})_i\in M^m$. By \cref{thmChinois} and compactness, for each $i$, there must exist $b_i\in M$ such that $\mathfrak{d}_{l^N}(b_i-a_{i, l})$ for all $l$, $N$. In particular, $\tp_l(b/A)=\tp_l(a_l/A)$ for all $l$, which concludes the proof.
\end{proof}

\begin{remark}
The map $\pi:\ S^m_{\qf}(A)\longrightarrow\prod\limits_{j\in J}S^m_j(A)$ is obviously injective. Note that $S^m_<(A)$ is not a factor of the product of \cref{thmChinois}, thus $\pi$ may not be surjective.
\end{remark}

\section{Basic properties of regular groups}
The class of regular ordered Abelian groups (\ROAG) has several equivalent definitions. The ones that are necessary for us are quantifier elimination in $\mathcal{L}_P$, and some compatibility conditions between the $<$-types and the $l$-types. The definitions involving Archimedean groups and definable convex subgroups will be relevant, because they give a motivation as to why we are interested in \ROAG.
\par The equivalence between these different definitions is folklore. In this section, we prove the easy implications between these equivalent definitions, and we give references for the harder implications.

\begin{definition}\label{defROAG}
Let $G$ be an ordered Abelian group. We  say that $G$ is \textit{regular} if, for every positive integer $n$, every interval of $G$ that contains at least $n$ elements intersects $nG$.
\par For any ordered Abelian group $G$, we recall that $\div(G)$ refers to the divisible closure of $G$, which we  see as an ordered group to which the order on $G$ naturally extends.
\end{definition}

\begin{lemma}\label{bijectionQLibre}
Let $G\models$ \ROAG, and $A$ a special subgroup of $G$. Let $F$ be the closed subspace of $S^m_{\qf}(A)$ of types of $m$-tuples that are $\mathbb{Q}$-free over $A$ (i.e. lifts in $G^m$ of $\mathbb{Q}$-free families from $\faktor{\div(G)}{\div(A)}$). Then the map $F\longrightarrow\pi_<(F)\times\prod_{l\textup{ prime}} S^m_l(A)$ is a homeomorphism.
\end{lemma}
\noindent In particular, $\pi_l(F)=S^m_l(A)$ for every prime $l$.
\begin{proof}
Two elements having the same type must have the same $j$-type for all $j\in J$, therefore the map is injective, and it is clearly continuous. Let us prove that it is surjective.
\par Let $p=(p_j)_j$ be in the direct product. Let $c=(c_i)_i$ be a realization of $p_<$ which is $\mathbb{Q}$-free over $A$. By \cref{surjPremier}, let $c'$ be some tuple which realizes simultaneously all the $(p_l)$ for $l$ prime ($c'$ might not be $\mathbb{Q}$-free over $A$, but it does not matter). Let us show that there exist $d=(d_i)_{i<m}$ such that $d_i-c'_i\in\bigcap\limits_{N>0} NG$ for all $i$, and $d\models p_<$ (in particular, $d$ will be $\mathbb{Q}$-free over $A$). The existence of $d$ will be established by a disjunction of two cases.
\par Suppose $G$ is discrete. Let $N>0$. Then, by regularity of $G$ applied to the interval $[c_i-c_i', c_i-c_i'+N\cdot \boldone]$, there must exist an integer $k_i$ between $0$ and $N$ such that $c_i-c_i'+k_i\cdot \boldone\in NG$. Let us show that $d'=(d_i')_i=(c_i+k_i\cdot \boldone)_i\models p_<$. If not, then there must exist some atomic formula with predicate $=$ or $<$ which is satisfied by $c$ and not $d'$, or vice-versa. Without loss, there must exist $a\in A$, and some non-zero $f\in \LC^m(\mathbb{Z})$ such that $f(c)<a\leqslant f((c_i+k_i\cdot \boldone)_i)$. By subtracting $f((k_i\cdot \boldone)_i)$ to the second inequality, $f(c)$ belongs to the interval $[a-f((k_i\cdot \boldone)_i), a]$, which is included in $A$ as $A$ is special, contradicting $\mathbb{Q}$-freeness of $c$ over $A$. It follows that $d'\models p_<$, and we conclude by compactness that $d$ exists.
\par Now, suppose $G$ is dense. Let $D$ be the special subgroup generated by $A$, and all the $(c_i)_i$. For every positive element $a$ of $D$, the open interval $]c_i-c_i', c_i-c_i'+a[$ is infinite, thus contains an element of $NG$ by regularity, for any $N>0$. By compactness, there exists, for each $i$, an element $d_i'$ such that $d_i'\in ]c_i-c_i', c_i-c_i'+a[$ for every positive $a\in D$, and $d'_i\in\bigcap\limits_{N>0} NG$. Let $d_i=d_i'+c_i'$. Notice that $d_i-c_i$ is a positive element which is infinitesimal with respect to $D$. It remains to show that $d\models p_<$. Let $f\in \LC^m(\mathbb{Z})$ be non-zero, and $a\in A$, such that $f(c)>a$. As each $d_i-c_i$ is infinitesimal with respect to $D$, this is also the case for $f(d-c)=f(d)-f(c)$, thus $|f(d)-f(c)|<f(c)-a$, thus $f(c)-f(d)<f(c)-f(a)$, and we conclude that $f(d)>f(a)$, proving that $d\models p_<$.
\par In both cases, we have that $d\models p_<$. As $d_i-c_i'\in l^NG$ for every $l$ prime and $N>0$, we have $d\models p_l$ for every $l$, which concludes the proof.
\end{proof}

\begin{remark}\label{clotureDefROAG}
Note that, in this proof, $d_i$ can always be chosen so that $d_i>c_i$ for each $i$ (in the discrete case, replace $k_i$ by $k_i+N>0$). In particular, if we have $c=c'$ in the construction of the proof, then the tuple $d$ that we build is distinct from $c$, and re-iterating this operation generates infinitely many pairwise-distinct tuples having the same quantifier-free type as $d$. It follows that any consistent quantifier-free type over $A$ whose realizations are $\mathbb{Q}$-free over $A$ has infinitely many realizations.
\end{remark}

\begin{definition}\label{defSpine}
\par Let $G$ be an ordered Abelian group, $n<\omega$ and $g\in G$. Define $H_n(g)$ to be the largest convex subgroup of $G$ for which we have $(g+H_n(g))\cap nG=\emptyset$ (by convention, if $g\in nG$, then $H_n(g)=\{0\}$). This convex subgroup is definable:
$$
H_n(g)=\{0\}\text{ or }\left\lbrace x\in G|\forall y\in G,\ \left(|y|\leqslant n|x|\Longrightarrow g+y\not\in nG\right)\right\rbrace
$$
for if $x\in G$, $z\in G$ satisfy $\Delta(x)=\Delta(z)$ ($\Delta$ is defined in \cref{defArchVal}) and $z+g\in nG$, then $x\not\in H_n(g)$ is witnessed by choosing $y=z\pm mn|x|$, with $m$ the least natural integer for which $|z|-mn|x|\leqslant n|x|$.
\end{definition}

In the literature, the set of the $H_n(g)$ for all $g\in G$ is called the $n$\textit{-spine} of $G$. This is a definable family of definable convex subgroups. One of the most general “complexity class" of ordered Abelian groups that is still considered rather “nice" is the class of ordered Abelian groups with finite spines. This class contains in particular the dp-finite ordered Abelian groups. For reference, section 2 of \cite{Farre} gives nice characterizations and a quantifier elimination result for this class.

\begin{proposition}\label{equROAG}
Let $G$ be an ordered Abelian group. Then the following conditions are equivalent:
\begin{enumerate}
\item $G$ does not have any proper non-trivial definable convex subgroup.
\item For all $g\in G$, for all $n<\omega$ $H_n(g)=\{0\}$.
\item $G$ is regular.
\item The theory of $G$ eliminates quantifiers in the Presburger language.
\item There exists an Archimedean ordered Abelian group elementarily equivalent to $G$.
\end{enumerate}
\end{proposition}

\begin{proof}
The directions that we show in order to establish the equivalence are $1\Longrightarrow 2\Longrightarrow 3$, $3\Longrightarrow 4\Longrightarrow 1$ and $3\Longrightarrow 5\Longrightarrow 1$.
\par The implications $1\Longrightarrow 2$ and $5\Longrightarrow 1$ are trivial, $3\Longrightarrow 4$ is due to Weispfenning (\cite{QEROAG}, Theorems 2.3 and 2.6), and Presburger in the discrete case, and $3\Longrightarrow 5$ follows from (\cite{Zakon}, Theorem 2.5) and (\cite{RobinsonZak}, Theorem 4.7).
\par Let us show $2\Longrightarrow 3$ by contraposition. We have an interval $I$ of $G$ and an integer $n$ such that $|I|\geqslant n$ and $I\cap nG=\emptyset$. If $|I|<2n+1$, then $G$ is discrete, and the points from $I+n\cdot \boldone$ have the same cosets $\mod nG$ as those of $I$, so we can assume $|I|\geqslant 2n+1$ by replacing $I$ by $I\cup (I+n\cdot \boldone)\cup (I+2n\cdot \boldone)$. Let $f$ be a strictly increasing map $\{0,\ldots  2n\}\longrightarrow I$, and $h=\min\left\lbrace f(i+1)-f(i)|0\leqslant i\leqslant 2n-1\right\rbrace$. As $f$ is strictly increasing, we have $h\neq 0$. As $h$ is minimal, we can sum $n$ many inequalities to get $n\cdot h\leqslant\sum\limits_{i=0}^{n-1} (f(i+1)-f(i))=f(n)-f(0)$. With the same argument, $n\cdot h\leqslant f(2n)-f(n)$, thus $\left[ f(n)-n\cdot h, f(n)+n\cdot h\right]\subset I$, which is disjoint from $nG$. It follows that $h\in H_n(f(n))\setminus \{0\}$, so $2$ fails.
\par Let us prove $4\Longrightarrow 1$ by contraposition. Let $G$ be an ordered Abelian group with a proper non-trivial definable convex subgroup $H$. Let $A$ be a definably closed subset of $G$ such that $H$ is $A$-definable. Let $A_1=A_{\geqslant 0}\cap H$, $A_2=(A_{>0}\setminus H)\cup\{+\infty\}$. Let $X_<=\left(\bigcap\limits_{a\in A_1, b\in A_2}]a, b[\right)$, an $A$-type definable set which is non-empty by compactness. For each prime $l$, let $X_l=\bigcap\limits_N l^NG$, which corresponds to $\tp_l(0/A)$. Let us show that the partial type defined as $q(x):\ x\in X_<\cap\bigcap\limits_l X_l$ is consistent with $H$ and $\neg H$. Let $a\in A_1$, $b\in A_2$, $N>0$. By compactness, it suffices to show that $Y=]a, b[\cap NG$ intersects $H$ and $\neg H$, which is witnessed by $N\cdot a\in Y\cap H$, and $N\cdot(b-a)\in Y\setminus H$. Now let $h\in H, g\not\in H$ be two realizations of $q$. Then $\qftp(h/A)=q=\qftp(g/A)$, but $\tp(h/A)\neq\tp(g/A)$, and the theory of $G$ does not eliminate quantifiers in $\mathcal{L}_P$.\qedhere

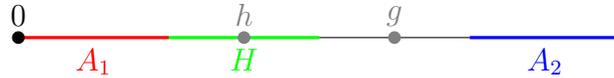
\begin{center}
\begin{tikzpicture}
\draw[red, very thick] (0, 0) -- (2, 0);
\draw[green, very thick] (2, 0) -- (4, 0);
\draw[black] (4, 0) -- (6, 0);
\draw[blue, very thick] (6, 0) -- (8, 0);
\filldraw[black, thick] (0, 0) circle (2pt) node[anchor=south]{$0$};
\filldraw[gray, thick] (3, 0) circle (2pt) node[anchor=south]{$h$};
\filldraw[gray, thick] (5, 0) circle (2pt) node[anchor=south]{$g$};
\draw[red] (1, 0) node[anchor=north]{$A_1$};
\draw[green] (3, 0) node[anchor=north]{$H$};
\draw[blue] (7, 0) node[anchor=north]{$A_2$};
\end{tikzpicture}
\end{center}
\end{proof}

Note that in particular, the domain of the homeomorphism given by \cref{bijectionQLibre} is in fact a space of \textit{complete types}.
\begin{corollary}\label{specialImpliqueDefClos}
The definable closure of a parameter set $A$ coincides with the special subgroup it generates.
\end{corollary}

\begin{proof}
The special subgroup generated by $A$ is clearly included in $\dcl(A)$. Conversely, if $A$ is a special subgroup, and $c\not\in A$, then the $1$-tuple $c$ is $\mathbb{Q}$-free over $A$, thus $c\not\in\acl(A)$ by \cref{clotureDefROAG}.
\end{proof}

There is another important corollary which allows us to better understand independence:

\begin{corollary}\label{indepCorrespProduit}
Let ${M}\models$\ROAG, let $A\leqslant B$ be special subgroups of $M$ such that ${M}$ is $|B|^+$ saturated and strongly $|B|^+$-homogeneous, and let $c=c_1\ldots c_n\in M$. Suppose $c_1\ldots c_n$ is $\mathbb{Q}$-free over $B$, and let $F$ be the closed subspace of $S^n(M)$ of tuples that are $\mathbb{Q}$-free over $M$. For each $j\in J$, consider the action of $\Aut({M}/A)$ on $\pi_j(F)$. Then the following conditions are equivalent:
\begin{enumerate}
\item $\indep{c}{A}{B}{\inv}$ (resp. $\indep{c}{A}{B}{\bo}$).
\item For each $j\in J$, $\tp_j(c/B)$, which is a partial type over $B$, can be extended to an $\Aut(M/A)$-invariant (resp. of bounded orbit) element of $\pi_j(F)$.
\end{enumerate}
\end{corollary}

\begin{proof}
First of all, note that the global extensions of $\tp(c/B)$ realized by tuples that are not $\mathbb{Q}$-free over $M$ divide over $A$, which implies that they have an unbounded orbit (see for instance \cite{Tent2012ACI}, Exercise 7.1.5), so we do have to restrict ourselves to $F$.
\par For each $j\in J$, $F\longrightarrow\pi_j(F)$ is an equivariant surjection, thus the stabilizer of a point of $F$ is a subgroup of that of its image, so the cardinal of its orbit is larger than the supremum of the cardinals of the orbits of each of its images, and we get the top-to-bottom direction by contraposition.
\par Suppose for each $j$, we have $p_j\in\pi_j(F)$ a witness of the second condition: it extends $\tp_j(c/B)$, and its orbit is a singleton (resp. bounded). Note that the orbit of $(p_j)_j$ under $\Aut({M}/A)$ is contained in the Cartesian product of the orbits of the $(p_j)_j$, therefore it is also a singleton (resp. bounded). Now, by \cref{bijectionQLibre} and quantifier elimination, the map $F\longrightarrow\prod\limits_j\pi_j(F)$ is an equivariant bijection, so the consistent global type corresponding to $(p_j)_j$ witnesses the first condition, and we get the bottom-to-top direction.
\end{proof}

Note that, given special subgroups $A\leqslant B$, and a tuple $c$, $c$ is always $A$-interdefinable with a subtuple $d$ which is $\mathbb{Q}$-free over $A$ (choose any of those subtuples, of maximal length), thus there is a natural equivariant homeomorphism between the space of global extensions of $\tp(c/B)$ and that of $\tp(d/B)$. Corollary \ref{indepCorrespProduit} completes the picture, and gives us easy conditions on $d$ to check whether $\indep{c}{A}{B}{\inv}$ or $\indep{c}{A}{B}{\bo}$.

\begin{remark}
Corollary \ref{indepCorrespProduit} raises one interesting question: does the orbit of an element of $\prod\limits_j\pi_j(F)$ coincide with the Cartesian product of the orbits of its components ? In other words, for $p_j\in\pi_j(F)$, is the natural (injective) map: 
$$\faktor{\Aut({M}/A)}{\bigcap\limits_j Stab(p_j)}\longrightarrow\prod\limits_j\faktor{\Aut({M}/A)}{Stab(p_j)}$$ a bijection ?
\end{remark}
\section{Invariance and boundedness of the global extensions of the partial types}

\begin{assumptions}\label{ROAGHypPrincipales}
Let ${M}\models$ \ROAG, let $A\leqslant B$ be special subgroups of ${M}$, let $\lambda=\max(|B|, 2^{\aleph_0})^+$, let $c=c_1\ldots c_n$ be a tuple from $M$, and suppose that ${M}$ is $\lambda$-saturated and strongly $\lambda$-homogeneous. Suppose $c$ is $\mathbb{Q}$-free over $B$, and let $F$ be the closed space of global types whose realizations are $\mathbb{Q}$-free over $M$. For each $j\in J$, let $p_j=\tp_j(c/B)$. A set of cardinality $\kappa$ is called \textit{small} if  ${M}$ is $\kappa^+$-saturated and strongly $\kappa^+$-homogeneous, else it is \textit{large}.
\end{assumptions}

\par By \cref{indepCorrespProduit}, in order to understand the global extensions of $\tp(c/B)$ which are invariant or have a bounded orbit under the action of $\Aut(M/A)$, one has to understand, for each $j\in J$ the global extensions in $\pi_j(F)$ of $p_j$ which are invariant or have a bounded orbit.

\subsection{Partial types using equality and order}

\begin{assumptions} 
In this subsection, on top of \cref{ROAGHypPrincipales} we fix $\overline{M}\models$ \DOAG~ some $|M|^+$-saturated, strongly $|M|^+$-homogeneous elementary extension of $\div({M})$.
\end{assumptions}

\begin{remark}\label{rqueCompletionDOAG=<}
Let $d$ be some tuple from $M$. Then $d\models p_<$ if and only if $c$ and $d$ have the same type over $B$ in $\overline{{M}}$.
\end{remark}

\begin{lemma}\label{realisationDansM}
Let $D$ be some small special subgroup of ${M}$. Let $\alpha=\alpha_0\ldots \alpha_{m-1}$ be a tuple from $\overline{{M}}$ which is $\mathbb{Q}$-free over $D$. Suppose for all non-zero $f\in \LC^m(\mathbb{Q})$, and for all $d\in D$, we have $f(\alpha)\not\in\left]d, d+\boldone\right[$ (this interval of $\overline{{M}}$ being by convention empty when $\boldone=0$). Then there exists $\sigma\in \Aut(\overline{{M}}/D)$ such that $\sigma(\alpha)$ is a tuple from $M$.
\end{lemma}

\begin{proof}
Suppose by induction we have $\sigma_i\in \Aut(\overline{{M}}/D)$ sending $\alpha_{<i}$ to a tuple from $M$ for some $i< m$ ($\sigma_0=id$ will do for $i=0$). Let $D_i$ be the special subgroup of ${M}$ generated by $D\sigma_i(\alpha_{<i})$, which is exactly the relative divisible closure in $M$ of $D+\sum\limits_{k<i}\mathbb{Z}\cdot\sigma_i(\alpha_k)$. Then, for all non-zero $f\in \LC^{n-i}(\mathbb{Q})$, for all $d\in D_i$, we have by hypothesis $f(\alpha_{\geqslant i})\not\in \left]d, d+\boldone\right[$. Let us find $\tau\in \Aut(\overline{{M}}/D_i)$ such that $\tau(\sigma_i(\alpha_i))\in M$. Then we could set $\sigma_{i+1}=\tau\circ\sigma_i$, and conclude by induction.
\par By strong homogeneity of $\overline{{M}}$, it is enough to show that any interval of $\overline{{M}}$ with bounds in $\div(D_i)\cup\{\pm\infty\}$ containing $\sigma_i(\alpha_i)$ has a point $\beta$ in $M$. Let $I$ be such an interval. If either the lower or the upper bound of $I$ is in $\{\pm\infty\}$, then $\beta$ clearly exists, one may choose some multiple of $d$ with large enough absolute value if $d\neq 0$, else choose any non-trivial element with correct sign. Now, suppose $I=\left]\dfrac{d_1}{N_1}, \dfrac{d_2}{N_2}\right[$, with $d_k\in D_i$, $N_k>0$ (by the $\mathbb{Q}$-freeness assumption on $\alpha$, $\sigma_i(\alpha_i)\not\in \div(D_i)$, thus it does not matter whether the bounds of $I$ belong to $I$). Then $I$ has a point in $M$ if and only if $]N_2 d_1, N_1 d_2[$ has a point in $N_1N_2M$. If ${M}$ is dense, then $I$ has infinitely many points in $M$, and we can use the axioms of \ROAG~ to conclude.
\par There remains to deal with the case where ${M}$ is discrete. We established earlier that the $\mathbb{Q}$-linear combinations of $\sigma_i(\alpha_{\geqslant i})$ do not belong to $]d, d+\boldone[$ for any $d\in D_i$. In particular, for every $N\in\mathbb{Z}$:
$$N_1N_2\sigma_i(\alpha_i)\not\in \bigcup\limits_{N\in\mathbb{Z}}]N_2d_1+N\cdot\boldone, N_2d_1+(N+1)\cdot\boldone[$$
and by $\mathbb{Q}$-freeness we also have $N_1N_2\sigma_i(\alpha_i)\not\in \left\lbrace N_2d_1+N\cdot\boldone|N\in\mathbb{Z}\right\rbrace$. As a result, $N_1N_2\sigma_i(\alpha_i)\in]N_2d_1, N_1d_2[\setminus]N_2d_1, N_2d_1+(N_1N_2+2)\cdot\boldone]$, which must imply that $N_2d_1+(N_1N_2+1)\cdot\boldone<N_1d_2$, so $I$ has at least $N_1N_2+1$ (in fact, infinitely many) points in $M$, and we can also conclude with the axioms of \ROAG.
\end{proof}

\begin{lemma}\label{realiserExtNonDev}
Let $\alpha\in M^n$ be a tuple which is $\mathbb{Q}$-free over $B$, such that $p=\tp^{\DOAG}(\alpha/B)$ does not fork over $A$. Let $D$ be a small special subgroup of $M$ containing $B$, and let $q\in S^n_\DOAG(\div(D))$ be an extension of $p$ which does not fork over $A$. Then $q$ has a realization in $M^n$.
\end{lemma}

\begin{proof}
First of all, the realizations of $q$ are $\mathbb{Q}$-free over $D$. Let $d\in D$, and let $f$ be a non-zero element of $\LC^n(\mathbb{Q})$. By \cref{realisationDansM}, one just has to prove that $q(x)\models f(x)\not\in ]d, d+\boldone[$. If not, then by \cref{thmTechniqueEnonce}, the interval $[d, d+\boldone]$ must have a point in $\div(A)$. By multiplying everything by a sufficiently large $N>0$, the interval $]Nd, Nd+N\cdot\boldone[$ has a point in $A$, and $Nd\in D$, thus $Nd\in A$.
\par In particular, we have $q_{|A}(x)\models Nf(x)\in]Nd, Nd+N\cdot\boldone[$, therefore $p(x)\models Nf(x)\in ]Nd, Nd+N\cdot\boldone[$, and $Nf(\alpha)\in ]Nd, Nd+N\cdot\boldone[$. As $Nd$ and $\alpha$ are in $A$, we have $Nf(\alpha)\in A$, contradicting $\mathbb{Q}$-freeness of $\alpha$.
\end{proof}

\begin{proposition}\label{ordreOrbites}
By \cref{rqueCompletionDOAG=<}, let $h$ be the natural injection:
$$\pi_<(F)\longrightarrow S^n_{\DOAG}(\div(M))$$
\par Consider the action of $\Aut({M}/A)$ on $Q=\pi_<(F)$. Then the following conditions are equivalent:
\begin{enumerate}
\item Some element of $Q$ is invariant and extends $p_<$.
\item Some element of $Q$ has a bounded orbit and extends $p_<$.
\item The partial type $p_<$ does not fork over $A$.
\item The partial type $p_<$ does not divide over $A$.
\item Every closed bounded interval with bounds in $B$ containing a $\mathbb{Z}$-linear combination of $c$ also has a point in $\div(A)$.
\item For some $p\in Q$ extending $p_<$, $h(p)$ extends to some $\Aut(\overline{{M}}/A)$-invariant type over $\overline{{M}}$.
\end{enumerate}
\end{proposition}

\begin{proof}
The directions $1\Longrightarrow 2$ and $3\Longrightarrow 4$ are immediate.
\par Let us prove $2\Longrightarrow 3$. Let $q_<\in Q$ witness $2$. For each prime $l$, let $q_l(x)
$ in $S^n_l(M)$ be the partial type $\left\lbrace \mathfrak{d}_{l^N}\left(f(x)\right)|n>0, f\in \LC(\mathbb{Z}), f\neq 0\right\rbrace=\tp_l(0/M)$. Note that $q_l$ is clearly invariant. Then the complete global type in $F$ corresponding to $(q_<, (q_l)_l)$ (recall $F\longrightarrow \pi_<(F)\times\prod\limits_{l\textup{ prime}} S^n_l(M)$ is a homeomorphism by \cref{bijectionQLibre}, thus this type is consistent) has a bounded orbit, and thus does not fork over $A$, and we get $3$.
\par Let us prove $4\Longrightarrow 5$. Suppose we have $b_k\in B$, $f\in \LC(\mathbb{Z})$ such that the formula $f(c)\in[b_1, b_2]$ witnesses the failure of $5$. Let us show that the formula $f(x)\in [b_1, b_2]$ divides over $A$, witnessing the failure of $4$. As in the proof of \cref{typePasInvariant}, it is enough to find $d\equiv_A b_1$ such that $d>b_2$. As $[b_1, b_2]$ has no point in $\div(A)$, we have $b_k\not\in A$, thus the singletons $b_1$ and $b_2$ are both $\mathbb{Q}$-free over $A$. Now, by \cref{bijectionQLibre}, one just has to show that $\tp_<(b_1/A)$ is consistent with $]b_2, +\infty[$. If not, then all the elements in $M$ of $]b_2, +\infty[$ have a type over $A$ in $\overline{{M}}$ which is distinct from that of $b_1$. As a result, by the characterization of $1$-types in \DOAG, there must exist in $\overline{{M}}$ a point in $\div(A)\cap ]b_1, b_2]$, a contradiction.
\par Let us show $5\Longrightarrow 6$. Suppose $6$ fails. Then, by \cref{thmTechniqueEnonce}, $p_<$ is inconsistent with the partial type $\left\lbrace f(x)\not\in I|I\in\mathcal{I}, f\in \LC(\mathbb{Z})\right\rbrace$, with $\mathcal{I}$ the set of closed bounded intervals with bounds in $M$ that have no points in $\div(A)$. By compactness, there exist finite subsets $\mathcal{I}'\subset\mathcal{I}$, $G\subset \LC^n(\mathbb{Z})$ such that $p_<(x)\models\bigvee\limits_{I\in\mathcal{I}', g\in G}g(x)\in I$. Let $D$ be the special subgroup of ${M}$ generated by $B$ and the bounds of the elements of $\mathcal{I}'$. Then, for all $p\in Q$ extending $p_<$, the restriction of $p$ to $D$ corresponds to a type in $S^n_{\DOAG}(D)$ which forks over $A$ in \DOAG.
\par Suppose now, by contradiction, that $5$ holds. Let $q$ be the type of $S^n_\DOAG(B)$ corresponding to $p_<$. Then $q$ does not fork over $A$. By extension, let $q'\in S^n_\DOAG(D)$ be an extension of $q$ which does not fork over $A$. By \cref{realiserExtNonDev}, $q'$ must admit a realization $\beta$ in $M^n$ (in particular $\beta\models p_<$). As $q'$ does not fork over $A$, we have $g(\beta)\not\in I$ for all $g\in G$, $I\in\mathcal{I}'$. As a result, $q_<=\tp_<(\beta/D)$ cannot be extended to any element of $F$. This means that in any elementary extension of ${M}$, no realization of $q_<$ is $\mathbb{Q}$-free over $M$. By compactness, this means that there exist finite subsets $P\subset M$, $G'\subset \LC^n(\mathbb{Q})\setminus\{0\}$, such that $q_<(x)\models \bigvee\limits_{m\in P, g\in G'}g(x)=m$.
\par Then we can reach a contradiction by extending $q'$ once more: let $\tilde{D}$ be the special subgroup of ${M}$ generated by $D\cup P$, let $\tilde{q}\in S^n_\DOAG(\tilde{D})$ be an extension of $q'$ which does not fork over $A$. Then with the same argument, $\tilde{q}$ has a realization $\gamma\in M^n$, and by hypothesis $\gamma$ is not $\mathbb{Q}$-free over $\tilde{D}$, a contradiction with the fact that $\tilde{q}$ does not fork over $A$.
\par Suppose $p\in Q$ witnesses $6$, and let us show that $p$ is $\Aut({M}/A)$-invariant, which would allow us to conclude the whole proof with the direction $6\Longrightarrow 1$. Let $q\in S^n_{\DOAG}(\overline{{M}})$ be some global $\Aut(\overline{{M}}/A)$-invariant extension of $h(p)$, and $\sigma\in \Aut({M}/A)$. Then $\sigma$ extends uniquely to $\sigma'$, an automorphism of the ordered group $\div(M)$, which fixes $A$ pointwise. We clearly have $h(\sigma(p))=\sigma'(h(p))$ (look at the atomic formulas with predicate $=, <$ that belong to $X_<$, check that their image by $\sigma$ is satisfied by the realizations of $\sigma'(h(p))$). By quantifier elimination in \DOAG, $\sigma'$ is a partial elementary map in $\overline{M}$. By strong homogeneity of $\overline{{M}}$, $\sigma'$ extends to some $\tilde{\sigma}\in \Aut(\overline{{M}})$. As $\sigma'$ pointwise-fixes $A$, so does $\tilde{\sigma}$, thus $\tilde{\sigma}(q)=q$. Now, we conclude that: $\sigma(p)=h^{-1}(\sigma'(h(p))=h^{-1}(\sigma'(q_{|\div(M)}))=h^{-1}(\tilde{\sigma}(q_{|\div(M)}))=h^{-1}(\tilde{\sigma}(q)_{|\tilde{\sigma}(\div(M))})=h^{-1}(q_{|\div(M)})=h^{-1}(h(p))=p$.
\end{proof}

The fifth condition of \cref{ordreOrbites} is a very simple geometric condition, the kind of statement that would be very satisfactory for a characterization of forking. In the subsections \ref{sousSectFini} and \ref{sousSectInfini}, we  look for similar conditions for the $(p_l)_l$.
\subsection{Mapping monsters to monsters}\label{sousSectHomeoOrdre}

\begin{remark}
    Up to isomorphism, the unique discrete subgroup of $\mathbb{R}$ is $\mathbb{Z}$, and all its elements are $\emptyset$-definable in $\mathcal{L}_P$. As a result, any discrete model of \ROAG~ is an elementary extension of $\mathbb{Z}$, which we identify with $\dcl(\emptyset)$.
\end{remark}

If $d$ is a tuple in some elementary extension of $M$, then \cref{ordreOrbites} shows that $p_<$ does not fork/divide over $A$ if and only if the corresponding partial type (via $h$) in $S_\DOAG(\div(M))$ does not fork over $A$. Moreover, one can easily see that $h$ is continuous, and the restriction of $h$ to the closed subspace of $\pi_<(F)$ of extensions of $p_<$ which do not fork over $A$ is, by \cref{realiserExtNonDev}, onto the closed subspace of $S(\div(M))$ of extensions of $h(p_<)$ which do not fork over $A$. It follows that they are homeomorphic.
\par One may think that we just proved that this topological space is homeomorphic to the one described in \cref{DOAGDescriptionFinaleHomeo}, but there is actually a subtle obstruction, coming from the fact that $\div(M)$ is not necessarily $|B|^+$-saturated. For the remainder of this subsection, we show that the space of extensions of $p_<$ which do not fork over $A$ is indeed homeomorphic to the space described \cref{DOAGDescriptionFinaleHomeo}. For this, we have to do a case disjunction depending on whether $M$ is dense or not.
\par If $M$ is dense, we show that there exists a sufficiently saturated elementary extension $N$ of $M$ such that $\div(N)$ is also sufficiently saturated.
\par If $M$ is discrete, then $\div(M)$ is never $\aleph_1$-saturated, for it does not admit non-zero elements having smaller Archimedean value than $\boldone$. The homeomorphism $h$ that we build does not come from the divisible closure, but from the quotient $\faktor{M}{\mathbb{Z}}$, which is a model of DOAG. We show that this structure is sufficiently saturated.
\par Note that we do not care about strong homogeneity here. Indeed, one may easily see that a partial global $A$-invariant type on an $|A|^+$-saturated first-order structure extends uniquely to an $A$-invariant type over any $|A|^+$-saturated elementary extension, regardless of strong homogeneity.

\begin{definition}
    Suppose $M$ is discrete. We define $\widetilde{M}$ as the quotient $\faktor{M}{\mathbb{Z}}$, and $\pi$ the canonical surjection. As $\mathbb{Z}$ is a convex subgroup, $\widetilde{M}$ is naturally an ordered Abelian group by \cref{cutDef}, and $\pi$ is an order-preserving map.
\end{definition}

\begin{proposition}
    If $M$ is discrete, then $\widetilde{M}\models \DOAG$.
\end{proposition}

\begin{proof}
    By saturation, $M\neq\mathbb{Z}$, therefore $\widetilde{M}$ is non-trivial. Let $\pi(a)\in \widetilde{M}$, and $m>0$. The quotient and remainder for the Euclidean division by $m$ are $\emptyset$-definable functions in $M$, therefore there exists $b\in M$ such that $a=m\cdot b+k\cdot\boldone$ for some $k\in\mathbb{Z}$ ($k<m$), in particular $\pi(a)=m\cdot\pi(b)$, concluding the proof.
\end{proof}

\begin{lemma}\label{ROAGHomeoModZ}
    The map $h':\ \tp_<(c/M)\longrightarrow \tp\left(\pi(c)/\widetilde{M}\right)$ is a well-defined continuous injection $\pi_<(F)\longrightarrow S^n_\DOAG\left(\widetilde{M}\right)$.
\end{lemma}

\begin{proof}
    Let $f\in\LC^n(\mathbb{Q})$, $b\in M$, and $c$, $d\in M^n$ such that $c$ and $d$ are $\mathbb{Q}$-free over $\dcl(b)$. Let $m>0$ be such that $m\cdot f\in\LC^n(\mathbb{Z})$.
    \par Let us show that $h'$ is well-defined. If $\tp_<(c/\dcl(b))=\tp_<(d/\dcl(b))$, then:
    $$
    \begin{array}{rcl}
    f(\pi(c))>\pi(b) & \Longleftrightarrow & m\cdot f(\pi(c))>m\cdot\pi(b)\\
    &\Longleftrightarrow&\forall k\in \mathbb{Z}\ m\cdot f(c)+k\cdot\boldone>b\\
    &\Longleftrightarrow&\forall k\in \mathbb{Z}\ m\cdot f(d)+k\cdot\boldone>b\\
    &\Longleftrightarrow&f(\pi(d))>\pi(b)
    \end{array}
    $$
    $$
    \begin{array}{rcl}
    f(\pi(c))=\pi(b)&\Longleftrightarrow&m\cdot f(\pi(c))=m\cdot\pi(b)\\
    &\Longleftrightarrow& m\cdot f(c)\in (m\cdot b+\mathbb{Z})\\
    &\Longrightarrow&m\cdot f(c)\in\dcl(b)\\
    &\Longrightarrow&\left\lbrace
    \begin{array}{l}
    f=0\\
    b\in\mathbb{Z}
    \end{array}
    \right.\textup{ by }\mathbb{Q}\textup{-freeness}\\
    &\Longrightarrow& f(\pi(d))=\pi(b)
    \end{array}
    $$
    it follows that $\tp(\pi(c)/\pi(\dcl(b)))=\tp(\pi(d)/\pi(\dcl(b)))$, therefore $h'$ is well-defined. The above computations also show us that $h'$ is continuous.
    \par Now, if $f\in\LC^n(\mathbb{Z})$, then, by $\mathbb{Q}$-freeness, we have:
    $$f(c)>b\Longleftrightarrow\forall k\in\mathbb Z\ f(c)>b+k\cdot\boldone\Longleftrightarrow f(\pi(c))>\pi(b)$$
    and:
    $$f(c)=b\Longrightarrow f=0\Longrightarrow f(\pi(c))=0$$
    it follows that $h'$ is an injection.
\end{proof}

\begin{lemma}\label{lemmeROAGHomeoOrdreTmp}
Suppose that $M$ is discrete. Let $h'$ be the map from \cref{ROAGHomeoModZ}, and let $q\in\pi_<(F)$. Then the following are equivalent:

\begin{enumerate}
    \item For all $f\in\LC^n(\mathbb{Z})$, for all closed bounded intervals $I$ with bounds in $M$ having no point in $\div(A)$, $q(x)\models f(x)\not\in I$.
    \item For all $f\in\LC^n(\mathbb{Q})$, for all closed bounded intervals $I$ with bounds in $\pi(M)$ having no point in $\pi(A)$, $h'(q)(x)\models f(x)\not\in I$.
\end{enumerate}
\end{lemma}

\begin{proof}
    Suppose $1$ fails, and let $f$, $I$ be a witness. As $A$ is a special subgroup, we have $\mathbb{Z}\subset A$, therefore the interval:
    $$[\min(I)-m\cdot\boldone, \max(I)+m\cdot\boldone]$$
has no point in $A$ for every $m<\omega$. It follows that:
$$I'=[\pi(\min(I)), \pi(\max(I))]$$
has no point in $\pi(A)$. As a result, $f$, $I'$ is a witness of the failure of $2$.
    \par Conversely, suppose $2$ fails, and let $f, I$ be a witness. We must have $f\neq 0$, otherwise $I$ would have a point in $\pi(A)$. Let $b_1$, $b_2$ be respective preimages of $\min(I)$, $\max(I)$ by $\pi$. Let $m>0$ be such that $m\cdot f\in\LC^n(\mathbb{Z})$. As $q\in\pi_<(F)$, and $f\neq 0$, $q(x)\models m\cdot f(x)\not\in M$. This would not be the case if we had $m\cdot \pi(b_1)=m\cdot \pi(b_2)$, therefore we have $m\cdot \pi(b_1)<m\cdot \pi(b_2)$, thus $m\cdot b_1<m\cdot b_2$, and $q(x)\models m\cdot f(x)\in I'=[m\cdot b_1, m\cdot b_2]$. In order to conclude the proof, it suffices to show that $I'$ has no point in $\div(A)$, which would establish the failure of $1$, with the witness $m\cdot f$, $I'$. Let $a\in A$, and $k>0$. If we had $\dfrac{a}{k}\in I'$, then we would have $\dfrac{1}{mk}\pi(a)\in I$. However, as $A$ is definably closed, and the quotient for the Euclidean division by $mk$ is $\emptyset$-definable, we would have $\dfrac{1}{mk}\pi(a)\in\pi(A)$, a contradiction.
\end{proof}

\begin{corollary}
    Suppose $M$ is discrete, and let $q\in\pi_<(F)$. Then the following are equivalent:
    \begin{enumerate}
        \item $q$ does not fork over $A$.
        \item $h'(q)$ does not fork over $\pi(A)$.
    \end{enumerate}
\end{corollary}

\begin{proof}
    By condition $5$ of \cref{ordreOrbites}, $q$ does not fork over $A$ if and only if condition $1$ of \cref{lemmeROAGHomeoOrdreTmp} holds. By \cref{defCutIndep}, the second condition of \cref{lemmeROAGHomeoOrdreTmp} exactly states that $h'(q)$ avoids all the definable sets witnessing cut-dependence (i.e. forking-dependence by \cref{thmTechniqueEnonce}) from $\pi(M)$ over $\pi(A)$, therefore $h'(q)$ does not fork over $\pi(A)$ if and only if condition $2$ of \cref{lemmeROAGHomeoOrdreTmp} holds.
\end{proof}

\begin{lemma}
    Suppose $M$ is discrete. Let $q\in S^n_\DOAG(\widetilde{M})$ be the type of a family which is $\mathbb{Q}$-free over $\widetilde{M}$. Then there exists $p\in\pi_<(F)$ such that $h'(p)=q$.
\end{lemma}

\begin{proof}
Let $I$ be a finite set. For each $i\in I$, let $b_i, b'_i\in M$, and let $f_i\in\LC(\mathbb{Q})$, such that $q(x)\models \pi(b_i)<f_i(x)<\pi(b'_i)$. By compactness, it suffices to find a tuple $d$ in $M$ such that $\pi(b_i)<f_i(\pi(d))<\pi(b'_i)$ for every $i$. By consistency of $q$, there exists a tuple $\delta$ in $\widetilde{M}$ such that $\pi(b_i)<f_i(\delta)<\pi(b'_i)$ for every $i$. We simply choose $d$ an arbitrary lift of $\delta$ by $\pi$.
\end{proof}

It follows that every type over $\widetilde{M}$ which extends $h'(p_<)$ and which does not fork over $\pi(A)$ has a lift in $\pi_<(F)$ by $h'$ which does not fork over $A$. As a result, $h'$ restricts to a homeomorphism between the space of types in $\pi_<(F)$ which extend $p_<$ and which do not fork over $A$, and the space of types in $S_\DOAG(\widetilde{M})$ which extend $h'(p_<)$ and which do not fork over $\pi(A)$.
\par Now we show that $\div(M)$ or $\widetilde{M}$ is sufficiently saturated:

\begin{proposition}
    Let $\kappa$ be an infinite cardinal such that $M$ is $\kappa$-saturated. If $M$ is dense, then $\div(M)$ is $\kappa$-saturated.
\end{proposition}

\begin{proof}
    Suppose $M$ is dense. Let $(I_i)_{i}$ be a consistent family of strictly less than $\kappa$ many intervals from $\div(M)$. It suffices to show that $\bigcap\limits_i I_i$ has a point in $\div(M)$ (as saturation for unary types implies saturation). We may assume every intersection of finitely many intervals from the family is also in the family. If some $I_i$ has empty interior, then it is a singleton, and the proof is trivial, therefore we may assume that every $I_i$ is open of non-empty interior.

    \par As $M$ is $\kappa$-saturated and dense, there exists a non-zero $\epsilon\in M$ such that $\Delta(\epsilon)<\Delta(\sup(I_i)-\inf(I_i))$ ($\Delta$ is defined in \cref{defArchVal}) for all $i$. In particular:

    $$
    \inf(I_i)<\inf(I_i)+|\epsilon|<\sup(I_i)-|\epsilon|<\sup(I_i)
    $$
    
    For each $i$, let $m_i<\omega$ be sufficiently large, such that $m_i\cdot\inf(I_i)$, $m_i\cdot (\inf(I_i)+|\epsilon|)$, $m_i\cdot(\sup(I_i)-|\epsilon|)$, $m_i\cdot\sup(I_i)$ are all in $M$. By regularity, as both the intervals $]m_i\cdot\inf(I_i), m_i\cdot (\inf(I_i)+|\epsilon|)[$ and $] m_i\cdot(\sup(I_i)-|\epsilon|), m_i\cdot\sup(I_i)[$ are infinite, they have a point in $m_i\cdot M$, therefore there exists $a_i, b_i\in M$ such that:
    $$
    \inf(I_i)<a_i<\inf(I_i)+|\epsilon|<\sup(I_i)-|\epsilon|<b_i<\sup(I_i)
    $$
    Let $I'_i=]a_i, b_i[$, which is an interval of $M$. We have $I'_i\subset I_i$ for every $i$. It suffices to show that the family $(I'_i)_i$ has a point in $M$. By saturation and density, it suffices to show that $a_i<b_j$ for every $i, j$. It suffices to show that $\inf(I_i)+|\epsilon|<\sup(I_j)-|\epsilon|$. As $(I_i)_i$ is closed under finite intersection, and by definition of $\epsilon$, we have:
    $$
    \begin{array}{ccl}
    2\cdot|\epsilon| & < &\sup(I_i\cap I_j)-\inf(I_i\cap I_j)\\
    &=&\min(\sup(I_i),\sup(I_j))-\max(\inf(I_i), \inf(I_j))\\
    & \leqslant & \sup(I_j)-\inf(I_i)
    \end{array}
    $$
    hence the result is proved.
\end{proof}

\begin{proposition}
    Let $\kappa$ be an infinite cardinal such that $M$ is $\kappa$-saturated, and suppose $M$ is discrete. Then $\widetilde{M}$ is $\kappa$-saturated.
\end{proposition}

\begin{proof}
    Let $(I_i)_i$ be a consistent family of strictly less than $\kappa$ many open intervals of non-empty interior of $\widetilde{M}$. Let $a_i, b_i\in M$ be such that $I_i=]\pi(a_i), \pi(b_i)[$. Then the family of every interval of the form $]a_i+k\cdot\boldone, b_i+k'\cdot\boldone[$, with $k, k'\in\mathbb{Z}$ is consistent in $M$. It can be defined with as many parameters as the initial family, therefore its intersection has a point $d\in M$. We clearly have $\pi(d)\in\bigcap\limits_i I_i$, concluding the proof.
\end{proof}

\begin{proposition}\label{propHomeoOrdre}
    The space of types in $\pi_<(F)$ which extend $p_<$ and do not fork over $A$ is homeomorphic to the space described in \cref{DOAGDescriptionFinaleHomeo}.
\end{proposition}

\begin{proof}
    The homeomorphism is given by $h$ if $M$ is dense, and $h'$ if it is discrete.
\end{proof}

\subsection{Partial types using a prime of finite index}\label{sousSectFini}
\par For the rest of this section, we fix a prime $l$.
\begin{proposition}\label{indicePuissance}
Let $G$ be a torsion-free Abelian group. Then, for all $N>0$, we have in $G^{eq}$ a $\emptyset$-definable group isomorphism between the two definable groups $\faktor{G}{lG}$ and $\faktor{l^NG}{l^{N+1}G}$.
\end{proposition}

\begin{proof}
Let $x, y\in G$, and $N>0$. Suppose $l^Nx-l^Ny\in l^{N+1}G$. Then, as $G$ is torsion-free, we have $x-y\in lG$, thus the map:
$$f:\ l^Nx\mod l^{N+1}G\longmapsto x\mod lG$$
is a well-defined group homomorphism, which is clearly surjective. Now, if $x\not\in lG$, then $l^Nx\not\in l^{N+1}G$, so the map is injective.
\end{proof}

\begin{corollary}\label{orbitesBorneesIndiceFini}
If $lG$ has a finite index $d$, then $\bigcap\limits_N l^NG$ has index at most equal to $2^{\aleph_0}$.
\end{corollary}

\begin{proof}
Define a tree structure on $\bigcup\limits_{N\geqslant 0}\faktor{G}{l^N G}$ of root $G$, such that the parent of some node $x\mod l^{N+1} G$ is $x\mod l^N G$. Then every node has $d$ children and the tree has $\aleph_0$ levels, thus it admits at most $2^{\aleph_0}$ branches. To conclude, the map $x\mod \bigcap\limits_N l^NG\longmapsto (x\mod l^NG)_N$ is a natural bijection between $\faktor{G}{\bigcap\limits_N l^NG}$ and the set of the branches.
\end{proof}

\begin{lemma}\label{bougerLesCosets}
Let $D$ be some small special subgroup of ${M}$. Suppose we have $e_1, e_2\in M$, such that $e_1, e_2\not\in D+l^NM$, and $e_1-e_2\in l^{N-1}M$, for some $N>0$. Then we have $\tp_l(e_1/D)=\tp_l(e_2/D)$.
\end{lemma}

\begin{proof}
Any atomic formula with parameters in $D$ and predicate in $\left\lbrace\mathfrak{d}_{l^m}|m>0\right\rbrace$ can clearly be written $\mathfrak{d}_{l^m}(\lambda x-d)$, with $\lambda\in\mathbb{Z},m>0, d\in D$. We have several cases:
\begin{itemize}
\item $\lambda\in l\mathbb{Z}$, $d\not\in lM$, in which case the formula always fails.
\item $\lambda\in l\mathbb{Z}$, $d\in lM$, $m=1$, in which case the formula always holds.
\item $\lambda\not\in l\mathbb{Z}$, in which case $\lambda$ and $l$ are coprime. Therefore, by Bézout's identities, the formula is equivalent to $\mathfrak{d}_{l^m}(x-d')$ for some $d'\in D$. By hypothesis on the $e_i$, this formula is satisfied by $e_1$ if and only if it is satisfied by $e_2$.
\item $\lambda\in l\mathbb{Z}$, $d\in lM$, $m>1$, in which case the formula is equivalent to $\mathfrak{d}_{l^{m-1}}\left(\dfrac{\lambda}{l}x-\dfrac{d}{l}\right)$, and we reduce by induction to one of the above three cases.
\end{itemize}
either way, we clearly see that those formulas are satisfied by $e_1$ if and only if they are satisfied by $e_2$, which concludes the proof.
\end{proof}

\begin{corollary}\label{completionTypeModL}
Let $d$ be a tuple from $M$, which is $\mathbb{Q}$-free over $B$. Then $d\models p_l$ if and only if, for each $f\in \LC^n(\mathbb{Z})$, for each $N>0$, either we have $f(c)-f(d)\in l^NM$, or neither $f(c)$ nor $f(d)$ are in $ B+l^NM$.
\end{corollary}

\begin{proof}
Any atomic formula $\varphi(x_1\ldots x_n)$ with predicate in $\left\lbrace\mathfrak{d}_{l^N}|N>0\right\rbrace$ can be written $\psi\left(f(x)\right)$, with $f\in \LC(\mathbb{Z})$, and $\psi$ an atomic formula on the same predicate with one variable. As a result, $d\models p_l$ if and only if $\tp_l\left(f(d)/B\right)=\tp_l\left(f(c)/B\right)$ for all $f\in \LC(\mathbb{Z})$. Fix $f\in \LC^n(\mathbb{Z})$.
\par Suppose that we have $f(c)\in B+l^NM$ for all $N>0$. For each $N$, let $b_N\in B$ such that $f(c)-b_N\in l^NM$. Then $\tp_l(f(c)/B)$ contains $\left\lbrace\mathfrak{d}_{l^N}(x-b_N)|N>0\right\rbrace$. This inclusion is an equality, for if $\mathfrak{d}_{l^N}(f(d)-b_N)$ for all $N$, then we have $f(d)-f(c)\in l^NM$ for all $N$, which clearly implies $\tp_l(f(c)/B)=\tp_l(f(d)/B)$. In particular, $\tp_l(f(d)/B)=\tp_l(f(c)/B)$ if and only if $f(d)-f(c)\in\bigcap\limits_N l^NM$.
\par Suppose now that there exists $N>0$ such that $f(c)\not\in B+l^NM$. The set of all such $N$ is a final segment of $\omega$, choose $N$ its least element. Let $b\in B$ such that $f(c)-b\in l^{N-1} M$ (if $N=1$, then $b=0$ will do). Then, by \cref{bougerLesCosets}, $\tp_l(f(c)/B)$ is generated by the partial type:
$$\{\mathfrak{d}_{l^{N-1}}(x-b)\}\cup\left\lbrace\neg\mathfrak{d}_{l^N}(x-b')| b'\in B\right\rbrace$$
this concludes the proof.
\end{proof}

\begin{remark}\label{indiceFiniOrbiteBornee}
If $[M:lM]$ is finite, then by \cref{orbitesBorneesIndiceFini} the $\emptyset$-type-definable equivalence relation $\left\lbrace\mathfrak{d}_{l^N}(x-y)|N>0\right\rbrace$ is bounded (has a small number of classes) and finer than the equivalence relation of having the same $l$-type over any parameter set, thus $S^n_l(M)$ is small, and each orbit of $S^n_l(M)$ under $\Aut({M}/A)$ is obviously bounded.
\par Moreover, the classes of this $\emptyset$-type-definable equivalence relation are all $\Aut(M^{eq}/\acl^{eq}(\emptyset))$-invariant, thus $\Aut(M^{eq}/\acl^{eq}(\emptyset))$ acts trivially over  $S^n_l(M)$.
\end{remark}

Let us recall that we adopted \cref{ROAGHypPrincipales} in the section.

\begin{proposition}\label{indiceFiniInvariant}
Suppose $[M:lM]$ is finite. Consider the action of $\Aut({M}/A)$ on $Q=S^n_l(M)$. Then $p_l$ has an invariant extension in $Q$ if and only if every $\mathbb{Z}$-linear combination of $c$ belongs to $\bigcap\limits_N(A+ l^NM)$.
\end{proposition}

\begin{proof}
Suppose every $\mathbb{Z}$-linear combination of $c$ belongs to $\bigcap\limits_N(A+ l^NM)$. Then $p_l(M)$ can be written as the intersection of $A$-definable sets of the form:
$$\left\lbrace x|\mathfrak{d}_{l^N}\left(f(x)-a\right)\right\rbrace,$$
with $N>0, a\in A$, $f\in \LC(\mathbb{Z})$. As a result, $p_l(M)$ is $A$-type definable (thus $\Aut(M/A)$-invariant), and all its realizations $d\in M^n$ satisfy the conditions $f(c)-f(d)\in\bigcap\limits_N l^NM$ for all $f\in \LC^n(\mathbb{Z})$. By \cref{completionTypeModL}, $p_l$ is complete as an element of $Q$, so we get the right-to-left direction.
\par Suppose now that we have $f\in \LC^n(\mathbb{Z})$ and $k>0$ such that $f(c)$ does not belong to $A+l^kM$. Let $p\in Q$ be an extension of $p_l$. By compactness and finiteness of $[M:lM]$, let $\alpha\in M$ be such that $p(x)\models \mathfrak{d}_{l^N}\left(f(x)-\alpha\right)$ for all $N>0$. As $[M:A+l^kM]>1$, we must have by \cref{indicePuissance} the inequalities $1<[M:l^kM]=[M:lM]^k$, thus $1<[M:lM]=[l^kM:l^{k+1}M]$, so let $m\in l^kM\setminus l^{k+1}M$. Then, by \cref{bougerLesCosets}, $\tp_l(\alpha/A)=\tp_l(\alpha+m/A)$. Let $D$ be a special subgroup of ${M}$ of size $\leqslant \max(2^{\aleph_0}, |A|)$ containing $A$, and a system of representatives of the cosets of $\bigcap\limits_N l^NM$. By saturation of $M$, we can find $\beta\in M$ such that $\tp_l(\beta/D)=\tp_l(\alpha-m/D)$ (which implies $\beta-\alpha-m\in\bigcap\limits_N l^NM$), and $\tp_j(\beta/D)=\tp_j(\alpha/D)$ for all $j\neq l$ (which implies $\tp_j(\beta/A)=\tp_j(\alpha/A)$). Then, by strong homogeneity of $M$, there exists $\sigma\in \Aut({M}/A)$ such that $\sigma(\alpha)=\beta$. Now, we have $\sigma(p)(x)\models \mathfrak{d}_{l^{k+1}}\left(f(x)-\alpha-m\right)$, a formula that is inconsistent with $\mathfrak{d}_{l^{k+1}}\left(f(x)-\alpha\right)$. As a result, $\sigma(p)\neq p$, concluding the proof.
\end{proof}

\begin{remark}
Just like in subsection \ref{sousSectHomeoOrdre}, we would like to describe the space of global invariant/bounded extensions of $p_l$. In case the condition of \cref{indiceFiniInvariant} holds, we saw that $p_l$ is $\Aut(M/A)$-invariant and complete in $S^n_l(M)$, thus this space is a single point. In case the condition does not hold, the space of invariant extensions of $p_l$ is obviously empty, but we would like to describe the topological space of all the extensions that have a bounded orbit. This is what we  do in the remainder of this subsection, we  exhibit a homeomorphism between this orbit and some closed subspace of the Cantor space.
\end{remark}

\begin{definition}
Let $N>0$. We define $\mathcal{T}_N$ the tree $\{0, \ldots  N-1\}^{<\omega}$, and $\mathcal{B}_N=\{0,\ldots N-1\}^\omega$ the set of its branches. The $m$-th-level of $\mathcal{T}_N$ is the set $\{0, \ldots  N-1\}^m$. We see $\mathcal{B}_N$ as a topological space, with the elementary open sets being the $O_w=\{b\in\mathcal{B}_N|w\textup{ is a prefix of }b\}$ for every $w\in\mathcal{T}_N$. For every $m$, the opens $(O_w)_{w\in\{0,\ldots N-1\}^m}$ form a finite disjoint open cover of $\mathcal{B}_N$, thus each $O_w$ is in fact clopen.
\end{definition}

\begin{lemma}
For all $N, d>0$, $\mathcal{B}_N$ is naturally homeomorphic to $\mathcal{B}_{N^d}$.
\end{lemma}

\begin{proof}
The natural bijection $\{0, \ldots  N^d-1\}\longrightarrow\{0, \ldots  N-1\}^d$ extends to a bijection $h:\ \mathcal{B}_{N^d}\longrightarrow\mathcal{B}_N$. For each $m<\omega$, the sets:
$$\left\lbrace h(O_w)|w\textup{ in the }m\textup{-th level of }\mathcal{T}_{N^d}\right\rbrace$$
$$\left\lbrace O_w|w\textup{ in the }md\textup{-th level of }\mathcal{T}_{N}\right\rbrace,
$$
coincide, thus $h^{-1}$ is continuous. To conclude, for each $w\in\mathcal{T}_N$, $O_w$ can be written as a finite union of elementary opens $O_v$, with $v$ in the $d\lceil|w|/d\rceil$-th level of $\mathcal{T}_N$, thus $h^{-1}(O_w)$ can be written as a finite union of some $O_v$, with $v$ in the $\lceil|w|/d\rceil$-th level of $\mathcal{T}_{N^d}$, therefore $h$ is continuous.
\end{proof}

\begin{lemma}
The two spaces $\mathcal{B}_l$ and $\mathbb{Z}_l$ (the topological space of $l$-adic integers) are naturally homeomorphic.
\end{lemma}

\noindent Recall that with our conventions, the $l$-adic valuation of $l^m$ is $-m$.

\begin{proof}
The tree of non-empty balls of $\mathbb{Z}_l$ (including the whole $\mathbb{Z}_l$) ordered by inclusion is isomorphic to $\mathcal{T}_l$ via the isomorphism mapping each $w\in\mathcal{T}_l$ to the ball of radius $-|w|$ around $\sum\limits_{m} w_ml^m$.
The corresponding bijection $\mathcal{B}_l\longrightarrow\mathbb{Z}_l$ is clearly $h:\ w\longmapsto\sum\limits_{m<\omega} w_ml^m$, which concludes the proof.
\end{proof}

\begin{lemma}
Suppose $[M:lM]=l^d$ is finite. Then $S^n_l(M)$ is naturally homeomorphic to the direct product $\mathcal{B}_{l^d}^n$ (with $n=|c|$).
\end{lemma}

\begin{proof}
Note that each point of $S^n_l(M)$ is a partial type that can be written $\left\lbrace x\in Y_m|m<\omega\right\rbrace$, with $Y_m\in \left(\faktor{M}{l^mM}\right)^n$. In particular, $S^n_l(M)$ is in natural bijection with $\left(\faktor{M}{\bigcap\limits_{m<\omega} l^mM}\right)^n$. For any $\mathcal{L}_P(M)$-term $t(x)$, the definable subset of $M$ defined by $\mathfrak{d}_{l^m}(t(x))$ is a union of cosets $\mod l^mM$, thus the elementary clopen subsets of $\left(\faktor{M}{\bigcap\limits_{m<\omega} l^mM}\right)^n$ induced by $S^n_l(M)$ are the Cartesian products of cosets of $l^mM$ for $m<\omega$. We build by induction a sequence of bijections $(f_m)_m$ between the $m$-th level of $\mathcal{T}_{l^d}$ and $\faktor{M}{l^mM}$. Suppose by induction hypothesis that we have $f_m$ ($f_0$ is the bijection between the point $\left\lbrace\faktor{M}{M}\right\rbrace$ and the point $\{\emptyset\}$). For each $w$ in the $m$-th level, $w$ has exactly $l^d$ children in the $m+1$-th level, and $f_m(w)$ is the union of exactly $l^d$ cosets $\mod l^{m+1}M$, so there is a bijection between them for each $w$, and we set $f_{m+1}$ as the union of those bijections. Then the $(f_m)_m$ define an homeomorphism between $\mathcal{B}_{l^d}$ and $\faktor{M}{\bigcap\limits_m l^mM}$. This concludes the proof.
\end{proof}

\begin{corollary}\label{homeoIndiceFini}
The closed space of every global extension of $p_l$ of bounded orbit under the action of $\Aut(M/A)$ is naturally homeomorphic to some closed subspace of $\mathbb{Z}_l^n$.
\end{corollary}
\subsection{Partial types using a prime of infinite index}\label{sousSectInfini}
We still fix a prime $l$, and we assume here that $[M:lM]$ is infinite.
\begin{lemma}\label{lemmechoixRpzTpL}
Let $D$ be a small special subgroup of ${M}$, and $\alpha=\alpha_1\ldots \alpha_m$ a finite tuple from $M$ which is $\mathbb{Q}$-free over $D$. Let $\beta=\beta_1\ldots \beta_m$ be another tuple from $M$ such that $\tp_l(\beta/D)=\tp_l(\alpha/D)$. Then there exists $\gamma=\gamma_1\ldots \gamma_m\equiv_D\alpha$ such that $\gamma_i-\beta_i\in\bigcap\limits_N l^NM$ for all $i$.
\end{lemma}

\begin{proof}
As $\bigcap\limits_N l^NM$ is large, choose by induction $\epsilon_i\in\bigcap\limits_N l^NM$ such that $\epsilon_i$ is not in the special subgroup generated by $D\alpha\beta\epsilon_{<i}$, and let $\beta_i'=\beta_i+\epsilon_i$. Then we have $\beta_i-\beta_i'\in\bigcap\limits_N l^NM$, and $\epsilon$ is $\mathbb{Q}$-free over the special subgroup generated by $D\alpha\beta$. As a result, $\alpha$ (and hence $(\alpha_i-\beta_i')_i$) is $\mathbb{Q}$-free over $D'$, the special subgroup generated by $D\beta'$. By \cref{bijectionQLibre}, there exists in $M$ a tuple $e=e_1\ldots e_m$ such that $e_i\in\bigcap\limits_N l^NM$ for all $i$, and $\tp_j(e/D')=\tp_j((\alpha_i-\beta'_i)_i/D')$ for all $j\neq l$. Let $\gamma_i=\beta_i'+e_i$. Then, for all $j\neq l$, we have $\tp_j(\gamma/D')=\tp_j(\alpha/D')$. Moreover, we have $\gamma_i-\beta_i'\in\bigcap\limits_Nl^NM$ for all $i$, thus $\tp_l(\gamma/D)=\tp_l(\beta'/D)=\tp_l(\beta/D)=\tp_l(\alpha/D)$, so we conclude that $\gamma\equiv_D\alpha$.
\end{proof}

\begin{remark}\label{ROAGInfiniImpliqueDense}
Recall that any discrete model of \ROAG~ is an elementary extension of $\mathbb{Z}$. In particular, if $[M:lM]$ is infinite for some $l$, then $M$ must be dense. In that case, the special subgroups of ${M}$ (which are exactly the definably closed sets) are its pure subgroups.
\end{remark}

\begin{lemma}\label{bougerNbreFiniCosetsIndiceInfini}
Let $(\alpha_i)_{i\in I}$ be a finite tuple from $M$ which is $\mathbb{Q}$-free over $B$. Suppose that for each $i\in I$, there exists $N>0$ such that $\alpha_i\not\in B+l^NM$, and choose $N_i$ the least of those integers for each $i\in I$. Then, there exists $(\beta_i)_i\equiv_B(\alpha_i)_i$ such that $\beta_i-\alpha_k\not\in l^{N_i}M$ for each $i, k\in I$.
\end{lemma}

\begin{proof}
Note that by \cref{ROAGInfiniImpliqueDense}, the special subgroups of ${M}$ are exactly its pure subgroups.
\par Let $b_i\in B$ such that $b_i-\alpha_i\in l^{N_i-1}M$. Suppose we can find a witness $(\gamma_i)_i$ of the statement with $(\alpha_i)_i$ replaced by $\left(\dfrac{\alpha_i-b_i}{l^{N_i-1}}\right)_i$ ($N_i$ will be replaced by $1$). Then we can choose $\beta_i=l^{N_i-1}\gamma_i+b_i$. For the remainder of the proof, we can suppose without loss that $N_i=1$ for every $i$ (in particular, $b_i=0$).
\par Let $V$ be the large $\mathbb{F}_l$-vector space $\faktor{M}{lM}$, and $U\leqslant V$ the $\mathbb{F}_l$-vector subspace $\faktor{(B+lM)}{lM}$. Note that $U$ is small, as it is the image by $B$ of the map $M\longrightarrow \faktor{M}{lM}$. Note also that we have $\alpha_i\in V\setminus U$ for every $i\in I$. Now, let $I_0\subset I$ such that $(\alpha_i)_{i\in I_0}$ is a lift of an $\mathbb{F}_l$-basis of the image in $\faktor{V}{U}$ of $U+\sum\limits_{i\in I}\mathbb{F}_l\cdot(\alpha_i\mod lM)$. Let $D$ be the special subgroup of ${M}$ generated by $B(\alpha_i)_{i\in I}$, and let $U'=\faktor{(D+lM)}{lM}$, a small vector subspace of $V$ containing $U$. As $\faktor{V}{U'}$ is large, one can choose an arbitrary family $(\beta_i)_{i\in I_0}\in M^{I_0}$ whose image in $\faktor{V}{U'}$ is $\mathbb{F}_l$-free. Now, for each non-zero $f\in \LC^{|I_0|}(\mathbb{Z})$, if $m$ is the largest integer such that all the coefficients of $f$ lie in $l^m\mathbb{Z}$, then we have:
$$f((\beta_i)_{i\in I_0})\in l^{m}M\ni f((\alpha_i)_{i\in I_0})$$
$$f((\beta_i)_{i\in I_0})\not\in B+l^{m+1}M\not\ni f((\alpha_i)_{i\in I_0})$$
thus we have $\tp_l((\alpha_i)_{i\in I_0}/B)=\tp_l((\beta_i)_{i\in I_0}/B)$ by the same reasoning as in \cref{completionTypeModL}. For each $i\in I_0$, as $\beta_i\mod lM\not\in U'$, we have $\beta_i\not\in D+lM$, thus $\beta_i-\alpha_k\not\in lM$ for each $k\in I$. By \cref{lemmechoixRpzTpL}, we can suppose without loss that $(\beta_i)_{i\in I_0}\equiv_B(\alpha_i)_{i\in I_0}$.
\par Choose $\sigma\in \Aut({M}/B)$ such that $\sigma(\alpha_i)=\beta_i$ for every $i\in I_0$. Let $\beta_i=\sigma(\alpha_i)$ for each $i\in I\setminus I_0$. Then we have $(\beta_i)_{i\in I}\equiv_B(\alpha_i)_{i\in I}$. In order to conclude, we need to show that $\beta_i-\alpha_k\not\in lM$ for each $i\in I\setminus I_0$, $k\in I$. Choose $i\in I\setminus I_0$. Then there must exist $f_i\in \LC(\mathbb{Z})$ such that $\alpha_i-f_i((\alpha_k)_{k\in I_0})\in B+lM$. Let $e_i\in B$ be such that $\alpha_i-f_i((\alpha_k)_{k\in I_0})-e_i$ lies in $lM$. Then $\beta_i-f_i((\beta_k)_{k\in I_0})-e_i\in lM$, so we just have to show that $f_i((\beta_k)_{k\in I_0})\not\in D+lM$. As $\alpha_i\not\in B+lM$, some coefficient of $f_i$ must be coprime with $l$, which concludes the proof as $(\beta_k\mod lM)_{k\in I_0}$ is $\mathbb{F}_l$-free over $U'$.
\end{proof}

\begin{lemma}\label{ROAGInfiniBLibre}
Let $(\alpha_i)_i$ be some tuple from $M$. Then there exists in $M$ a tuple $(\beta_i)_i$ which is $\mathbb{Q}$-free over $B$, and for which $\alpha_i-\beta_i\in\bigcap\limits_N l^NM$ for every $i$.
\end{lemma}

\begin{proof}
One merely has to define by induction $\epsilon_i\in\bigcap\limits_Nl^NM$ which does not lie in the small special subgroup generated by $B(\alpha_k)_k\epsilon_{<i}$, and choose $\beta_i=\alpha_i+\epsilon_i$.
\end{proof}

\begin{lemma}\label{lemmeIndiceInfini}
Let $N>0$. Suppose we have $f_1,\ldots f_N\in \LC^n(\mathbb{Z})$, as well as $Y_1,\ldots Y_N\in\bigcup\limits_m \faktor{M}{l^mM}$, such that every tuple of $M$ which realizes $p_l$ satisfies the formula $\bigvee\limits_i f_i(x)\in Y_i$. Then at least one of the $Y_i$ must have a point in $B$.
\end{lemma}

\begin{proof}
Let $\alpha_i\in Y_i$ for each $i$. By \cref{ROAGInfiniBLibre}, we can assume without loss that $(\alpha_i)_i$ is $\mathbb{Q}$-free over $B$. Suppose by contradiction that none of the $Y_i$ has a point in $B$. Then we can apply \cref{bougerNbreFiniCosetsIndiceInfini} to $(\alpha_i)_i$, we find $(\beta_i)_i\equiv_B (\alpha_i)_i$ such that, for all $m>0$ and for all $i$, $k$, if $\alpha_i\not\in B+l^mM$, then $\beta_i-\alpha_k\not\in l^mM$. In particular, as $Y_k$ does not intersect $B$, we have $\beta_i\not\in Y_k$ for all $i$, $k$. Let $\sigma\in \Aut({M}/B)$ such that $\sigma(\alpha_i)=\beta_i$ for each $i$. Then we must have $\left(\bigcup\limits_i Y_i\right)\cap\left(\bigcup\limits_i\sigma(Y_i)\right)=\emptyset$. However, as $p_l$ is $\Aut(M/B)$-invariant, we must have $p_l(x)\models\bigvee\limits_i f_i(x)\in \sigma(Y_i)$, a contradiction.
\end{proof}

Let us recall that we adopted \cref{ROAGHypPrincipales} in the section.

\begin{proposition}\label{indiceInfini}
Consider the action of $\Aut({M}/A)$ on $Q=S^n_l(M)$. Let $C$ be the group of all $\mathbb{Z}$-linear combinations of $c$. Then the following conditions are equivalent:
\begin{enumerate}
\item Some invariant element of $Q$ extends $p_l$.
\item Some element of $Q$ of bounded orbit extends $p_l$.
\item $p_l$ does not fork over $A$.
\item $p_l$ does not divide over $A$.
\item For each $N>0$, we have $(C+l^NM)\cap (B+l^NM)=A+l^NM$.
\end{enumerate}
\end{proposition}

\begin{proof}
The directions $1\Longrightarrow 2$ and $3\Longrightarrow 4$ are trivial.
\par Let us prove $2\Longrightarrow 3$. For each prime $l'\neq l$, we can define $q_{l'}$ just as in \cref{ordreOrbites}. There remains to show that there exists $q_<\in\pi_<(F)$ whose orbit is $\Aut(M/A)$-invariant. In some elementary extension of ${M}$, define by induction $(a_i)_i, (D_i)_i$ such that $D_i$ is the special subgroup generated by $Ma_{<i}$, and $a_i\in\bigcap\limits_{b\in D_i}]b, +\infty[$. Define $q_<=\tp_<(a/M)$. Then, for all $(\lambda_1\ldots \lambda_n)\in\mathbb{Z}^n$, and $b\in M$, whether $\sum\limits_i\lambda_i a_i>b$ depends uniquely on the sign of $(\lambda_i)_i$ in the anti-lexicographic sum $\mathbb{Z}^n$, and it does not depend on $b$. Moreover, $(a_i)_i$ is clearly $\mathbb{Q}$-free over $M$. As a result, $q_<$ is an $\Aut({M})$-invariant (and hence $\Aut({M}/A)$-invariant) element of $\pi_<(F)$, and we conclude.
\par Let us prove $4\Longrightarrow 5$. Let $f\in \LC(\mathbb{Z})$ and $N>0$ be such that $f(c)$ is in $B+l^NM\setminus A+l^NM$. Let $b\in B$ such that $f(c)-b\in l^NM$ (we know that $b\not\in A+l^NM$ by hypothesis, in particular it is $\mathbb{Q}$-free over $A$ as a singleton). Let us show that the formula $\varphi(x, b):=\mathfrak{d}_{l^N}\left(f(x)-b\right)$ divides over $A$, which implies that $4$ fails. We can repeatedly apply \cref{bougerNbreFiniCosetsIndiceInfini} to find $(b_i)_{i<\omega}$ such that $b_0=b$, $b_i\equiv_Bb_{i+1}$, and $b_i-b_j\not\in l^NM$ for all $j\neq i$. Then the set of formulas $\left\lbrace\varphi(x, b_i)|i<\omega\right\rbrace$ is clearly $2$-inconsistent, so we can conclude.
\par Let us prove $5\Longrightarrow 1$. Suppose $5$ holds, and let us build an explicit invariant element of $Q$ which extends $p_l$, witnessing $1$. Define:
$$F=\left\lbrace (f, N)\in \LC^n(\mathbb{Z})\times\omega|f(c)\in A+l^NM\right\rbrace$$
\par For each $(f, N)\in F$, let $a_{f, N}\in A$ such that $f(c)-a_{f, N}\in l^NM$. Define the following partial type:
$$
\begin{array}{ccl}
p(x) & = & \left\lbrace \mathfrak{d}_{l^N}(f(x)-a_{f, N})|(f, N)\in F\right\rbrace\\
&&\cup\left\lbrace\neg\mathfrak{d}_{l^N}(f(x)-e)|(f, N)\not\in F, e\in M\right\rbrace
\end{array}
$$
\par Suppose by contradiction that $p$ is not consistent with $p_l$. Then, by compactness, there exists a finite tuple $(f_i, N_i, e_i)_i$ such that, for every $i$, $(f_i, N_i)\not\in F$, and $p_l(x)\models \bigvee\limits_i \mathfrak{d}_{l^{N_i}}(f_i(x)-e_i)$. By \cref{lemmeIndiceInfini} applied to $Y_i=e_i\mod l^{N_i}M$, there must exist $i$ such that $e_i\in B+l^{N_i}M$. As $(f_i, N_i)\not\in F$, we must have $e_i\not\in A+l^{N_i}M$, thus $e_i\in B+l^{N_i}M\setminus A+l^{N_i}M$. By hypothesis $5$, we have $p_l(x)\models \neg\mathfrak{d}_{l^{N_i}}(f_i(x)-e_i)$, which implies that $p_l(x)\models \bigvee\limits_{k\neq i} \mathfrak{d}_{l^{N_k}}(f_k(x)-e_k)$. We keep decreasing by induction the size of this disjunction, until we inevitably reach a contradiction. As a result, $p$ must be consistent with $p_l$, and it is clearly $\Aut(M/A)$-invariant. This partial type is of course complete as an element of $Q$, because every atomic formula with predicate in $\left\lbrace\mathfrak{d}_{l^N}|N>0\right\rbrace$ and parameters in $M$ either lies in $p$, or its negation lies in $p$.
\end{proof}

\begin{remark}\label{homeoIndiceInfini}
When the conditions of \cref{indiceInfini} hold, the explicit type in the proof of $5\Longrightarrow 1$ is actually the unique global extension of $p_l$ which has a bounded orbit. Indeed, if $q$ is another global extension, and $d$ is a realization of $q$ in some elementary extension, then we must have $f(d)-\alpha\in l^NM$ for some $(f, N)\not\in F$, and $m\in M$. In that case, $m\not\in A+l^NM$, thus the fifth condition from \cref{indiceInfini} fails with $c, B$ replaced by $d, M$. In particular, $q$ divides over $A$, thus its orbit is unbounded.
\end{remark}

\section{Computation of forking}
\begin{theorem}\label{resultatPpal}
Let ${M}\models$\ROAG, let $A$, $B$ be parameter subsets of $M$, let $\kappa=\max(|AB|, 2^{\aleph_0})^+$, and let $c=c_1\ldots c_n\in M^n$. Suppose ${M}$ is $\kappa$-saturated and strongly $\kappa$-homogeneous. Let $C$ be the subgroup of $M$ generated by $c$, and $A', B'$ the special subgroups generated by $A, AB$. Then the following conditions are equivalent:
\begin{itemize}
\item $\indep{c}{A}{B}{\f}$
\item $\indep{c}{A}{B}{\d}$
\item $\indep{c}{A}{B}{\bo}$
\item $\tp(c/AB)$ has a global $\Aut(M^{eq}/A\cup\acl^{eq}(\emptyset))$-invariant extension.
\item the following conditions hold:

\begin{enumerate}
\item Every closed bounded interval of $B'$ that has a point in $C$ already has a point in $\div(A')$.
\item For all prime $l$, if $[M:lM]$ is infinite, then we have, for all $N>0$, $(C+l^NM)\cap (B'+l^NM)=A'+l^NM$.
\end{enumerate}
\end{itemize}
\par Moreover, $\indep{c}{A}{B}{\inv}$ if and only if the above conditions hold, and, additionally, for every prime $l$ for which $[M:lM]$ is finite, we have $C\subset A'+\bigcap\limits_{N>0} l^NM$.
\end{theorem}

\noindent Note that, if $\notIndep{c}{A}{B}{\d}$, then $\tp(c/AB)$ divides over $\acl^{eq}(A)$, thus it does not admit any global $\Aut(M^{eq}/A\cup\acl^{eq}(\emptyset))$-invariant extension.

\begin{proof}
Firstly, we have $A'=\dcl(A)$ and $B'=\dcl(AB)$ by \cref{specialImpliqueDefClos}, thus we have $\indep{c}{A}{B}{}$ if and only if $\indep{c}{A'}{B'}{}$ for every $\indep{}{}{}{}\in\{\indep{}{}{}{\f}, \indep{}{}{}{\d}, \indep{}{}{}{\bo}, \indep{}{}{}{\inv}\}$. Secondly, given $c'=c'_1\ldots c'_m$ a maximal subtuple of $c$ which is $\mathbb{Q}$-free over $A'$, $c$ and $c'$ are $A$-interdefinable, thus one can show that for all of those independence notions we have $\indep{c}{A'}{B'}{}$ if and only if $\indep{c'}{A'}{B'}{}$. Likewise, $\tp(c/AB)$ has an $\Aut(M^{eq}/A\cup\acl^{eq}(\emptyset))$-invariant extension if and only if $\tp(c'/B')$ has an $\Aut(M^{eq}/A'\cup\acl^{eq}(\emptyset))$-invariant extension. 
\par If $c'$ was not $\mathbb{Q}$-free over $B'$, then one could show that $\notIndep{c'}{A'}{B'}{\d}$ (hence $\notIndep{c'}{A'}{B'}{\f}$, $\notIndep{c'}{A'}{B'}{\bo}$, and $\tp(c/AB)$ does not admit any global extension which is $\Aut(M^{eq}/A\cup\acl^{eq}(\emptyset))$-invariant). Moreover, condition $1$ would fail on some singleton of $B\setminus A$ witnessing the fact that $c'$ is not $\mathbb{Q}$-free over $B$.
\par Suppose $c'$ is $\mathbb{Q}$-free over $B'$. For each $j\in J$, let $p_j=\tp_j(c'/B')$, and let $F_j$ be the image by $\pi_j$ of the space of complete global types with realizations that are $\mathbb{Q}$-free over $M$.
\par If condition 1 fails, then $p_<$ divides over $A'$ by \cref{ordreOrbites}, and if condition 2 fails for some $l$, then $p_l$ divides over $A'$ by \cref{indiceInfini}. Either way, we have $\notIndep{c'}{A'}{B'}{}$ for every $\indep{}{}{}{}\in\{\indep{}{}{}{\f}, \indep{}{}{}{\d}, \indep{}{}{}{\bo}, \indep{}{}{}{\inv}\}$.
\par Suppose conditions 1 and 2 hold, then $p_<$ extends to an $\Aut({M}/A')$-invariant (thus $\Aut({M^{eq}}/A'\cup\acl^{eq}(\emptyset))$-invariant) element of $F_<$ by \cref{ordreOrbites}, and, for every prime $l$ for which $[M:lM]$ is infinite, $p_l$ extends to some $\Aut({M}/A')$-invariant extension of $F_l$ by \ref{indiceInfini}. By \cref{indiceFiniOrbiteBornee}, for each prime $l$ of finite index, any global extension of $p_l$ is automatically $\Aut({M^{eq}}/A'\cup\acl^{eq}(\emptyset))$-invariant, thus it follows that $\tp(c'/A'B')$ has a global $\Aut(M^{eq}/A'\cup\acl^{eq}(\emptyset))$-invariant extension. Moreover, by \cref{indiceFiniOrbiteBornee}, we know that $p_l$ extends to some element of $F_l$ of bounded orbit under $\Aut({M}/A')$ for every prime $l$ for which $[M:lM]$ is finite. By \cref{indepCorrespProduit}, $\indep{c'}{A'}{B'}{\bo}$, thus $\indep{c'}{A'}{B'}{\f}, \indep{c'}{A'}{B'}{\d}$. Moreover, by \cref{indepCorrespProduit}, we have $\indep{c'}{A'}{B'}{\inv}$ if and only if, for every prime $l$ for which $[M:lM]$ is finite, $p_l$ extends to some $\Aut({M}/A')$-invariant element of $F_l$. Then we can conclude using \cref{indiceFiniInvariant}.
\end{proof}

\begin{corollary}
    We have $\indep{}{}{}{\f}=\indep{}{}{}{\Sh}$ in \ROAG, where $\indep{c}{A}{B}{\Sh}$, Shelah-independence, is equivalent to $\tp(c/AB)$ having a global extension which is $\Aut(M^{eq}/\acl^{eq}(A))$-invariant.
\end{corollary}

\begin{proof}
Shelah-invariance is a weaker condition than being invariant under $\Aut(M^{eq}/A\cup\acl^{eq}(\emptyset))$, which is weaker than forking in general. The other direction follows from the theorem.
\end{proof}

\par We know (see \cref{DOAGRappelsNIP}) that in any NIP theory, $\indep{}{}{}{\f}$ coincides with $\indep{}{}{}{\inv}$ over models. We can now refine that result for \ROAG:

\begin{corollary}\label{ROAGComparaisonFInv}
Let $A'$ be the special subgroup generated by $A$. Then $\indep{}{A}{}{\f}=\indep{}{A}{}{\inv}$ if and only if $A'+\bigcap\limits_N l^NM=M$ for all primes $l$ of finite index.
\par Moreover, $\indep{}{}{}{\f}=\indep{}{}{}{\inv}$ if and only if, for each prime $l$, either $M$ is $l$-divisible, or the index $[M:lM]$ is infinite.
\end{corollary}

\begin{remark}\label{conclusionHomeo}
Suppose that $M$ is dense. Then, by \cref{propHomeoOrdre}, \cref{homeoIndiceInfini}, \cref{homeoIndiceFini}, and \cref{DOAGDescriptionFinaleHomeo}, we get the three following properties :
\begin{itemize}
\item If $\indep{c}{A}{B}{\f}$, then every global extension of $p_<$ which does not fork over $A$ is $\Aut(M/A)$-invariant. Moreover, the Stone space of those extensions can be written as $S_1$, a finite product of finite coproducts of closed subspaces of $S_{\DOAG}(\{0\})$.
\item Let $l$ be prime such that $[M:lM]$ is finite. If $C\subset A'+\bigcap\limits_N l^NM$, then $p_l$ has only one global extension, and it is $\Aut(M/A)$-invariant. Else, none of the global extensions of $p_l$ is $\Aut(M/A)$-invariant, but they are all non-forking over $A$, and the space of those extensions can be written $S_{2, l}$, a closed subspace of $\mathbb{Z}_l^{|c'|}$.
\item Let $l$ be prime such that $[M:lM]$ is infinite. If $\indep{c}{A}{B}{\f}$, then $p_l$ has only one global extension which does not fork over $A$, and this extension is $\Aut(M/A)$-invariant.
\end{itemize}

\noindent As a result, if we define:
$$L=\left\lbrace l\textup{ prime }|[M:lM]<\omega, C\not\subset A'+\bigcap\limits_N l^NM\right\rbrace,$$
then the space $S$ of global extensions of $\tp(c/AB)$ which do not fork over $A$ is either empty, or homeomorphic to $S_1\times\prod\limits_{l\in L} S_{2, l}$.
\par Moreover, if $S\neq\emptyset$ and $L=\emptyset$, then all the elements of $S$ are $\Aut(M/A)$-invariant, thus $S$ is homeomorphic to $S_1$, and it coincides with the space of global $\Aut(M/A)$-invariant extensions of $\tp(c/AB)$.
\end{remark}

\section{Examples}\label{sectROAGEx}
\subsection{Non-forking versus finite satisfiability}\label{fs}
\par Let us start with the very basic remark that the cut independence notion does not coincide with finite satisfiability in \DOAG. In an elementary extension of the ordered group $\mathbb{R}$, let $\varepsilon>0$, $c>0$ such that $\Delta(c)<\Delta(\varepsilon)<\Delta(1)$. Let $A=\mathbb{Q}$, $B=\mathbb{Q}+\mathbb{Q}\varepsilon$. Then we clearly have $\indep{c}{A}{B}{\cut}$, but the type of $c$ over $B$ is not finitely satisfiable over $A$: the $B$-definable set $]0, \varepsilon]$ contains $c$ and has no point in $A$.

\subsection{Non-forking versus sequential independence}\label{contreExLemmePairs}
\par We would like to compare non-forking independence with a stronger independence notion:

\begin{definition}
Let ${M}$ be any first-order structure, and $A$, $B$, $C\subset{M}$ parameter sets. We  write $\indep{C}{A}{B}{\seq}$ when, for any finite tuple $c$ from $C$, there exists a finite tuple $d=(d_0,\ldots d_n)$ from $\acl(AC)$ such that $\acl(Ac)\subset\acl(Ad)$ and $\indep{d_i}{Ad_{\leqslant i}}{B}{\f}$ for every $i\leqslant n$.
\end{definition}

This relation, which we  call \textit{sequential independence}, is stronger than $\indep{}{}{}{\f}$ in any theory (an easy application of left-transitivity of non-forking). One can show that it coincides with $\indep{}{}{}{\f}$ in simple theories, and this is also the case in some non-simple theories, like DLO (a parameter set coincides with its algebraic closure, and it is easy to guess the global invariant extensions of the type of some finite tuple). In a theory where $\indep{}{}{}{\seq}=\indep{}{}{}{\f}$, it is really easy to understand forking, because one can reduce the problem to forking in dimension one. We showed in the previous section that forking in \ROAG~ is essentially a dimension one phenomenon, so it makes sense to ask whether $\indep{}{}{}{\seq}=\indep{}{}{}{\f}$ in these structures. We  show that the answer is no with a counterexample in \DOAG. The only other explicit example we know of $\indep{}{}{}{\f}\neq\indep{}{}{}{\seq}$ is in (\cite{HHM}, example 13.6), in the theory of algebraically closed valued fields. This example does not occur exclusively in the value group, one needs field-theoretic tools to make it work, which is why our example is arguably less complicated.
\par In the reduct to the ordered group structure of an elementary extension of the real-closed field $\mathbb{R}$, let $\varepsilon>0$ such that $\Delta(\varepsilon)<\Delta(1)$. Let $A=\mathbb{Q}$, $B=\mathbb{Q}+\mathbb{Q}\sqrt{2}+\mathbb{Q}\sqrt{3}$, and $c=(c_1, c_2)=(\sqrt{2}+\varepsilon,\sqrt{3}+\varepsilon\sqrt{2})$.
\par Let us show that we have $\indep{c}{A}{B}{\f}$. We know that any two distinct real numbers have a different cut over $\mathbb{Q}=A$, so any bounded $B$-definable closed interval that does not have any point in $A$ must be a singleton. As we have $\dcl(Ac)\cap \dcl(B)=A$, we have in fact $\indep{c}{A}{B}{\cut}$. In particular, we have $\indep{c}{A}{B}{\f}$.
\par Now let us show $\notIndep{c}{A}{B}{\seq}$. Let $d=(d_1,\ldots d_n)$ be a finite tuple from $\acl(Ac)$ such that $\acl(Ac)\subset\acl(Ad)$. We can assume that the family $d$ is $\mathbb{Q}$-free over $A$, for if $d_i\in\dcl(Ad_{<i})$, then we trivially have $\indep{d_i}{Ad_{<i}}{B}{\f}$. We can also assume $d_i\in \mathbb{Q}c_1+\mathbb{Q}c_2$, for if $a\in A$, then $\acl(Ad_{\leqslant_i})=\acl(Ad_{<i}(d_i+a))$, so we can freely translate any $d_i$ by an element of $A$. By basic linear algebra, we have $n=2$, and $(d_1, d_2)$ is a $\mathbb{Q}$-basis of $\mathbb{Q}c_1+\mathbb{Q}c_2$. Let $f_i$ be the unique map of $\LC^2(\mathbb{Q})$ such that $f_i(c_1, c_2)=d_i$. As $(c_1, c_2)$ is also $\mathbb{Q}$-free, $(f_1, f_2)$ must be $\mathbb{Q}$-free. For any $g\in \LC^2(\mathbb{Q})$:
$$
\begin{array}{ccl}
d_2-g(1, d_1) & = & f_2(\sqrt{2}, \sqrt{3})-g(1, f_1(\sqrt{2}, \sqrt{3}))\\
&&+\varepsilon\left[f_2(1, \sqrt{2})-g(0, f_1(1, \sqrt{2}))\right]
\end{array}$$
As $(f_1, f_2)$ and $(1, \sqrt{2}, \sqrt{3})$ are both $\mathbb{Q}$-free, we must have:
$$f_2(\sqrt{2}, \sqrt{3})-g(1, f_1(\sqrt{2}, \sqrt{3}))\neq 0$$
so $\Delta(d_2-g(1, d_1))=\Delta(1)$. We established that, for all $x\in \dcl(Ad_1)$, we have $\Delta(d_2-x)>\Delta(\epsilon)$, which implies $\epsilon\in H(d_2/\dcl(Ad_1))$. Now, if we had $\indep{d_2}{Ad_1}{B}{\f}$, then we would necessarily have $\epsilon\in H(d_2/\dcl(Bd_1))$ by \cref{HPlusGrand}. Let us show that this fails. It is enough to show that $\Delta(d_2-b)=\Delta(\epsilon)$ for some $b\in \dcl(Bd_1)$. Choose $b=d_1-f_1(\sqrt{2}, \sqrt{3})+f_2(\sqrt{2}, \sqrt{3})$. Then, as $d_i=f_i(c_1, c_2)=f_i(\sqrt{2}, \sqrt{3})+f_i(\epsilon, \epsilon\sqrt{2})$, we have $d_2-b=f_2(\epsilon, \epsilon\sqrt{2})-f_1(\epsilon, \epsilon\sqrt{2})$. As $(f_1, f_2)$ is $\mathbb{Q}$-free, $f_2-f_1\neq 0$, therefore $d_2-b$ is a non-trivial linear combination of $(\epsilon, \epsilon\sqrt{2})$. This latter family is clearly $\Delta$-separated, thus $\Delta(d_2-b)=\Delta(\epsilon)$, and we can conclude.
\par This holds for any choice for the family $d$, so we have $\notIndep{c}{A}{B}{\seq}$, which concludes the example.

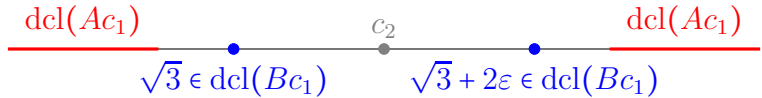
\begin{center}
\begin{tikzpicture}[scale=2]
\draw[black] (0, 0) -- (5, 0);
\draw[red, very thick] (0, 0) -- (1, 0);
\draw[red, very thick] (4, 0) -- (5, 0);
\draw[red, very thick] (0.5, 0) node[anchor=south]{$\dcl(Ac_1)$};
\draw[red, very thick] (4.5, 0) node[anchor=south]{$\dcl(Ac_1)$};
\filldraw[gray] (2.5, 0) circle(1pt) node[anchor=south]{$c_2$};
\filldraw[blue] (1.5, 0) circle(1pt) node[anchor=north]{$\sqrt{3}\in\dcl(Bc_1)$};
\filldraw[blue] (3.5, 0) circle(1pt) node[anchor=north]{$\sqrt{3}+2\varepsilon\in\dcl(Bc_1)$};
\end{tikzpicture}
\end{center}

\subsection{Weak versus strong versions of sequential independence}
\par In the definition of sequential independence, we require that there exists a good enumeration $d$ that witnesses independence. Does it change anything to ask this for all the enumerations? We  show that the answer is yes with a counterexample in \DOAG.
\begin{definition}
Define $\indep{C}{A}{B}{\seq*}$ when, for every finite tuple $c=(c_0..c_n)$ from $\acl(AC)$, we have $\indep{c_i}{Ac_{<i}}{B}{\f}$ for all $i\leqslant n$.
\end{definition}
This notion is at least as strong as $\indep{}{}{}{\seq}$, which is strictly stronger than $\indep{}{}{}{\f}$ in \DOAG~ by the above example. Let us show that $\indep{}{}{}{\seq*}$ is strictly stronger than $\indep{}{}{}{\seq}$ in \DOAG.
Let $\varepsilon>0$, $c_2>0$ such that $\Delta(\varepsilon)<\Delta(c_2)<\Delta(1)$. Let $A=\mathbb{Q}$, $B=\mathbb{Q}+\mathbb{Q}\sqrt{2}$, and $c_1=\sqrt{2}+\varepsilon$.
\par With the same reasoning as in the previous example, $\indep{c_1c_2}{A}{B}{\cut}$. In particular, we have $\indep{c_1}{A}{B}{\f}$ and $\indep{c_2}{A}{B}{\f}$.
\par As $\Delta(c_2)\not\in\Delta(\dcl(Ac_1))$, $c_2$ is ramified over $\dcl(Ac_1)$. The only points of $\dcl(Bc_1)$ that have the same cut over $\dcl(Ac_1)$ as that of $c_2$ are the $\lambda \varepsilon$, with $\lambda\in\mathbb{Q}_{>0}$. They are all smaller than $c_2$, so $c_2$ leans right with respect to $\dcl(Ac_1)$ and $\dcl(Bc_1)$, and $\indep{c_2}{Ac_1}{B}{\f}$, which implies $\indep{c_1c_2}{A}{B}{\seq}$.
\par Let us show $\notIndep{c_1c_2}{A}{B}{\seq*}$. The interval $[\sqrt{2}, \sqrt{2}+c_2]$ contains $c_1$ and has no point in $\dcl(Ac_2)$, which implies $\notIndep{c_1}{Ac_2}{B}{\f}$.

\subsection{Forking versus geometric cut-dependence} 
\par Let us look at a different definition for cut independence:

\begin{definition}
Let ${M}$ be the expansion of some totally ordered set, and $A$, $B$, $C$ parameter sets. We  write $\notIndep{C}{A}{B}{\cut*}$ when there exists $c\in\dcl(AC)$ and $b_1$, $b_2\in\dcl(AB)$ such that $b_1\leqslant c\leqslant b_2$, $b_1\equiv_A b_2$ and $b_i\not\in\dcl(A)$.
\end{definition}

We clearly have the implications $\indep{}{}{}{\d}\subset\indep{}{}{}{\cut*}$ and:
$$\indep{C}{A}{B}{\cut*}\Longleftarrow \indep{\div(\dcl(AC))}{\div(\dcl(A))}{\div(\dcl(AB))}{\cut}$$
in any ordered Abelian group. It is not hard to prove that these two implications are equivalences in \DOAG. In a dense model of \ROAG~$M$, if $B$ is an $|A|^+$-saturated model containing $A$, then one may show that the second implication is an equivalence, i.e. $\indep{C}{A}{B}{\cut*}$ holds in ${M}$ if and only if $\indep{\div(C)}{\div(A)}{\div(B)}{\cut}$ holds in $\div(M)$.
\par Whether the first inclusion is an equality is a little more complicated. Let $G$ be an ordered Abelian group, $A$ a parameter set, and $B$ an $|A|^+$-saturated model containing $A$. In case $G$ is regular and $|\faktor{G}{lG}|$ is infinite for some prime $l$ (this is possible, take for instance $G=\sum\limits_{n<\omega}\mathbb{Z}r_n$, with $(r_n)_n$ a $\mathbb{Q}$-free subfamily of $\mathbb{R}$), the characterization of forking and dividing that we have in \cref{resultatPpal} implies that $\indep{}{A}{B}{\d}$ does not coincide with $\indep{}{A}{B}{\cut*}$, because we have partial types $\mod\equiv_l$ that divide over $A$. However, if $G$ is dp-minimal, then we have indeed $\indep{}{A}{B}{\f}=\indep{}{A}{B}{\cut*}$. In fact, Simon showed in (\cite{Simon2011OnDO}, Proposition 2.5) that, for $c$ a singleton, we have $\indep{c}{A}{B}{\f}$ if and only if $\indep{c}{A}{B}{\cut*}$ if $G$ is dp-minimal, even when $G$ is not regular. However, we will see an example where $G$ is a dp-minimal, poly-regular, non-regular dense ordered Abelian group where $\indep{c}{A}{B}{\cut*}$ (and thus $\indep{c}{A}{B}{\f}$) holds in $G$, but $\indep{c}{\div(A)}{\div(B)}{\cut}$ does not hold in $\div(G)$.
\par Let $G$ be any ordered group elementarily equivalent to the lexicographical product $\mathbb{Z}\left[\dfrac{1}{2}\right]\times\mathbb{Z}\left[\dfrac{1}{2}\right]$. Then $G$ is clearly dp-minimal, poly-regular and dense. In $\mathbb{Z}\left[\dfrac{1}{2}\right]\times\mathbb{Z}\left[\dfrac{1}{2}\right]$, we have $H_3((1, 0))=\{0\}\times\mathbb{Z}\left[\dfrac{1}{2}\right]\neq\{0\}$, so $G$ is not regular, and it has exactly one proper non-trivial definable convex subgroup, let us call it $H(G)$. Let $A$ be a definably closed parameter subset of $G$, $B$ an $|A|^+$-saturated elementary extension of $G$, and ${M}$ a $|B|^+$-saturated elementary extension of $B$. Let $c\in H({M})$ such that $c>H(B)$. Then we clearly have $\indep{c}{A}{B}{\cut*}$, hence $\indep{c}{A}{B}{\f}$ by \cite{Simon2011OnDO}. Let us show $\notIndep{c}{\div(A)}{\div(B)}{\cut}$ in $\div(M)$ (in fact, we  even show $\notIndep{c}{\div(G)}{\div(B)}{\cut}$). Let $b_1\in H(B)$ such that $b_1>H(G)$. Let $b_2\in B$ such that $b_2\not\in H(B)$, but $b_2<g$ for every $g\in G_{>0}\setminus H(G)$. Then the interval $I=[b_1, b_2]$ contains $c$.

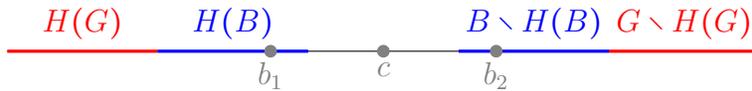
\begin{center}
\begin{tikzpicture}[scale=2]
\draw[black] (0, 0) -- (5, 0);
\draw[red, very thick] (0, 0) -- (1, 0);
\draw[blue, very thick] (1, 0) -- (2, 0);
\draw[red, very thick] (0.5, 0) node[anchor=south]{$H(G)$};
\draw[blue, very thick] (1.5, 0) node[anchor=south]{$H(B)$};
\draw[red, very thick] (4, 0) -- (5, 0);
\draw[blue, very thick] (3, 0) -- (4, 0);
\draw[red, very thick] (4.5, 0) node[anchor=south]{$G\setminus H(G)$};
\draw[blue, very thick] (3.5, 0) node[anchor=south]{$B\setminus H(B)$};
\filldraw[gray, thick] (2.5, 0) circle(1pt) node[anchor=north]{$c$};
\filldraw[gray, thick] (1.75, 0) circle(1pt) node[anchor=north]{$b_1$};
\filldraw[gray, thick] (3.25, 0) circle(1pt) node[anchor=north]{$b_2$};
\end{tikzpicture}
\end{center}

\par Let us show that $I$ is disjoint from $\div(G)$. Let $g\in \div(G)$ with $0<g<b_2$, and suppose by contradiction $g>b_1$. Let $0<N<\omega$ such that $Ng\in G$. We have $Ng\geqslant g>b_1$, so $Ng\not\in H(G)$. The ordered group $\faktor{G}{HG}$ is dense (as $\mathbb{Z}\left[\dfrac{1}{2}\right]$ is), so there exists $g'\in G$ such that we have $g'\not\in H(G)$ and $Ng'<(Ng\mod H(G))$ (choose $N$ distinct cosets between $H(G)$ and $(Ng\mod H(G))$, choose $g'$ a representative of the coset that corresponds to the minimal distance between them in $\faktor{G}{H(G)}$). As $g'\not\in H(G)$ and $g'>0$, we have $b_2<g'$. However, we have $Ng'<Ng<Nb_2$, a contradiction. This concludes the example.
\par For future work, we should try to generalize our results to dp-minimal ordered groups. In such a group, one can hope that a global type over a monster model $B$ does not fork over $A$ if and only if its realizations $C$ satisfy $\indep{C}{A}{B}{\cut*}$. The dimension one case has already been proved by Simon.

\subsection{Forking versus Presburger-quantifier-free dependence} For this last example, we  study an ordered Abelian group $G$ with definably closed parameter sets $A\subset B$, and a singleton $c\in G$ for which $\indep{c}{\div(A)}{\div(B)}{\cut}$ in $\div(G)$ (hence $\indep{c}{A}{B}{\cut*}$), and:
$$\left(\dcl(Ac)+l^NG\right)\cap\left(B+l^NG\right)=A+l^NG$$
for every prime $l$, $N>0$, but $\notIndep{c}{A}{B}{\d}$. The dividing definable set will not be built from the cuts or the cosets $\mod l$, but from the chain of definable convex subgroups of $G$. Our example needs to have an infinite spine.
\par Let $\Gamma$ be the lexicographical product $\mathbb{Q}_{\leqslant 0}\times\mathbb{Z}_{\leqslant 0}$. Let $G$ be the ordered group of all Hahn series with coefficients in $\mathbb{Z}$ and powers in $\Gamma$ (the valuation coincides with the Archimedean valuation). Then $G$ is a dense ordered Abelian group. For each $g\in G$, let $S(g)\subset\Gamma$ be the support of $g$. Each $g\in G$ can be written $\sum\limits_{\delta\in S(g)}\lambda_\delta t^\delta$, with $\lambda_\delta\in\mathbb{Z}$, and $t$ an indeterminate. If $\sigma$ is an automorphism of the ordered set $\mathbb{Q}_{\leqslant 0}$, then let $\bar{\sigma}\in \Aut(G)$ be defined as $\sum\limits_{(\mu, n)\in\Gamma}\lambda_{(\mu, n)}t^{(\mu, n)}\longmapsto \sum\limits_{(\mu, n)\in\Gamma}\lambda_{(\mu, n)}t^{(\sigma(\mu), n)}$. Let $A=\left\lbrace a\in G|S(a)\subset\left(\{0\}\times\mathbb{Z}_{\leqslant 0}\right)\right\rbrace$ and $B=\dcl\left(A\cup\{t^{(-1, 0)}\}\right)$. The action on $\mathbb{Q}_{\leqslant 0}$ of its automorphism group has only two orbits: $\{0\}$ and $\mathbb{Q}_{<0}$. As a result, we must have $\dcl(\emptyset)\subset A$. By (\cite{QEOAG}, Corollary 1.10), $\dcl(A)$ is exactly the relative divisible closure in $G$ of $A$, in other words $A$ is definably closed. With the same argument, we have $B=A+\mathbb{Z}t^{(-1, 0)}$. Let $c=t^{(-1, -1)}$, so $\dcl(Ac)=A+\mathbb{Z}c$. We have $\Delta(c)<\Delta(B_{\neq 0})$, so $c$ clearly leans left with respect to $\div(A)$ and $\div(B)$ in $\div(G)$, which implies $\indep{c}{A}{B}{\cut}$ in $\div(G)$. Moreover, for every prime $l$, and $N>0$, the elements of $l^NG$ are exactly the series $\sum\limits_\delta l^N\lambda_\delta t^\delta$ with $\lambda_\delta\in\mathbb{Z}$, so we clearly have $\left(\dcl(Ac)+l^NG\right)\cap\left(B+l^NG\right)=A+l^NG$.

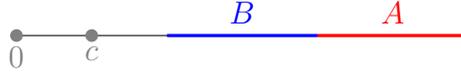
\begin{center}
\begin{tikzpicture}
\draw[black] (0, 0) -- (6, 0);
\draw[red, very thick] (4, 0) -- (6, 0);
\draw[blue, very thick] (2, 0) -- (4, 0);
\draw[red, very thick] (5, 0) node[anchor=south]{$A$};
\draw[blue, very thick] (3, 0) node[anchor=south]{$B$};
\filldraw[gray, thick] (0, 0) circle(2pt) node[anchor=north]{$0$};
\filldraw[gray, thick] (1, 0) circle(2pt) node[anchor=north]{$c$};
\end{tikzpicture}
\end{center}

\par Now let us prove $\notIndep{c}{A}{B}{\d}$. For each $n<\omega$, let $\sigma_n\in \Aut(\mathbb{Q}_{\leqslant 0})$ such that $\sigma_n(1)=-n$. For each prime $l$, for each $g=\sum\limits_\delta\lambda_\delta t^\delta\in G$, if $g\not\in l^NG$, then one can show (see for instance \cite{QEOAG}, example 4.2) that $H_l^N(g)$ is the group of elements of $G$ with an Archimedean class smaller or equal to $\Delta(t^\delta)$, where $\delta$ is the largest element of $S(g)$ for which $\lambda_\delta\not\in l^N\mathbb{Z}$. As a result, there does not exist $g\in G$ for which $H_2(c)<H_2(g)<H_2(t^{(-1, 0)})$. The following $B$-definable set which contains $c$:
$$ X=\left\lbrace g\in G|g\not\in 2G \wedge\left(\forall h\in G,\ \left(H_2(h)\leqslant H_2(g)\vee H_2(t^{(-1, 0)})\leqslant H_2(h)\right)\right)\right\rbrace
$$
divides over $A$, because the $(\bar{\sigma_n}(X))_n$ are pairwise-disjoint. This concludes the example.

\tableofcontents
\bibliographystyle{plain}
\bibliography{biblio}

\end{document}